\documentclass[12pt]{article}
\usepackage{graphicx}
\usepackage{xcolor}
\usepackage{amssymb,amsfonts,amsthm,amsmath,calligra,tikz}
\usepackage[utf8]{inputenc}
\usepackage{accents}
\usepackage{slashed}
\usepackage{yfonts}
\usepackage{mathrsfs,pifont}
\theoremstyle{definition}
\usepackage{tikz}
\usepackage[all]{xy}
\usepackage{enumitem}
\usepackage{float}

\usepackage{lscape}

\def\be{\begin{eqnarray}}
\def\ee{\end{eqnarray}}

\def\matZ{{\mathbb{Z}}}
\def\matR{{\mathbb{R}}}
\def\matQ{{\mathbb{Q}}}
\def\matC{{\mathbb{C}}}

\newcommand{\Stab}{\mathrm{Stab}}
\newcommand{\bA}{\mathsf{A}}

\newcommand{\bT}{\mathsf{T}}

\newcommand{\Lie}{\mathrm{Lie}}

\newcommand{\cL}{\mathscr{L}}

\def\sl{{ s}}

\def\cochar{{\textrm{cochar}}}

\def\DT{\mathsf{DT}}

\def\GW{\mathsf{GW}}

\theoremstyle{definition}
\newtheorem{Definition}{Definition}
\newtheorem{Proposition}{Proposition}
\newtheorem{Lemma}{Lemma}
\newtheorem{Corollary}{Corollary}
\newtheorem{Conjecture}{Conjecture}
\newtheorem{Theorem}{Theorem}

\newtheorem{Remark}{Remark}
\newtheorem{Example}{Example}

\def\tb {{\cal{V}}}  

\def\X{\mathsf{X}} 
\def\Y{\mathsf{Y}} 
\def\fp{ \mathbf{S}} 
\def\bU{{\bf U}} 
\def\bL{{\bf L}}

\def\bD{{\bf D}}
\def\sing{{\mathsf{Sing}}} 

\makeatletter
\newcommand{\somespecialrotate}[3][]{%
\begingroup
\sbox\@tempboxa{#3}%
\@tempdima=.5\wd\@tempboxa
\sbox\@tempboxa{\rotatebox[#1]{#2}{\usebox\@tempboxa}}%
\advance\@tempdima by -.5\wd\@tempboxa
\mbox{\hskip\@tempdima\usebox\@tempboxa}%
\endgroup}
\linethickness{1 pt}

\def\M{\mathbf{M}} 
\def\T{\mathbf{T}} 
\def\B{\mathbf{B}} 
\def\Walls{\mathsf{Walls}} 
\def\Mon{\mathbf{Mon}}

\def\B{\mathbf{B}} 

\newcommand{\Uq}{\mathscr{U}_\hbar}

\newlength{\dhatheight}
\newcommand{\doublehat}[1]{%
	\settoheight{\dhatheight}{\ensuremath{\widehat{#1}}}%
	\addtolength{\dhatheight}{-0.35ex}%
	\widehat{\vphantom{\rule{1pt}{\dhatheight}}%
		\smash{\widehat{#1}}}}

\def\gt{{\Uq(\doublehat{\frak{gl}}_1)}}	

\def\gtn{{\Uq(\doublehat{\frak{gl}}_N)}}

\def\gtb{{\Uq(\doublehat{\frak{gl}}_b)}}

\begin{document}
\title{Quantum differential and difference equations for $\mathrm{Hilb}^{n}(\matC^2)$.}


\author{\it \small{Andrey Smirnov, asmirnov@email.unc.edu}\\
\it \small{University of North Carolina at Chapel Hill}\\ 
\small \it{Steklov Mathematical Institute of Russian Academy of Sciences}}

\date{}

\maketitle
\thispagestyle{empty}
	
\begin{abstract}
We consider the quantum {\it difference} equation of the Hilbert scheme of points in $\matC^2$. This equation is the K-theoretic generalization of the quantum differential equation discovered by A.~Okounkov and R.~Pandharipande in \cite{OkPanDiff}. We obtain two explicit descriptions for the monodromy of these equations - representation-theoretic and algebro-geometric. In the representation-theoretic description, the monodromy acts via certain explicit elements in the quantum toroidal algebra $\gt$. In the algebro-geometric description, the monodromy features as transition matrices between the stable envelope bases in equivariant K-theory and elliptic cohomology.  Using the second approach we identify the monodromy matrices for the differential equation with the K-theoretic $R$-matrices of cyclic quiver varieties, which appear as subvarieties in the $3D$-mirror Hilbert scheme.  Most of the results in the paper are illustrated by explicit examples for cases $n=2$ and $n=3$ in the Appendix.






\end{abstract}

\setcounter{tocdepth}{2}
	
\tableofcontents

\section{Introduction}

\subsection{Quantum differential equation for $\textrm{Hilb}^{n}(\matC^2)$} 

In the study of Gromov-Witten/Donaldson-Thomas theories of threefolds and representation theory of affine Yangians a certain first order differential equation with remarkable properties emerges naturally. This equation features as the quantum differential equation (qde) in the equivariant quantum cohomology of the Hilbert scheme of points in the plane $\X=\textrm{Hilb}^{n}(\matC^2)$ and was first discovered and investigated by  A. Okounkov and R. Pandharipande in~\cite{OkPanCoh,OkPanDiff}.

The quantum differential equation describes a connection in a trivial bundle over $\mathbb{P}^1$ with fiber given by the equivariant cohomology $H^{\bullet}_{\bT}(\X)$. The singularities of this connection are located at: 
$$ 
\sing=\{0,\infty, \sqrt[k]{1}, k=2,...,n\} \setminus \{ 1 \} \, \subset \mathbb{P}^1
$$
where $\sqrt[k]{1}$ denotes the set of all complex $k$-th roots of $1$. The most important global properties of the qde are encoded in the homomorphism 
\be \label{homom}
\pi_{1}(\mathbb{P}^1 \setminus \sing,z) \to \textrm{End}(H_{\bT}^{\bullet}(\X))
\ee 
provided by the monodromy of the corresponding connection. 
One of the goals of this paper is to give a complete description of this homomorphism  for special choices of the based point $z \in \mathbb{P}^1$. 

We recall the explicit form of the qde and its main properties in Section~\ref{qdesection}.

\subsection{Quantum difference equation for $\textrm{Hilb}^{n}(\matC^2)$}

In order to understand the monodromy of qde, we first consider its $q$-difference generalization: the quantum difference equation of the Hilbert scheme $\X$ is the K-theoretic version of the quantum connection. This equation has the following form
\be \label{qdeintro}
\Psi(z q) = \M(z) \Psi(z), \ \ \ \Psi(z) \in K_{\bT}(\X), \ \ \ |q|<1
\ee
The operator $ \M(z) \in \textrm{End}(K_{\bT}(\X))$ was computed in Section 8 of \cite{OS}. It is expressed via the action of the quantum toroidal algebra $\gt$ on the equivariant K-theory of the Hilbert scheme $K_{\bT}(\X)$ as follows. The algebra $\gt$ is generated by elements $\alpha^{w}_{k}$ parameterized by  $w\in\matQ$ and  $k\in \matZ$, see Section \ref{fockmod} below. The action of $\gt$ on the equivariant $K$-theories $K_{\bT}(\X)$ was constructed in \cite{Feigin2011}.   For each rational number $w=a/b$ we associate the {\it wall-crossing operator}
$$
\B_{w}(z,q)=: \! \exp\Big( \sum\limits_{k=1}^{\infty}
\dfrac{r^{w}_{k} \alpha^{w}_{-k} \alpha^{w}_{k} }{1-z^{-b k } {q^{a k}}}  \Big)\!:
$$
which is an element in a completion of $\gt$. 
The operator $\M(z)$  acts on $K_{\bT}(\X)$ by:
\be \label{mintr}
\M(z) = \mathscr{O}(1)\!\! \prod^{\longrightarrow}\limits_{{w\in \matQ} \atop {-1 \ll w <0 }} \B_w(z,q)
\ee
where $\mathscr{O}(1)$ denotes the operator of multiplication by the corresponding line bundle in K-theory.  
In the cohomological limit, which corresponds to sending $q\to 1$, the $q$-difference equation (\ref{qdeintro}) degenerates to the quantum differential equation.

\subsection{Monodromy of the quantum difference equation} 
It is well known that the $q$-difference equations are easier to work with than their differential limits. In particular, the multiplicative nature of $q$-difference equations allows one to describe their fundamental solutions in terms of infinite products. 
In Section \ref{monsolsec} we use this idea to compute two fundamental solution matrices for (\ref{qdeintro}), the first is holomorphic at $z=0$ and the second at $z=\infty$.  The {\it monodromy} of (\ref{qdeintro}) is defined (according to Birkhoff) as the transition matrix between these two fundamental solutions.  From this we find that the monodromy has the form of an ordered infinite product:
\be \label{mon1}
\Mon(z) \sim \prod^{\longrightarrow}\limits_{{w\in \matQ}} \B_w(z,q).
\ee
We explain our notation for the product over rational numbers later in Section \ref{prodsect}.
In cohomological limit $q\to 1$ of $\Mon(z)$  we obtain description of the monodromy of the differentiation equation. 
In particular, we find that the operators 
\be \label{monodintro}
\B_{w}:=\B_{w}(\infty,1)= :\! \exp\Big( \sum\limits_{k=1}^{\infty}
r^{w}_{k} \alpha^{w}_{-k} \alpha^{w}_{k}   \Big)\!:
\ee
describe the monodromy of the qde along the loop based at $z=0$ which goes around the singularity located at the root of unity $z=e^{2\pi i w}$, see Fig.\ref{monodex} in Section \ref{mondsec}. As a result (Theorem \ref{nomodtheor}) we obtain a complete description of the homomorphism (\ref{homom}). 

\subsection{Elliptic stable envelopes and mirror symmetry}
Thanks to its geometric origin, equation (\ref{qdeintro}) plays  distinguished role in the world of $q$-difference equations. A geometric  method for computing the monodromy for equations of this type, was recently proposed by M. Aganagic and A. Okounkov in \cite{AOElliptic,OkEll1,OkEll2}. In Section \ref{ellsection} we use their results to show that in the basis of fixed points the monodromy operator admits a Gauss decomposition:
\be \label{mon2}
\Mon(z) = {\bf{U}}(z)^{-1} \, {\bf{L}}(z)
\ee
where ${\bf{U}}(z)$ and ${\bf{L}}(z)$ are certain upper and lower-triangular matrices, with matrix elements given by the fixed points components of the {\it stable envelope classes} in the equivariant elliptic cohomology of $\X$.  

Comparing (\ref{mon1}) and (\ref{mon2}) and using $3D$-mirror symmetry of elliptic stable envelopes we obtain the following result (see Theorem \ref{brmatth} for precise statement): 
\be \label{brintro}
\B_{w}(z,q) = {\bf R}^{-}_{\Y_{w}}(z q^{-w})
\ee
where ${\bf R}^{-}_{\Y_{w}}(z)$ is the K-theoretic {\it $R$-matrix} of certain subvariety $\Y_{w} \subset \X^{!}$, in the {\it 3D-mirror} (also known as the symplectic dual) Hilbert scheme $\X^{!}$. For a rational number $w=a/b$ the irreducible components of $\Y_{w}$ are isomorphic to Nakajima varieties associated with a cyclic quiver with $b$ vertices, see Fig.~\ref{cicqvpic}. Form the representation-theoretic viewpoint,  the $R$-matrix ${\bf R}^{-}_{\Y_{w}}(z q^{-w})$  correspond to the trigonometric $R$-matrix of the quantum toroidal algebra $\gtb$.

Comparing (\ref{brintro}) with (\ref{monodintro}) we conclude that the monodromy of the qde along  a loop which goes around the  singularity located at $z=e^{2\pi i w}$  is described by the K-theoretic $R$-matrix of the mirror variety $\Y_w$:
$$
\B_{w}={\bf R}^{-}_{\Y_{w}}:={\bf R}^{-}_{\Y_{w}}(\infty)
$$ 
Finally, let us note that the wall-crossing operators $\B_{w}(z,q)$ were computed in \cite{OS} using the machinery  of Hopf algebras developed in \cite{EtingofVarchenko}, which is not geometric in its nature. Now, we can view relation (\ref{brintro}) as independent, purely
geometric {\it definition} of this operators.

\subsection{Quantum difference equation as qKZ}
Using (\ref{brintro}) we may rewrite the operator (\ref{mintr}) as an ordered product of $R$-matrices. Thus, (\ref{qdeintro}) takes a form similar to the {\it quantum Knizhnik–Zamolodchikov equation}, see (\ref{qKZ}) and Section \ref{qdesec} for discussion. We note, however, an important difference: the $R$-matrices entering the standard qKZ equations are all $R$-matrices of the same quantum group. In contrast, ${\bf R}^{-}_{\Y_{w}}(z)$ in (\ref{qKZ}) are the $R$-matrices associated with cyclic quivers with length $b$ which vary with~$w$. First, this suggests to view  the equation (\ref{qKZ}) as a proper generalization of qKZ equations to the case of toroidal quantum groups. Second, this suggest considering the the quantum difference and differential equations of $\X$ as the generalized qKZ equations  associated with mirror variety $\X^{!}$. 

\subsection{Connections with other topics in representation theory}

From representation-theoretic perspective, the quantum differential equation for the Hilbert scheme $\X$ describes the so called Casimir connection for the affine Yangian $Y_{\hbar}(\widehat{\frak{gl}}_1)$ \cite{OkMaul}. 
The monodromy of the Casimir connection for the classical symplectic resolutions $X=T^{*}(G/P)$, where $P\subset G$ is a parabolic subgroup, describes the action of the quantum Weyl group  associated with $G$ \cite{valer2,TolMon}. Thus, the monodromy operators $\B_{w}$ can be viewed as generators of the ``quantum Weyl group of the quantum toroidal algebra $\gt$''. The wall-crossing operators $\B_w(z,q)$ with the parameter $z$ ``turned on'' provide the {\it dynamical version} of this Weyl group, with $z$ playing the role of the dynamical parameter. We refer to \cite{EtingofVarchenko} for the introduction  to the dynamical Weyl groups.

Important structures emerge the from the algebra
$$
\mathscr{A}_{X}=\textrm{quantization of } \, \X
$$
known as Cherednik's spherical DAHA for $\frak{gl}(n)$. Using quantization in prime characteristic $p\gg 0$ one constructs the action of the fundamental group in~(\ref{homom}) by the autoequivalences of derived category $D^{b}_{\textrm{Coh}}(\X)$ \cite{OkBez}. This leads to a categorification of the monodromy operators $\B_{w}$. The categorification of the dynamical operators $\B_{w}(z,q)$ remains to be understood in this approach. Perhaps  (\ref{brintro}), relating these operators to the K-theoretic $R$-matrices of the symplectic dual variety $\X^{!}$ gives a hint in this direction.

\section*{Acknowledgments}
This paper emerged from the author's attempt to write an ``example'' for a project currently developed in \cite{KononovSmirnov1,KononovSmirnov2} and we thank Yakov Kononov for collaboration. We thank Hunter Dinkins for reading preliminary version of the paper and useful suggestions.  We thank Andrei Okounkov from whom we learned many of the ideas discussed here.  

The work is partially supported by the Russian Science Foundation under grant
19-11-00062.

\section{Quantum differential equation of $\textrm{Hilb}^{n}(\matC^2)$ \label{qdesection}}
In this paper $\X=\textrm{Hilb}^{n}(\matC^2)$ - the Hilbert scheme of $n$ points in $\matC^2$. We denote by 
$\bT$ the two-dimensional complex torus acting on $\X$ via $$
(x,y) \to (\epsilon_1 x, \epsilon_2 y)
$$
where $(x,y)$ denote the coordinates on $\matC^2$. The $\bT$-equivariant quantum cohomology ring of  $\X$ is computed in \cite{OkPanCoh}. The corresponding quantum differential equation is considered in \cite{OkPanDiff}. This section is a brief overview of these results.


\subsection{Fock space presentation}
Let us consider $\matQ(\epsilon_1,\epsilon_2)$-vector space
\be \label{cohfock}
\textsf{Fock}=\matQ[p_1,p_2,p_3,\dots] \otimes_{\matZ} \matQ(\epsilon_1,\epsilon_2).
\ee
Let us consider the Heisenberg algebra generated by $\alpha_n$, $n\in \matZ\setminus\{0\}$
satisfying the relations
$$
[\alpha_n,\alpha_m]=n \delta_{n+m}.
$$
This algebra acts on the Fock space via 
\be \label{cohaction}
\alpha_{n}(f) = \left\{\begin{array}{ll}
	-n \dfrac{d f}{d p_n}, & n>0,\\
	-p_{-n} f, & n<0.
\end{array}\right.
\ee
The vectors 
\be \label{gwbas}
p_{\mu}=\prod_{i} p_{\mu_i}
\ee
form a basis of the Fock space labeled by partitions $\mu=(\mu_1,\mu_2,\dots)$. There exist an isomorphism of $\matQ(\epsilon_1,\epsilon_2)$ vector spaces 
\be \label{FockDef}
\bigoplus_{n=0}^{\infty} H^{\bullet}_{\bT}( \textrm{Hilb}^{n}(\matC^2) )_{loc}= \textsf{Fock}
\ee
Under this isomorphism, the canonical basis of the torus fixed points in $H^{\bullet}_{\bT}( \textrm{Hilb}^{n}(\matC^2) )_{loc}$ is identified with the basis of normalized Jack polynomials $J_{\lambda}$ in the Fock space. The summand in (\ref{FockDef}) corresponding to a fixed values of $n$ is spanned by the Jack polynomials $J_{\lambda}$ with $|\lambda|=n$. Examples of Jack polynomials can be found in Appendix \ref{appjack}.

\subsection{Quantum differential equation}
Let us consider the following operator acting on the Fock space:
\be \label{cohomm}
\begin{aligned} 
	\textbf{m}(z)=  \dfrac{1}{2} \sum\limits_{k,l>0} ( \epsilon_1\epsilon_2 \, \alpha_{k+l} \alpha_{-k} \alpha_{-l}+\alpha_{-k-l} \alpha_{k} \alpha_{l} )\\
	+(\epsilon_1+\epsilon_2) 
	\sum\limits_{k=1}^{\infty} \dfrac{k}{2} \dfrac{z^k+1}{z^k-1}\alpha_{-k} \alpha_{k}- \dfrac{\epsilon_1+\epsilon_2}{2} \dfrac{z+1}{z-1} \, |\cdot|
\end{aligned}
\ee
where $|\cdot|$ maps $p_{\mu}$ to $|\mu| p_{\mu}$. The operator $\textbf{m}(0)$ corresponds to the classical multiplication by the fist Chern class in equivariant cohomology. In particular, it is diagonal in the basis of fixed points (Jack polynomials) with the following eigenvalues: 
\be \label{eigvald}
\textbf{m}(0)( J_{\lambda} )= c^{(0)}_{\lambda} J_{\lambda}, \ \ c^{(0)}_{\lambda}=\sum_{(i,j) \in \lambda} (i-1) \epsilon_1+ (j-1)\epsilon_2.
\ee
One also observes that 
\be \label{minver}
\textbf{m}(z^{-1})=-(-1)^{l} \textbf{m}(z) (-1)^{l}
\ee
where we denote the operator
$$
(-1)^{l} : p_{\mu} \to (-1)^{l_{\mu}} p_{\mu}.
$$
This means that
\be \label{minf}
\textbf{m}(\infty) (J^{\ast}_{\lambda}) = c^{(\infty)}_{\lambda} J^{\ast}_{\lambda}, 
\ee
where $J^{\ast}_{\lambda}=\left.J_{\lambda}\right|_{p_{i}= -p_{i}}$ denote the Jack polynomials in variables $-p_i$ and
\be \label{infev}
c^{(\infty)}_{\lambda}=-c^{(0)}_{\lambda}.
\ee
The main object of study in \cite{OkPanDiff} is the differential equation (qde):
\be \label{connec}
\nabla \psi =0, \ \ \psi \in \textsf{Fock}, \ \ \nabla=z \dfrac{d}{dz} - {\bf m}(z).
\ee
For a fixed value of $n$ this equation has regular singularities
located at
\begin{equation} \label{siings}
\sing=\{0,\infty, \sqrt[k]{1}, k=2,...,n\} \setminus \{ 1 \} \, \subset \mathbb{P}^1
\end{equation}
where $\sqrt[k]{1}$ denotes the set of all complex $k$-th roots of 1. Note that $z=1$ is {\it not a singularity} of qde.

\subsection{Fundamental solution near $z=0$}
Let $\textsf{J}$ be the matrix with $\lambda$-th column given by $J_{\lambda}$. We assume that the columns of $\textsf{J}$ are {\it ordered by the standard dominance order on partitions}. 
By (\ref{eigvald}) we have
$$
\textbf{m}(0) \, \textsf{J} = \textsf{J} \,c^{(0)}
$$
where $c^{(0)}$ is the diagonal matrix with eigenvalues $c^{(0)}_{\lambda}$.  Explicit examples of matrices $\textsf{J}$ can be found in Appendix \ref{appjack}. 

\begin{Remark} \label{basrem}
Unless otherwise stated, all operators acting in the Fock space are considered in the basis (\ref{gwbas}). 
Thus, we sometimes refer to the operators as ``matrices''  without confusion. In this view, the matrix $\textsf{J}$ 
is the transition matrix from the basis ``Gromov-Witten'' basis $p_{\lambda}$ to the ``Donaldson-Thomas'' basis $J_{\mu}$ of the Fock space, see Appendix \ref{appjack}. 
\end{Remark}

By basic theory of ordinary differential equations, the solutions of qde (\ref{connec}) in $D_{0}=\{|z|<1\}$ are described by the fundamental solution matrix of the form:
\be \label{solzer}
\psi^{0}(z) =\psi^{reg}_{0}(z)  z^{c^{(0)}}, 
\ee
where $z^{c^{(0)}}$ denotes the diagonal matrix with eigenvalues $z^{c^{(0)}_{\lambda}}$, the matrix $\psi^{reg}_{0}(z)$ is holomorphic in $D_{0}$ and  is normalized by the condition 
\be \label{solzernorm}
\psi^{reg}_{0}(0)=\textsf{J}\,   \ \ \ 
\ee
The coefficients in the Taylor expansion of $\psi^{reg}_{0}(z)$  at $z=0$ are uniquely determined by equation (\ref{connec}) and ``initial condition'' (\ref{solzernorm}).

\begin{Remark} \label{monrem}
By definition, the {\it columns} of the matrix $\psi^{0}(z)$ form a basis of solutions of qde. Any other solution is a $\matQ(\epsilon_1,\epsilon_2)$-linear combination of these solutions. This means that all other fundamental solution matrices of the qde are of the form
$$
\psi^{0}(z)\, A
$$
for some invertible $A\in \textrm{End}(\textsf{Fock})$.
\end{Remark}

\subsection{Quantum connection and monodromy}
$\nabla$ is a flat connection in a trivial bundle over $\mathbb{P}^{1}\setminus \sing$ with a fiber $\textsf{Fock}$. 
The monodromy of this connection defines a representation of the fundamental group $\pi_1(\mathbb{P}^1 \setminus \sing,p)$ based  at a point $p \in \mathbb{P}^{1}\setminus \sing$. 

In Section \ref{mondsec} we compute the monodromy for $p=0_{+}$ - a point ``infinitesimally'' close to $0$\footnote{$0$ is a singularity of qde and we can not take $p=0$ as a base point.}. In this case it is convenient to normalize the fundamental solution matrix by
\be \label{dtnorms}
\psi^{0}_{\mathsf{DT}}(z):=\psi_{0}(z)\, \Gamma_{\DT}
\ee
where $\Gamma_{\DT}$ denotes the diagonal matrix with eigenvalues
\be \label{gamDT}
(\Gamma_{\textsf{DT}})_{\lambda,\lambda}=\prod\limits_{\textsf{w} \in \textrm{char}_{\bT}(T_{\lambda} \X)  } \dfrac{1}{\Gamma(\textsf{w}+1)}.
\ee
and $\Gamma(x)$ stands for the standard Gamma function of $x$.
In other words, the $\lambda$-th column of the fundamental solution matrix (\ref{solzer}) is multiplied by the product of gamma functions (\ref{gamDT}).

Let $\gamma \in \pi_1(\mathbb{P}^1 \setminus \sing,0_{+})$ and let us denote by $\psi^{0}_{\mathsf{DT}}(\gamma \cdot z)$ the solution of the qde which is obtained from $\psi^{0}_{\mathsf{DT}}(z)$ via an analytic continuation along a loop representing~$\gamma$. The columns of the matrix $\psi^{0}_{\mathsf{DT}}(\gamma \cdot z)$ are solutions of the qde and thus are linear combinations of the columns of the original fundamental solution matrix, see Remark \ref{monrem}. Thus
$$
\psi^{0}_{\mathsf{DT}}(\gamma \cdot z) = \psi^{0}_{\mathsf{DT}}(z) \cdot \nabla_{\DT}(\gamma)
$$
for some matrix $\nabla_{\DT}(\gamma) \in \textrm{End}(\textsf{Fock})$. The map
$$
\nabla_{\DT}: \, \gamma \mapsto \nabla_{\DT}(\gamma)
$$
defines an antihomomorphism 
\be \label{defmonod}
\nabla_{\DT}: \, \pi_1(\mathbb{P}^1 \setminus \sing,0_{+}) \longrightarrow \mathrm{End}(\textsf{Fock}), \ \ 
\ee
which is called the monodromy of qde.
\begin{Remark}
This map is an antihomomorphism because fundamental solutions transform in (\ref{defmonod}) by multiplication from the right, in particular: 
$$
\psi^{0}_{\DT}(\gamma_{w_2} \cdot \gamma_{w_1} \cdot z)=
\psi^{0}_{\DT}(\gamma_{w_1} \cdot z) \nabla_{\DT}(\gamma_{w_2}) =\psi^{0}_{\DT}( z)  \nabla_{\DT}(\gamma_{w_1}) \nabla_{\DT}(\gamma_{w_2})
$$
so that 
$$
\nabla_{\mathsf{DT}}(\gamma_{w_1} \cdot \gamma_{w_2})=\nabla_{\mathsf{DT}}(\gamma_{w_2}) \cdot \nabla_{\mathsf{DT}}(\gamma_{w_1}).
$$
\end{Remark}
Thanks to the choice of normalization (\ref{gamDT}) the monodromy operators have good properties:
\begin{Theorem}
{\it  The matrix elements of matrices $\nabla_{\DT}(\gamma)$, $\gamma \in \pi_1(\mathbb{P}^1 \setminus \sing,0_{+})$ are rational functions in
\be \label{tedef}
t_1 = e^{2 \pi i \epsilon_1}, \ \ \ t_2 = e^{2 \pi i \epsilon_2}.
\ee
}
\end{Theorem}
\begin{proof}
Apply logic of Theorem 3  in \cite{OkPanDiff}.
\end{proof}

In Section \ref{mondsec} we compute the monodromy matrices $\nabla_{\DT}(\gamma)$ for a certain choice of generators of $\pi_1(\mathbb{P}^1 \setminus \sing,0_{+})$.

\subsection{Monodromy based at $z=1$}
In \cite{OkPanDiff}, instead of the fundamental group based at $z=0_{+}$ the nonsingular based point $z=1$ is considered. In this case it is convenient to work with the fundamental solution matrix 
$\psi_{\GW}(z)$ which is holomorphic near $z=1$ and is normalized by 
\be \label{gwnorm}
\psi_{\GW}(1)=\Gamma_{\GW}. 
\ee
where $\Gamma_{\GW}$ denotes the diagonal matrix with eigenvalues
$$
(\Gamma_{\GW})_{\lambda,\lambda}=\prod\limits_{i} g(\mu_i,t_1) g(\mu_i,t_2), \ \ \ g(x,t)=\dfrac{x^{x t}}{ \Gamma(t x)}.
$$
For $\gamma \in \pi_1(\mathbb{P}^1 \setminus \textrm{sing},1)$ this solution transforms as
$$
\psi_{\GW}(\gamma \cdot z)= \psi_{\GW}(z) \, \nabla_{\GW}(\gamma)
$$
which provides an antihomomorphism 
\be \label{holons}
\nabla_{\GW}:\, \pi_1(\mathbb{P}^1 \setminus \sing,1) \longrightarrow \textrm{End}(\textsf{Fock})
\ee
With normalization (\ref{gwnorm}) the monodromy matrices $\nabla_{\GW}(\gamma)$ have good properties:
\begin{Theorem}[Theorem 3  in \cite{OkPanDiff}] \label{thmbased1}
{\it Let $\gamma \in \pi_1(\mathbb{P}^1 \setminus \sing,1)$, then the matrix elements of
matrices $\nabla_{\mathsf{GW}}(\gamma)$ are Laurent polynomials  in $t_1$, $t_2$ given by (\ref{tedef}).}
\end{Theorem}
\noindent
The monodromies of $\nabla_{\textsf{DT}}$ and $\nabla_{\textsf{GW}}$ are related by the transport of qde from $0$ to $1$, which is described by the following important result: 
\begin{Theorem}[Theorem 4 in \cite{OkPanDiff}] \label{thmtransp}
{\it At $z=1$ the solution (\ref{dtnorms}) has the following form:
$$
\psi^{0}_{\DT}(1) = \Gamma_{\GW} \, \mathsf{P} \, 
$$
where $\mathsf{P}$ is the matrix with $\lambda$-th column given by Macdonald polynomial $P_{\lambda}$ in Haiman's normalization.}
\end{Theorem}
Examples of Macdonald polynomials and matrices $P_{\lambda}$ can be found in Appendix \ref{macappen}.

Let $\gamma \in \pi_1(\mathbb{P}^1 \setminus \textrm{sing},0_{+})$ and let $\gamma'=t \gamma t^{-1} \in \pi_1(\mathbb{P}^1 \setminus \textrm{sing},1)$ where  $t$ is the real path from $0_{+}$ to $1$, then we have:
\begin{Proposition} \label{gwdtmonod}
{\it The matrices $\nabla_{\DT}(\gamma)$ and $\nabla_{\GW}(\gamma')$ are related by:
\be \label{monrelat}
\nabla_{\GW}(\gamma') =\mathsf{P}\, \nabla_{\DT}(\gamma) \, \mathsf{P}^{-1}.
\ee}
\end{Proposition}
\begin{proof}
By Theorem \ref{thmtransp} and (\ref{gwnorm}) we have
$$
\psi_{\DT}(1)=\psi_{\GW}(1)  \mathsf{P}
$$
thus the corresponding monodromies are conjugated by $\mathsf{P}$. 
\end{proof}

\subsection{Fundamental solution near $z=\infty$}
Let $\textsf{J}^{\ast}$ be the matrix with $\lambda$-th column given by the vector $J^{*}_{\lambda}$. By (\ref{minf}) we have
$$
\textbf{m}(\infty) \, \textsf{J}^{\ast}=\textsf{J}^{\ast} \, c^{
(\infty)}
$$
where $c^{(\infty)}$ is the diagonal matrix with eigenvalues $c^{(\infty)}_{\lambda}$. In $D_{\infty} = |z|>1$  we have unique fundamental solution of the form
\be \label{infsol}
\psi_{\infty}(z) =\psi^{reg}_{\infty}(z)  z^{c^{(\infty)}}, 
\ee
with $\psi^{reg}_{\infty}(z)$ holomorphic in $D_{\infty}$ normalized by
\be
\label{initcondinf}
\psi^{reg}_{\infty}(\infty)=\textsf{J}^{\ast}.
\ee
The coefficients of the Taylor series expansion of $\psi^{reg}_{\infty}(z)$ at $z=\infty$ are uniquely determined by the qde (\ref{connec}) and initial condition (\ref{initcondinf}). As above we also define
\be 
\psi^{\infty}_{\mathsf{DT}}(z):=\psi_{\infty}(z)\, \Gamma_{\DT}.
\ee
{\bf Important}: this time we assume that the columns of matrix 
$\psi^{\infty}_{\mathsf{DT}}(z)$ are ordered by {\it opposite dominance order on partitions}. 
Explicitly, this means the following: the dominance order and the opposite dominance order on partitions are related by the transposition 
$\lambda \to \lambda'$. Let us consider the corresponding matrix $$
\fp_{\lambda,\mu}:=\delta_{\lambda,\mu^{'}}=\left(\begin{array}{lll}
0&\cdots& 1\\
\cdots&\cdots & \cdots\\
1& \cdots & 0
\end{array}\right).
$$
Then, we assume that the initial conditions are given by
$$
\psi^{reg}_{\infty}(\infty)=\textsf{J}^{\ast} = (-1)^{l} \textsf{J}\, \fp
$$
and, respectively 
$$
\left.\psi^{\infty}_{\mathsf{DT}}(z) z^{-c^{(\infty)}}\right|_{z=\infty}=(-1)^{l} \textsf{J}\, \Gamma_{\DT}\,\fp.
$$
As we will see below, with this choice of normalization, the transports of the qde from $z=0$ to $z=\infty$ have a natural Gauss decomposition.

\subsection{Transport of qde}
Assume that we have  analytic continuations of the fundamental solutions $\psi^{0}_{\DT}(z)$ and $\psi^{\infty}_{\DT}(z)$ to some larger domains. In this case we can compare them at the points of
$$
\{z \in \matC: |z|=1, z\not\in \sing\}. 
$$
Explicitly, this is the set of points  $z=e^{2 \pi i s}$ for some $s\in \matR$.  Let us consider the transition matrix between two fundamental solutions:
\be \label{transpdef0}
\textrm{Tran}_{\DT}(s) =  \psi^0_{\DT}(e^{2 \pi i s})^{-1} \psi^{\infty}_{\DT}(e^{2 \pi i s}).
\ee
This operator describes the transport of the qde from $z=0$ to $z=\infty$ along the line  which intersects $|z|=1$ at $e^{2 \pi i s}$. The transport depends only on homotopy equivalence class of the path. This means that $\textrm{Tran}_{\DT}(s)$ is a piecewise constant function of $s$, which changes value only when $e^{2 \pi i s}$ hits a singularity. 
\begin{Theorem}[Section 4.6 \cite{OkPanDiff}] \label{tran0}
{\it The transport of qde along the positive part of the real axis $\matR_{+}$ equals:
$$
\mathrm{Tran}_{\DT}(0)= \mathsf{P}^{-1}\, \mathsf{P}^{*}.
$$
where 
$$
\mathsf{P}^{*}=(-1)^{l}\, \mathsf{P} \, \fp.
$$
}
\end{Theorem}
\begin{proof}
Let $\psi^{\infty}_{\DT}(z)$ be the solution (\ref{infsol}). 
Theorem \ref{thmtransp} together with (\ref{minver}) gives
$$
\psi^{\infty}_{\DT}(1)=\Gamma_{\mathsf{GW}} (-1)^{l}   \, \mathsf{P}\, \fp
$$
Thus, 
$$
\textrm{Tran}_{\DT}(0)= \psi^{0}_{\DT}(1)^{-1} \psi^{\infty}_{\DT}(1)= \mathsf{P}^{-1}\, (-1)^{l}\, \mathsf{P}\, \fp.
$$
\end{proof}

\noindent
In Section \ref{mondsec} we generalize this result and describe 
$\textrm{Tran}_{\DT}(s)$ for an arbitrary $s\in \matR$ such that $e^{-2 \pi i s}\not\in \sing$.

\section{K-theoretic $q$-difference equation \label{kthsec}}

The quantum difference equation (QDE) is the K-theoretic generalization of the quantum differential equation (qde) in cohomology. For the Nakajima quiver varieties these equations were computed in \cite{OS}. In particular, the case 
of the Hilbert scheme $\textrm{Hilb}^{n}(\matC^2)$ is considered in detail in Section 8 of \cite{OS}. In this section we briefly review this construction.   

\subsection{Fock representation of quantum toroidal $\frak{gl}_1$ \label{fockmod}}
In this section we denote
\be \label{fockkth}
\textsf{Fock}=\matQ[p_1,p_2,p_3,\dots] \otimes_{\matZ} \matQ(t_1,t_2).
\ee
with $t_1, t_2$ given by (\ref{tedef}). There is an isomorphism of $\matQ(t_1,t_2)$-vector spaces 
\be \label{Focksum}
\bigoplus_{n=0}^{\infty}\, K_{\bT}(\textrm{Hilb}^{n}(\matC^2))_{loc}= \textsf{Fock}
\ee
Under this isomorphism the K-theory classes  $\mathscr{O}_{\lambda} \in K_{\bT}(\X)$ of fixed points $\lambda \in \X^{\bT}$ are mapped to {\it Macdonald polynomials} $
P_{\lambda} \in \textsf{Fock}$ in Haiman's normalization~\cite{HaimanHilb}.
These polynomials form a basis of the Fock space, see Appendix \ref{macappen} for examples of $P_{\lambda}$.


Similarly to the qde in cohomology, the K-theoretic QDE for $\textrm{Hilb}^{n}(\matC^2)$ is described via an action of certain algebra on $\textsf{Fock}$. The K-theoretic structure is, however, much richer. The Fock space (\ref{fockkth}) is a natural representation of a quantum group $\gt$, called {\it quantum toroidal} ${\frak{gl}}_1$. The representation theory of this algebra and its role in mathematical physics is an exceptionally rich subject: we refer to \cite{Miki,DingIohara,NegShuf,ShifVass,ZenMor,Feigin2012,Feigin2011} for a very incomplete list of research by different groups.


The algebra $\gt$, has an explicit presentation in terms of generators and relations \cite{ShifVass}:
$$
\gt = \Big\langle e_{(n,m)}:  (n,m)\in \matZ^{2}\setminus (0,0) \Big\rangle/ \textrm{relations}
$$
Given a rational number 
$w=\frac{a}{b}$ with $\textrm{gcd}(a,b)=1$ one can consider the generators with ``slope'' $w$, see Fig. \ref{strogg}:
$$
\alpha^{w}_{k}=e_{b k,a k}, \ \ k\in \matZ\setminus \{0\}. 
$$
For each $w\in \matQ \cup \{\infty\}$ the elements $\alpha^{w}_{k}$ generate the slope $w$ Heisenberg subalgebra $H_{w} \subset \gt $ 
subject to the following relations:
$$
[\alpha^{w}_{i},\alpha^{w}_{j}] = \delta_{i+j} r^{(w)}_{i}, \ \ \ r^{(w)}_{i}= \dfrac{i\, (\hbar^{\frac{k b}{2}}- \hbar^{-\frac{k b}{2}}) }{(t_1^{\frac{k}{2}}-t_1^{-\frac{k}{2}})(t_2^{\frac{k}{2}}-t_2^{-\frac{k}{2}})( \hbar^{\frac{k}{2}}-\hbar^{-\frac{k}{2}}) }
$$
In particular, the slope $w=0$ Heisenberg subalgebra $\alpha^{0}_{k}=e(k,0)$ acts on the Fock space by
$$
\alpha^{0}_{m} (f)=  \left\{\begin{array}{rr}
\dfrac{-p_{-m} \cdot f}{(t_1^{m/2}-t_1^{-m/2})(t_2^{m/2}-t_2^{-m/2})}  &m<0\\
-m \dfrac{d f}{d p_m} & m>0
\end{array}\right. 
$$
which is a natural deformation of (\ref{cohaction}). Note that  the slope $w=\infty$-subalgebra is commutative (in this case $b=0$). The elements $\alpha^{\infty}_{k}=e(0,k)$ act diagonally in the basis of fixed points $P_{\lambda}$ and can be identified with operators of multiplication by tautological bundles in the equivariant K-theory $K_{\bT}(\textrm{Hilb}^{n}(\matC^2))$, see Section 8.1.5 of \cite{OS} for more details.  The  operators $\alpha^{\infty}_{k}$  can also be identified with so called Macdonald operators for $\frak{gl}_{\infty}$.


For a general slope $w=\frac{a}{b}$ the operators $\alpha^{w}_{k}$ act in (\ref{Focksum})  changing the weight by $k b$ units
\be \label{creann}
\alpha^{w}_{k} : K_{\bT}(\textrm{Hilb}^{n}(\matC^2)) \longrightarrow K_{\bT}(\textrm{Hilb}^{n-k b}(\matC^2)). 
\ee
In particular, $\alpha^{w}_{k} (f)=0$ for any $f\in K_{\bT}(\textrm{Hilb}^{n}(\matC^2))$ with $n< bk$.
For this reason $\alpha^{w}_{-k}$ with $k>0$ are usually referred to as {\it creation operators} and $\alpha^{w}_{k}$ as {\it annihilation operators}. The action of the operators $\alpha^{w}_{k}$ on the Fock space for general $w$ and $k$ is quite complicated. Still, it can be described explicitly in the fixed point basis~\cite{MR3322196}.

\begin{figure}[H]
	\centering
	\includegraphics[width=10cm]{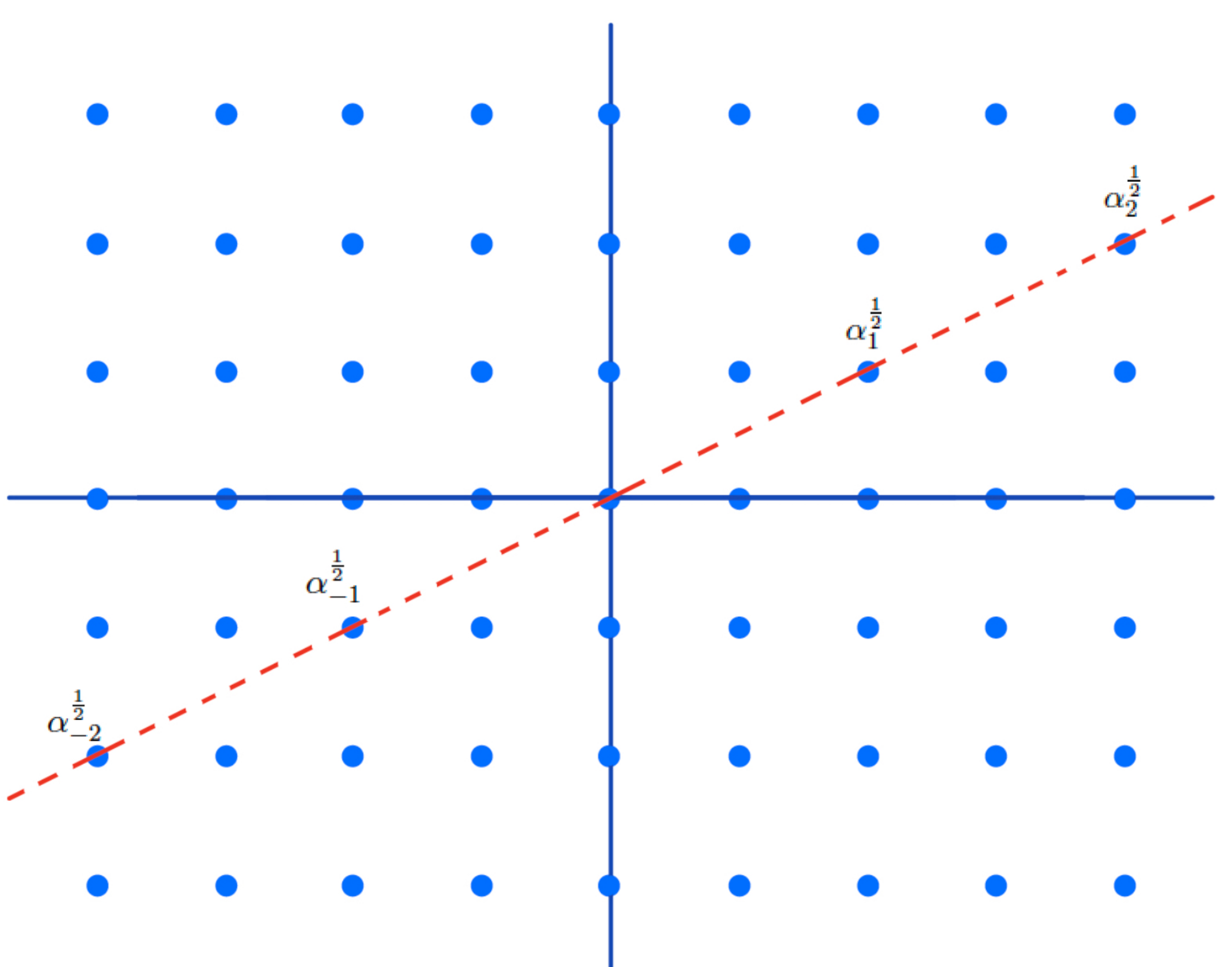}
	\caption{ \label{strogg} The structure of the toroidal algebra  and the Heisenberg subalgebra of slope $w=1/2$.}
\end{figure}

\subsection{Wall-crossing operators}
For any $w\in \matQ \cup \{\infty\}$ we define the {\it wall-crossing operator} $\B_{w}(z,q)$ acting on the Fock space by
\be \label{Bdef}
\B_{w}(z,q)=: \! \exp\Big( \sum\limits_{k=1}^{\infty}
\dfrac{r^{w}_{k}}{1-z^{-b k } {q^{a k}}} \alpha^{w}_{-k} \alpha^{w}_{k} \Big)\!:
\ee
where $a$ and $b$ denote the numerator and the denominator of $w$ and $q$ is a formal complex parameter. The symbol $::$ denotes the normally ordered operator. This means that in the Taylor expansion of  the exponent (\ref{Bdef})  the {\it annihilation operators}  $\alpha^{w}_{k}$ with $k>0$ act before the {\it creation operators} $\alpha^{w}_{-k}$.  By (\ref{creann}) only finitely many terms in this expansion act non-trivially. Thus the action of $\B_{w}(z,q)$ for $w\in \matQ$ is well defined. 

It is clear from (\ref{Bdef}) that the wall-crossing operators preserve the summands in (\ref{Focksum}). By the same reason as above the operators $\B_{w}(z,q)$ act non-trivially (i.e., $\B_{w}(z,q)\neq 1$ ) on $K_{\bT}(\textrm{Hilb}^{n}(\matC^2))$ only for $w$ corresponding to the Farey sequence of order $b$:
$$
\Walls_{n}=\Big\{w= \frac{a}{b} \in \matQ: 1 \leq |b|\leq n \Big\} \subset \matQ
$$
which we will call {\it walls}. From the explicit formulas for the action of the $\infty$-slope Heisenberg algebra one finds that the action of $\B_{\infty}(z,q)$ is also well defined, but we will not need it in this paper.  

Finally, we note that the matrix elements of $\B_{w}(z,q)$ are rational functions in $z$ and $q$. Using the action of $\gt$ on the Fock space the matrices $\B_{w}(z,q)$ can be computed very explicitly, see examples in Sections 8.3.7-8.3.8 of \cite{OS}. 
\subsection{Monodromy operators}
Note that $\B_{w}(0,q)=1$. We denote 
\be \label{monodromyB}
\B_w=\B_{w}(\infty,q) = : \! \exp\Big( \sum\limits_{k=1}^{\infty}
r^{w}_{k} \alpha^{w}_{-k} \alpha^{w}_{k} \Big)\!:
\ee
We call $\B_w$ the {\it monodromy operators}. The matrix elements of $\B_w$ do not depend on $z$ or $q$. 
From the explicit formulas above one computes
$\B_0=(t_1 t_2)^n$. It follows from Proposition \ref{proshiftslope} below that 
$$
\B_{w}=(t_1 t_2)^n, \ \ w\in \matZ.
$$
In contrast, the monodromy operators $\B_{w}$ for non-integral values of $w$ are quite nontrivial, see the examples in Appendix \ref{monapp}.

We denote $\B_{w}(z):=\B_{w}(z,1)$. 
The following Lemma is immediate from definition (\ref{Bdef}).
\begin{Lemma}
	{\it For $s\in \matQ$ we have
		\be \label{blim}
		\lim\limits_{q\to 0}\, \B_{w}(z q^{s},q) = 
		\left\{\begin{array}{ll}
			\B_{w} & w>s  \\
			1 & w<s  \\
			\B_{w}(z) & s=w
		\end{array}\right.
		\ee}
\end{Lemma}
\noindent
In Section \ref{loopmondrs} we prove that the operator $\B_w$ describes the monodromy of $\nabla_{\textsf{DT}}$ along the loop around the singularity of $\nabla_{\textsf{DT}}$ at $z=e^{2\pi i w}$, which explains our choice of its name. 

\subsection{Basic properties of wall-crossing operators}
Let us denote by  $\cL_{0}$  the operator of tensor multiplication by the line bundle $\mathscr{O}(1)$ in  equivariant K-theory. In the basis of fixed points it is characterized by the following eigenvalues 
\be \label{eigval}
\cL_{0} \,( P_{\lambda})= (\prod\limits_{(i,j)\in \lambda} \, t_1^{i-1} t_2^{j-1})\, P_{\lambda}.
\ee

\begin{Proposition} \label{proshiftslope}
	{\it $\cL_{0}$ intertwines the action of the Heisenberg algebras $H_{w}$ and $H_{w-1}$ on the Fock space:
		$$
		\cL_{0}^{-1} \alpha^{w}_k = \alpha^{w-1}_k \cL_{0}^{-1}.
		$$
		In particular, it intertwines the wall crossing operators
		\be \label{btrans}
		\cL^{-1}_{0} \B_{w}(z q,q) \cL_{0}  = \B_{w-1}(z,q)
		\ee
		and the monodromy operators
		$$
		\cL^{-1}_{0} \B_{w}(z) \cL_{0}  = \B_{w-1}(z), \ \ \ \cL^{-1}_{0} \B_{w} \cL_{0}  = \B_{w-1}.
		$$}
\end{Proposition}
\begin{proof}
	See \cite{OS}, Section 8.
\end{proof}

\begin{Proposition}
	{\it For arbitrary complex parameters $a$ and $b$ the wall-crossing operators commute
		$$
		\B_w(a,q) \B_w(b,q)= \B_w(b,q)\B_w(a,q), 
		$$
		i.e., the coefficients of the Taylor expansion $\B_{w}(z)=\sum\limits_{k=0}^{\infty} \B_{w,k} z^k $ commute
		$$
		\B_{w,i}\B_{w,j}=\B_{w,j}\B_{w,i} 
		$$}
\end{Proposition}
\begin{proof}
	Follows directly from (\ref{Bdef}) and relations in $H_w$.
\end{proof}
The following result  describes the transformation of  the wall-crossing  operators  and the quantum difference equation  under $z\to z^{-1}$.

\begin{Proposition} \label{invz}
	{\it $$
		\B_{w}(z,q) \B_w^{-1}= \B_w^{-1} \B_{w}(z,q) = \B_{w}(z^{-1},q^{-1})^{-1}  
		$$
		in particular, at $q=1$
		$$
		\B_{w}(z) \B_w^{-1}= \B_w^{-1} \B_{w}(z) = \B_{w}(z^{-1})^{-1}. 
		$$}
\end{Proposition}
\begin{proof}
	Follows directly from (\ref{Bdef}) and relations in $H_w$.
\end{proof}

\subsection{Quantum difference equation for $\textrm{Hilb}^{n}(\matC^2)$ \label{prodsect}}
The {\it quantum difference equation}  for $\textrm{Hilb}^{n}(\matC^2)$ (QDE) has the following form
\be \label{qdiffe}
\Psi(z q) = \M(z) \Psi(z), \ \ \Psi(z) \in \textsf{Fock}
\ee
where $\M(z) \in \textrm{End}(\textsf{Fock})$ is given by
\be \label{Mdef}
\M(z)=\cL_{0} \prod\limits_{w\in [-1,0)}^{\longrightarrow}\, \B_{w}(z,q).
\ee
In the ``classical limit'' $z=0$ the operator $\M(z)$
coincides with~$\cL_{0}$:
\be \label{m0}
\M(0)=\cL_{0}.
\ee
In (\ref{Mdef}) and throughout this paper we use the following conventions:
$$
\prod\limits_{w \in I}^{\longrightarrow} \, \B_{w}(z,q)  
$$
denotes the product over the rational numbers in an interval $I \subset \matR$ ordered so that $w$ increases from the left to the right. Similarly, 
$$
\prod\limits_{w \in I}^{\longleftarrow} \, \B_{w}(z,q)
$$
denotes the ordered product of operators with $w$ increasing from right to left. If $I\cap \Walls_{n}$ is finite, as for instance in (\ref{Mdef}), then all but finitely many terms in these products act on the Fock space as identity operators. Thus, for fixed $n$ and bounded $I$ these products are finite and well defined.  

\begin{Example}
	Let us assume that $n=3$ then, for instance
	$$
	\prod\limits_{w \in [0,1)}^{\longrightarrow} \, \B_{w}(z)  =\B_{0}(z) \B_{\frac{1}{3}}(z)\B_{\frac{1}{2}}(z) \B_{\frac{2}{3}}(z) 
	$$
	or
	$$
	\prod\limits_{w \in (-1,1)}^{\longleftarrow} \, \B_w =\B_{\frac{2}{3}} \B_{\frac{1}{2}} \B_{\frac{1}{3}} \B_{0} \B_{-\frac{1}{3}}\B_{-\frac{1}{2}} \B_{-\frac{2}{3}}.
	$$
	
\end{Example}

\section{Solutions of QDE and monodromy \label{monsolsec}}
The multiplicative nature of $q$-difference equations allows us to construct fundamental solutions and monodromy matrices as infinite products of the wall crossing operators. For instance, in the case of zero-dimensional $A_{\infty}$  quiver varieties we are dealing with the first order {\it scalar} $q$-difference equations and the analysis of monodromy is elementary, see Section 5 of \cite{dinksmir3}. In this section apply the same logic for the Hilbert scheme $\mathrm{Hilb}^{n}(\matC^2)$. 

\subsection{Fundamental solution of QDE near $z=0$ \label{zersol}} 
By (\ref{m0}) the operator $\M(0)$ is diagonal in the basis of fixed points, which in the Fock space corresponds to the basis of Macdonald polynomials $P_{\lambda}$. We denote by $\textsf{P}$ the matrix with columns given by eigenvectors $P_{\lambda}$. We denote by ${\mathbf{E}_{0}}$ the diagonal matrix of eigenvalues, so that 
$$
\M(0)\, \textsf{P}=\textsf{P}\, {\mathbf{E}_{0}}.
$$
The eigenvalues of $ {\mathbf{E}_{0}}$ are monomials in $t_1$ and $t_2$ given by (\ref{eigval}).

From the basic theory of $q$-difference equations there exists a unique fundamental solution of the QDE of the form
\be \label{zerte}
\Psi_{0}(z)=\mathsf{P}\, \Psi^{reg}_{0}(z)\, e^{\frac{\ln({\mathbf{E}}_0) \ln(z)}{\ln(q)}}, \ \ \ \Psi^{reg}_{0}(z)=1+\sum\limits_{k=1}^{\infty}\,\Psi^{reg}_{0,k}\, z^k
\ee
The matrix  $\Psi^{reg}_{0}(z)$ solves
\be \label{equatzero}
\Psi^{reg}_{0}(z q) {\mathbf{E}}_0=\M^{\bullet}(z) \Psi^{reg}_{0}(z).
\ee
where $\M^{\bullet}(z)=\mathsf{P}^{-1} \M(z) \mathsf{P}$ is the matrix for the operator $\M(z)$ in the basis of the fixed points $\mathsf{P}$. Here, again, we assume implicitly that $\M(z)$ is the ``matrix'' of the corresponding operator in the basis $p_{\mu}$ and $\mathsf{P}$ is the transition matrix from $p_{\mu}$ to the basis of Macdonald polynomials $P_{\mu}$, see Remark \ref{basrem} above and Appendix \ref{macappen} for examples of $\mathsf{P}$. 

Let us define
$$
\M^{\bullet}_{0}(z):={\mathbf{E}}^{-1}_0 \M^{\bullet}(z), \ \ \ \M^{\bullet}_{k}(z):= {\mathbf{E}}^{-k}_0 \M^{\bullet}_{0}(z q^{k}) {\mathbf{E}}^{k}_0.
$$
Then the infinite product 
\be \label{p0}
\Psi^{reg}_{0}(z)=\M^{\bullet}_{0}(z)^{-1} \M^{\bullet}_{1}(z)^{-1} \M^{\bullet}_{2}(z)^{-1}\dots
\ee
solves (\ref{equatzero}). We denote by
$$
\B^{\bullet}_{w}(z):=\mathsf{P}^{-1}\B_{w}(z) \mathsf{P}, \ \ \ \B^{\bullet}_{w}:=\mathsf{P}^{-1}\B_{w} \mathsf{P}
$$
the matrices of operators $\B_{w}(z)$ and $\B_{w}$ in the basis of fixed points $\mathsf{P}$. 
\begin{Proposition} \label{asymzero}
	{\it For $s\in \matQ$ the fundamental solution (\ref{p0}) has the following limits:
		$$
		\lim\limits_{q\to 0} \Psi^{reg}_{0}(z q^{s}) = \left\{\begin{array}{ll} 1, & s\geq 0\\ 
		\Big(\prod\limits^{\longleftarrow}_{w\in (s,0)}\, (\B^{\bullet}_{w})^{-1} \Big)\cdot \B^{\bullet}_{s}(z)^{-1} & s<0
		\end{array}\right.
		$$}
\end{Proposition}
\begin{proof}
	In the basis of fixed points $\mathsf{P}$ the matrix of the operator $\cL_{0}$ is given by ${\mathbf{E}}_0$, thus
	$$
	\M^{\bullet}_{0}(z)=\prod\limits_{w\in [-1,0)}^{\longrightarrow}\,\B^{\bullet}_{w}(z)
	$$
	and from (\ref{btrans}) we find
	$$
	\M^{\bullet}_{k}(z) =\prod\limits_{w\in [-1-k,-k)}^{\longrightarrow}\,\B^{\bullet}_{w}(z)
	$$
	which gives
	$$
	\Psi^{reg}_{0}(z)= \prod\limits^{\longleftarrow}_{w\in (-\infty,0)}\, \B^{\bullet}_{w}(z)^{-1}.
	$$
	The proposition follows from (\ref{blim}).
\end{proof}

\subsection{Fundamental solution of QDE near $z=\infty$} 
By the general theory of quantum difference equations for Nakajima quiver varieties \cite{OS} the operator $\M(z)$ near $z=\infty$ defines the QDE for the Nakajima variety with opposite stability condition. In particular, the matrix $\M(\infty)$ is diagonalizable over $\matQ(t_1,t_2)$ with eigenvalues given by monomials in $t_1,t_2$. 

Let $\textsf{H}^{\ast}$ be the matrix with columns given by the eigenvectors of $\M(\infty)$. Let ${\mathbf{E}_{\infty}}$ be the diagonal matrix of eigenvalues so that:
$$
\M(\infty)\, \textsf{H}^{\ast}=\textsf{H}^{\ast}\, {\mathbf{E}_{\infty}}
$$
In Section \ref{tmatsec} we will show that
$
\textsf{H}^{\ast}=\textsf{P}^{\ast}$ with $\textsf{P}^{\ast}$ as in Theorem \ref{tran0}, but at the moment we only need the fact that the eigenbasis  $\textsf{H}^{\ast}$ exists. 
The QDE has unique fundamental solution of the form
\be \label{psiinf}
\Psi_{\infty}(z)= \textsf{H}^{\ast}\, \Psi^{reg}_{\infty}(z)\, e^{\frac{\ln({\mathbf{E}}_{\infty}) \ln(z)}{\ln(q)}}, \ \ \ \Psi^{reg}_{\infty}(z) = 1+ \sum \limits_{k=1}^{\infty}\Psi^{reg}_{\infty,k}\, z^{-k}.
\ee
The matrix $\Psi^{reg}_{\infty}(z)$ solves the equation:
\be \label{psiinf2}
\Psi^{reg}_{\infty}(z) \, {\mathbf{E}}_{\infty}=\M^{\ast}(z) \Psi^{reg}_{\infty}(z)
\ee
where $\M^{\ast}(z) = (\textsf{P}^{\ast})^{-1}\, \M(z) \, \textsf{P}^{\ast}$ denotes the matrix of $\M(z)$ in the basis of eigenvectors $\textsf{P}^{\ast}$. We denote
$$
\M_{0}^{\ast}(z)={\mathbf{E}}^{-1}_{\infty} \M^{\ast}(z), \ \ \ \M^{\ast}_{k}(z)={\mathbf{E}}^{k}_{\infty}\, \M^{\ast}_{0}(z q^{-k})\, {\mathbf{E}}^{-k}_{\infty} 
$$
then the infinite product
\be \label{infsolut}
\Psi^{reg}_{\infty}(z)=\M_1^{\ast}(z)\M_2^{\ast}(z)\M_2^{\ast}(z)\cdots
\ee
solves (\ref{psiinf2}). Let us denote 
$$
\B^{\ast}_{w}(z):=(\textsf{P}^{\ast})^{-1}\, \B_{w}(z)\, \textsf{P}^{\ast}, \ \ \B^{\ast}_{w}:=(\textsf{P}^{\ast})^{-1}\, \B_{w}\, \textsf{P}^{\ast}
$$
the matrices  of the operators $\B_{w}(z)$ and $\B_{w}$ in the basis $\textsf{P}^{\ast}$.
\begin{Proposition} \label{inflimps}
	{\it The solution (\ref{infsolut}) has the following limits:
		$$
		\lim\limits_{q\to 0} \, \Psi^{reg}_{\infty}(z q^{s})=
		\left\{
		\begin{array}{ll}
		\B^{\ast}_{s}(z^{-1})^{-1} \prod\limits_{w\in [0,s)}^{\longleftarrow}\, (\B^{\ast}_w)^{-1}  & s \geq 0  \\ & \\
		1, & s<0
		\end{array}\right.
	$$}
\end{Proposition}
\begin{proof}
	Let us consider a finite approximation of infinite product (\ref{infsolut}):
	$$
	\Psi^{reg}_{N}(z)=\M_1^{\ast}(z)\M_2^{\ast}(z)\cdots \M_N^{\ast}(z)=\M^{\ast}(z/q)\M^{\ast}(z/q^2)\cdots \M^{\ast}(z/q^N) {\mathbf{E}}^{-N}_{\infty}
	$$
	Using (\ref{btrans}) we rewrite the last expression in the form
	$$
	\Psi^{reg}_{N}(z)=\Big(\prod\limits_{[0,N)}^{\longrightarrow}\, \B^{*}_{w}(z,q)\Big)\, (\M(0)^{*})^{N}\, {\mathbf{E}}^{-N}_{\infty}.
	$$
	From (\ref{blim})  we find:
	\be \label{f1}
	\lim\limits_{q\to 0}\, \Psi^{reg}_{N}(z q^{s})=\B^{*}_{s}(z) \Big(\prod\limits_{(s,N)}^{\longrightarrow}\, \B^{*}_{w} \Big)\, (\M(0)^{*})^{N}\, {\mathbf{E}}^{-N}_{\infty}
	\ee
	Let us note that
	$$
	\M^{\ast}(z) \M^{\ast}(z/q)\cdots \M^{\ast}(z/q^{N-1})=\Big(\prod\limits_{w\in [0,N)} \, \B^{\ast}_{w}(z,q)\Big)\, \M^{\ast}(0)^N
	$$
	where the last equality is by (\ref{btrans}). At $z=\infty$ we obtain
	$$
	{\mathbf{E}}^{N}_{\infty}=\Big(\prod\limits_{w\in [0,N)} \, \B^{\ast}_{w}\Big)\, \M^{\ast}(0)^N
	$$
	which is the same as
	\be \label{f2}
	\Big(\prod\limits_{[0,N)}^{\longrightarrow}\, \B^{*}_{w} \Big)\, (\M(0)^{*})^{N}\, {\mathbf{E}}^{-N}_{\infty}=1.
	\ee
	Dividing (\ref{f1}) by (\ref{f2}) and assuming that $N> s$ we  obtain
	$$
	\lim\limits_{q\to 0}\, \Psi^{reg}_{N}(z q^{s})=\B^{*}_{s}(z) \prod\limits_{[0,s]}^{\longleftarrow}\, (\B^{*}_{w})^{-1}=\B^{*}_{s}(z^{-1})^{-1} \prod\limits_{[0,s)}^{\longleftarrow}\, (\B^{*}_{w})^{-1}
	$$
	where the last equality is by Proposition \ref{invz}.
	We see that for large $N$, the limit does not depend on $N$ and the proposition follows. 
\end{proof}

\subsection{Monodromy of QDE}
The monodromy of QDE is defined as the transition matrix between the two fundamental solutions constructed above:
\be \label{mondef}
\Mon(z,t_1,t_2):=\Psi_{0}(z)^{-1} \Psi_{\infty}(z).
\ee
Clearly, $\Mon( z q,t_1,t_2)=\Mon(z,t_1,t_2)$. 
\begin{Remark}
	If we change a basis of the Fock space by  $A \in \textrm{End}(\textsf{Fock})$ then the fundamental solutions transforms as 
	$$
	\Psi_{0}(z) \to \textsf{A} \Psi_{0}(z), \ \  \Psi_{\infty}(z) \to \textsf{A} \Psi_{\infty}(z)
	$$
	Thus, $\Mon(z,t_1,t_2)$ does not depend on the basis in which  the QDE is considered. 
\end{Remark}
\noindent
It will be convenient to work with {\it the regular part of monodromy} which is defined like (\ref{mondef}) but without exponential factors in (\ref{zerte}) and (\ref{psiinf}):
\be \label{monregmon}
\Mon^{reg}(z,t_1,t_2):= e^{\frac{\ln({\mathbf{E}}_0) \ln(z)}{\ln(q)}}\,\Mon(z,t_1,t_2)\, e^{-\frac{\ln({\mathbf{E}}_{\infty}) \ln(z)}{\ln(q)}}.
\ee
The regular part is not $q$-periodic
\be \label{regqw}
\Mon^{reg}(z q,t_1,t_2) ={\mathbf{E}}_0\, \Mon^{reg}(z,t_1,t_2)\,{\mathbf{E}}^{-1}_{\infty}.
\ee
Combining Propositions \ref{asymzero} and \ref{inflimps} we obtain:
\begin{Theorem} \label{limmon}
	{\it For $s\in \matQ$ the monodromy of qde has the following asymptotic at $q\to 0$:
		\be \label{monlimt}
		\lim_{q\to 0} \Mon^{reg}(z q^{s},t_1,t_2)=\left\{\begin{array}{ll}
			\B^{\bullet}_s(z^{-1})^{-1}\, \prod\limits_{w\in [0,s)}^{\longleftarrow}\, (\B^{\bullet}_w)^{-1}   \cdot \,  \T   & s \geq 0, \\
			\B^{\bullet}_{s}(z)\, \prod\limits^{\longrightarrow}_{w\in (s,0)}\, \B^{\bullet}_{w} \,  \cdot  \T, & s<0.
		\end{array}\right.
		\ee
		where $\B^{\bullet}_{w}(z)$, $ \B^{\bullet}_{w}$ denote the matrices of the operators $\B_{w}(z)$ and $ \B_{w}$ in the basis of fixed points $\mathsf{P}$  and $\T$ is the matrix of transition matrix between the bases $\mathsf{P}$ and $\mathsf{H}^{\ast}$:
		\be \label{tmatdef}
		\T:=\mathsf{P}^{-1} \mathsf{H}^{\ast}. 
		\ee}
	\end{Theorem}
	\begin{Remark}
		The above theorem says that the limit $\lim\limits_{q\to 0} \Mon^{reg}(z q^{s},t_1,t_2)$ is a piecewise constant function of $s \in \matQ$, which changes only when $s$ crosses a ``wall'' from $\Walls_n\subset \matQ$. Moreover, if $s\not \in \Walls_n$ then $\B^{\bullet}_{s}(z)=1$ and the limit is independent on the K\"ahler variable $z$, see example below.
	\end{Remark}
	
	\begin{Example}
		Let us consider the case $n=3$ and $s=-2/3$ then
		$$
		\lim\limits_{q\to 0} \Mon^{reg}(z q^{-2/3},t_1,t_2)=\B^{\bullet}_{-2/3}(z) \B^{\bullet}_{-1/2} \B^{\bullet}_{-1/3}\, \T,
		$$
		but for $n=2$ we obtain
		$$
		\lim\limits_{q\to 0} \Mon^{reg}(z q^{-2/3})= \B^{\bullet}_{-1/2} \, \T,
		$$
		in particular, the last limit does not depend on $z$.
	\end{Example}

\section{Elliptic stable envelope \label{ellsection}}
A new approach to studying the monodromy  was recently suggested in \cite{AOElliptic,OkEll1,OkEll2}. 
In this approach, the monodromy of the QDE for a variety $\X$ is identified with the transition matrix between the {\it elliptic stable bases} in the equivariant elliptic cohomology of $\X$.
In this section we apply this idea to $\X=\textrm{Hilb}^{n}(\matC^2)$. 
In this and next sections we will also use equivariant parameters $a$ and $\hbar$ defined by
$$
t_1=a \hbar^{1/2}, \ \ t_2=a^{-1} \hbar^{1/2}.
$$

\subsection{Multiplicative notations}
Let us introduce the following functions:
$$
\hat{s}(x):=x^{1/2}-x^{-1/2}, \ \ 
\varphi(x):=\prod\limits_{i=0}^{\infty}(1-x q^{i})
$$
the last product converges for $|q|<1$ which we assume throughout.  Let 
\be \label{thetafun}
\theta(x): = \varphi(q x) \hat{s}(x)   \varphi(q x^{-1}).
\ee
denote the $q$-theta function.  We will denote by $\hat{S}$, $\Phi$, $\Theta$  multiplicative extensions of these functions to Laurent polynomials via the rules:
$$
\hat{S}(a+b)=\hat{S}(a) \hat{S}(b), \ \ \Phi(a+b)=\Phi(a) \Phi(b), \ \ \Theta(a+b)=\Theta(a) \Theta(b)
$$
Given a Laurent polynomial $P$ we defile by $\det(P)$ via
$
\det(a+b)=a b. 
$
\begin{Example} 
	Let $P=a-2b +3 c$ then 
	$$
	\Phi(P)=\dfrac{\varphi(a) \varphi(c)^3}{\varphi(b)^2}, \ \ 
	\det(P)=\dfrac{a c^3}{b^2}.
	$$
	If $P^{*}=a^{-1}-2b^{-1} +3 c^{-1}$ denote the dualization in K-theory then
	$$
	\Theta(P)=\Phi(q P) \hat{S}(P) \Phi(q P^{*}).
	$$
	The last formula, clearly, holds for any K-theory class $P$. 
\end{Example}
Let $\tb \in K_{\bT}(\X)$ be the class of the tautological bundle over the Hilbert scheme $\X$. The class 
$$
P=\tb+\tb \otimes \tb^{\ast} t_1 - \tb\otimes \tb^{\ast}
$$
is a polarization of $\X$, in other words, it is half of the class of the tangent bundle:
$$
T \X =P + \hbar^{-1} P^{*}  \in K_{\bT}(\X).
$$
For the computations below, the following Lemma is convenient. 
\begin{Lemma} \label{philemm} {\it 
		$$
		\dfrac{1}{\hat{s} (T \X)}\, \dfrac{\Theta(P) }{\Phi((q-\hbar) P)}=\Phi(q T\X) \, \mathscr{O}(1)^{1/2}\, \hbar^{-n/2}\, t_1^{n^2/2} 
		$$}
\end{Lemma}
\begin{proof}
	Using the multiplicative notation we write:
	$$
	\dfrac{1}{\hat{s} (T \X)}\, \dfrac{\Theta(P) }{\Phi((q-\hbar) P)}=\dfrac{1}{\hat{s} (T \X)}\, \dfrac{\Phi(q P) \hat{s}(P) \Phi(q P^{\ast}) }{\Phi((q-\hbar) P)}=
	\dfrac{ 
		\det(P)^{1/2} \Phi( P^{\ast} + \hbar P) }{\hat{s} (T \X)}.
	$$
	Since $P^{\ast} + \hbar^{-1} P = T \X$ we have
	$$
	\dfrac{1}{\hat{s} (T \X)}\, \dfrac{\Theta(P) }{\Phi((q-\hbar) P)} = \det(P)^{1/2} \det(T\X)^{1/2}  \Phi(q T \X). 
	$$
	Using the explicit form of the polarization and $\det(\tb)=\mathscr{O}(1)$ we compute
	$$
	\det(P)^{1/2}=\det(\tb)^{1/2} t_1^{n^2/2} = \mathscr{O}(1)^{1/2} t_1^{n^2/2}. 
	$$
	As $\X$ is symplectic, with the $\bT$-weight of the symplectic form given by $\hbar$, for each $w \in \textrm{weight}_{\bT}(T_{\lambda} \X)$ we have the symplectic dual dual weight $w^{-1} \hbar^{-1} \in \textrm{weight}_{\bT}(T_{\lambda}\X)$. Thus
	$$
	\det(T \X) = \hbar^{-n}.
	$$
	Lemma follows.
\end{proof}

\subsection{Fixed point components of elliptic stable envelopes} 
For a cocharacter $\sigma \in \Lie_{\matR}(\bA)$ and a fixed point $\lambda \in \X^{\bT}$ the construction of \cite{AOElliptic,OkEll1} provides a class $\Stab^{Ell}_{\sigma}(\lambda)$ in the equivariant elliptic cohomology of $\X$, called the  elliptic stable envelope of~$\lambda$. In this paper we assume that $\sigma$ corresponds to the cocharacter $a\to 0$. 

The classes $\Stab^{Ell}_{\sigma}(\lambda)$ can be characterized by their fixed point components
\be \label{ellipticmat}
U_{\lambda,\mu}(a,z,\hbar):= \left.\Stab^{Ell}_{\sigma}(\lambda)\right|_{\mu}
\ee
By their definition, the components $U_{\lambda,\mu}(a,z,\hbar)$ are sections of a certain line bundle over the abelian variety $E^{\,3}=E_{z}\times E_{a}\times E_{\hbar}$ where $E=\matC^{\times}/q^{\matZ}$ denotes the  elliptic curve with modulus $q$. The parameters $a,z,\hbar$ are viewed as coordinates on the factors. 
For the Hilbert scheme $\X$ the explicit formulas for $U_{\lambda,\mu}(a,z,\hbar)$ in terms of the theta functions were computed in \cite{SmirnovEllipticHilbert}.

Let us denote by $\succ$ the standard dominance order on partitions. From the support condition, in the definition of the elliptic stable envelope  classes $U_{\lambda,\mu}(a,z,\hbar)=0$ if $\lambda \succ \mu$. In other words, the matrix $U_{\lambda,\mu}(a,z,\hbar)$ is {\it upper triangular} if the fixed points are ordered by $\succ$. 

The diagonal elements of this matrix are given explicitly by
$$
U_{\lambda,\lambda}(a,z,\hbar) = \Theta(N^{-}_{\lambda}) := \prod\limits_{{w\in \textrm{char}_{\bT}(T_{\lambda }\X )} \atop {\langle \sigma,w \rangle < 0} } \, \theta(w).
$$
where $T_{\lambda}\X=N^{+}_{\lambda} \oplus N^{-}_{\lambda}$ is the decomposition into repelling and attracting subspaces for $\sigma$, and the product is over the $\bT$-weights of $N^{+}_{\lambda}$.  Similarly, the elliptic stable envelopes for the opposite cocharacter are characterized by the {\it lower triangular} matrix
\be \label{lellmat}
L_{\lambda,\mu}(a,z,\hbar):= \left.\Stab^{Ell}_{-\sigma}(\lambda)\right|_{\mu}
\ee
with $L_{\lambda,\mu}(a,z,\hbar)=0$ if $\mu\succ\lambda$ and 
$$
L_{\lambda,\lambda}(a,z,\hbar) = \Theta(N^{+}_{\lambda}) := \prod\limits_{{w\in \textrm{char}_{\bT}(T_{\lambda }\X )} \atop {\langle \sigma,w \rangle > 0} } \, \theta(w).
$$

\subsection{Properties of the elliptic stable envelope}
Let us consider the automorphism of the torus 
$$
\iota : \bT\to \bT, \ \ \ \iota(t_1,t_2)=(t_2,t_1).
$$
and let $\iota^{*}$ be the corresponding induced maps in equivariant K-theory or elliptic cohomology. On the classes of fixed points it acts by $\lambda \to \lambda^{'}$ ($\lambda'$ denotes the transposed partition). Note that $\iota(a)=a^{-1}$ and thus the corresponding map of the Lie algebras maps the cocharacter $\sigma$ to $-\sigma$.   The uniqueness of the elliptic stable envelope classes thus implies that:
$$
\iota^{*} (\Stab^{Ell}_{\sigma}(\lambda))= \Stab^{Ell}_{-\sigma}(\lambda').
$$
In the fixed point components this gives:
\begin{Lemma} \label{lulemma}
	{\it The matrices of stable envelopes (\ref{ellipticmat}) and (\ref{lellmat}) are related by
		\be \label{luslem}
		L(a,z,\hbar) = \fp\ U(a^{-1},z,\hbar)\, \fp
		\ee
		where $\fp$ is the antidiagonal matrix describing the transformation $\lambda\to \lambda'$. In the basis of partitions ordered by $\succ$ it has the form:
		$$
		\fp_{\lambda,\mu}:=\delta_{\lambda,\mu^{'}}=\left(\begin{array}{lll}
		0&\cdots& 1\\
		\cdots&\cdots & \cdots\\
		1& \cdots & 0
		\end{array}\right).
		$$}
\end{Lemma}
\noindent
Comparing the diagonal elements of (\ref{luslem}) we find
\be \label{sthetn}
\fp \left.\Theta(N^{-})\right|_{a=1/a} \fp= \Theta(N^{+})
\ee
The following property is also convenient for explicit computations:
\begin{Lemma}  \label{invlem}
	{\it There are matrix identities 
		$$
		L^{t}(a,z^{-1} \hbar^{-1},\hbar) U(a,z,\hbar)=\Theta(T \X)
		$$
		where $\Theta(T \X)$ denotes the diagonal matrix with
		$
		\Theta(T \X)_{\lambda,\lambda}= \Theta(T_{\lambda} \X).
		$
	}
\end{Lemma}
\begin{proof}
	This identity is Proposition 3.4 in \cite{AOElliptic} applied to $\X$. 
\end{proof}
\noindent
In this and following sections we use $\,\tilde\,$ to denote the normalized matrices of stable envelopes, defined by
$$
\tilde{\bf{U}}_{\lambda,\mu}(a,z):=\dfrac{{{U}}_{\lambda,\mu}(a,z)}{{{U}}_{\mu,\mu}(a,z)}, \ \ \ \tilde{{L}}_{\lambda,\mu}(a,z):=\dfrac{{{L}}_{\lambda,\mu}(a,z)}{{{L}}_{\mu,\mu}(a,z)},
$$
so that $\tilde{\bf{U}}_{\lambda,\lambda}(a,z)=\tilde{\bf{L}}_{\lambda,\lambda}(a,z)=1$. 
Note that the previous lemma implies that
\be \label{invluiden}
\tilde{\bf{L}}^{t}(a,z^{-1} \hbar^{-1}) \tilde{\bf{U}}(a,z)=Id.
\ee

\subsection{Twisted elliptic stable envelopes}
It is convenient to introduce twisted elliptic envelopes by
\be \label{stabnorm}
\textbf{Stab}^{Ell}_{\sigma}(\lambda) =\kappa^{*}(\Theta(N^{+}_{\lambda}))\, {\Stab}^{Ell}_{\sigma}(\lambda)
\ee
where $\kappa^{*}$ denotes the change of variables
\be \label{mirsym}
\kappa^{*}: \ \ a \to z \sqrt{\hbar},  \ \ \hbar \to 1/\hbar,  \ \ z \to a \sqrt{\hbar}. \ \ \
\ee
Note that the new prefactor $\kappa^{*}(\Theta(N^{\pm}_{+}))$ only depends on $z$ and $\hbar$. 
We denote the corresponding matrices of the fixed point components
$$
{\bf{U}}_{\lambda,\mu}(a,z): = \left.\textbf{Stab}^{Ell}_{\sigma}(\lambda)\right|_{\mu}, 
$$
and as in Lemma \ref{lulemma} we define: 
$$
\ \ \  
{\bf{L}}(a,z): = \fp {\bf{U}}(a^{-1},z) \fp
$$
Let $\kappa^{*}(\Theta(N^{\pm}))$ be the diagonal matrices with diagonal  $\kappa^{*}(\Theta(N^{\pm}_{\lambda}))$ then 
\be \label{lunorms2}
{\bf{U}}(a,z)=\kappa^{*}(\Theta(N^{+})) {{U}}(a,z), \ \ \ {\bf{L}}(a,z)=\left.\kappa^{*}(\Theta(N^{-}))\right|_{z=z^{-1} \hbar^{-1}} {{L}}(a,z).
\ee

\subsection{Mirror conjecture for $\X$}
The twisted version of the elliptic stable envelope (\ref{stabnorm}) behaves better with respect to the so called {\it 3D-mirror symmetry}. Informally, the mirror conjecture states that there exists a dual variety $\X^{!}$ so that the twisted elliptic stable envelopes of $\X$ and $\X^{!}$ coincide after identification of K\"ahler and equivariant parameters by (\ref{mirsym}). In terms of the fixed point components this can be formulated as follows.
\begin{Conjecture} \label{conjmir}
	{\it The Hilbert scheme is self-dual with respect to $3D$-mirror symmetry $\X\cong \X^{!}$ and
		\be \label{mrshilb}
		{\bf{U}}(a,z) = \kappa^{*}\Big({\bf{L}}(a,z^{-1} \hbar^{-1})^t \Big)
		\ee where $t$ denotes transposed matrix. }
\end{Conjecture}
\noindent
By Lemma \ref{invlem} this conjecture is equivalent to the identity:
$$
\tilde{\bf{U}}(a,z)^{-1} =\kappa^{*}(\tilde{\bf{U}}(a,z)) 
$$
For explicit examples of $3D$-mirror symmetry of the elliptic stable envelope we refer to \cite{RSVZ1,RSVZ2,SmirnovZhuHypertoric,RW}, see also \cite{Dink5} for approach which uses {\it vertex functions}. The proof of $3D$-mirror symmetry for  the Hilbert scheme (\ref{mrshilb}) is a work in progress by several groups \cite{KorZeitToroid,OkAganToAppear}.  

\begin{Remark}
	Conjecture \ref{conjmir} implies that the twisted stable envelopes are fixed point components of a certain class $\frak{m}$ in the equivariant elliptic cohomology of $\X\times \X^{!}$, called  {\it duality interface} in \cite{SmirnovZhuHypertoric}. 
\end{Remark}

\begin{Remark}
	Conjecture \ref{conjmir} can be checked directly for several first values of $n$ using a computer and the explicit formulas from \cite{SmirnovEllipticHilbert}.
\end{Remark}

\subsection{Monodromy from the elliptic stable envelope}
The one-leg vertex functions with a descendent $\tau \in K_{\bT}(\X)$: 
$$
{\bf V}^{(\tau)}_{0}(z)  \in K_{\bT}(\X)[[z]]
$$
of the Hilbert scheme $\X$ are defined as the generating series of quasimaps to $\X$, see Section 7.2 of \cite{Okpcmi} for the definition.
The vertex functions are certain generalizations of hypergeometric functions. In particular, they provide a basis  to a certain system of $q$-hypergeometric equations, which are holomorphic near $z=0$ and satisfy ${\bf V}^{(\tau)}_{0}(0)=\tau$. Solving the same hypergeometric system at $z=\infty$ with the same boundary conditions gives functions
$$
{\bf V}^{(\tau)}_{\infty}(z)  \in K_{\bT}(\X)[[z^{-1}]]
$$
The functions ${\bf V}^{(\tau)}_{\infty}(z)$ can be considered as vertex functions for the Nakajima variety associated with the same combinatorial data (quiver and dimension vectors) but with the opposite stability condition. We need the following result of A.Okounkov \cite{OkEll2}, which describes the monodromy of the properly normalized vertex functions. 

\begin{Theorem}
	{\it 
		Let us consider the vertex functions of $\X$ normalized so that 
		\be \label{vernorm}
		\tilde{{\bf V}}_{0}(z)=     \frac{\Phi((q-\hbar) P) }{\Theta(P)}\, {{\bf V}}_{0}(z), \ \ \ \tilde{{\bf V}}_{\infty}(z)=     \frac{\Phi((q-\hbar) P) }{\Theta(P)}\, {{\bf V}}_{\infty}(z)
		\ee
		then 
		\be \label{vermonoddes}
		\tilde{{\bf V}}_{\infty}(z)={\bf \Lambda}(z) \tilde{{\bf V}}_{0}(z)
		\ee
		where the monodromy matrix is expressed via the elliptic stable envelopes (\ref{lunorms2}) by:
		\be  \label{vertmondr}
		{\bf \Lambda}(z)= (-1)^n\, {\bf{L}}(a^{-1},z^{-1} \hbar^{-1})^{-1} \,{\bf{U}}(a,z). 
		\ee
	}
\end{Theorem}
\begin{proof}
	This is Corollary 3.2 in \cite{OkEll2} applied to Hilbert scheme $\X = \textrm{Hilb}^{n}(\matC^2)$. We only need to explain the sign $(-1)^{n}$: in \cite{OkEll2} the normalization of stable envelope is fixed by formula (45) in \cite{AOElliptic}.
	It differs from the one accepted in this paper by 
	$(-1)^{\textrm{rk}(\textrm{ind}^{\sigma}_{\lambda})}$ where $\textrm{ind}^{\sigma}_{\lambda}$ denotes the index of the fixed point $\lambda$ corresponding to the chamber $\sigma$. Thus, the ratio of stable envelopes ${\bf{L}}(a, \hbar)$ and ${\bf{U}}(a,z)$  differs from those in Corollary 3.2 of \cite{OkEll2} by a sign
	$$
	(-1)^{\textrm{rk}(\textrm{ind}^{+\sigma}_{\lambda})-\textrm{rk}(\textrm{ind}^{-\sigma}_{\lambda})} =(-1)^n 
	$$
	The last identity is by direct computation as in Section 3.8 of  \cite{SmirnovEllipticHilbert}.
\end{proof}
The fundamental solutions $\Psi_{0}(z)$, $\Psi_{\infty}(z)$ play the role of the {\it capping operators} in the enumerative geometry. The relation between the bare vertex functions, capping operators and capped vertex functions is the following, see Section 7.4 of \cite{Okpcmi}:
\be \label{basrecap}
\Psi_{0}(z) \, \hat{s}(T \X)^{-1}\, {\bf V}^{(\tau)}_{0}(z) = \Psi_{\infty}(z) \,\hat{s}(T \X)^{-1}\, {\bf V}^{(\tau)}_{\infty}(z) =\langle \tau \rangle \in K_{\bT}(\X)(z)
\ee
where $\langle \tau \rangle \in K_{\bT}(X)(z)$ is known as the capped vertex with a descendent $\tau$. As a rational function of $z$, the capped vertex $\langle \tau \rangle$ has trivial monodromy. Thus, the equation (\ref{basrecap}) says that the monodromy of the vertex function and the capping operators (the solutions of QDE) are inverses of each other. Combining all the factors together we find:
\begin{Theorem}
	{\it Let us normalize solutions of QDE by 
		\be \label{psinorm}
		\tilde{\Psi}_{0}(z)= {\Psi}_{0}(z) \Phi( q\, T\X), \ \ \tilde{\Psi}_{\infty}(z)= {\Psi}_{\infty}(z) \Phi( q\, T\X)
		\ee
		where $\Phi( q\, T\X)$ denotes the diagonal operator 
		$$
		\Phi( q\, T\X)_{\lambda,\lambda}=\prod\limits_{w\in \{\bT-\textrm{weights of} \, \,  T_{\lambda}\X \}} \, \varphi(q w)
		$$
		(the product runs over the $\bT$-weights appearing in the tangent space $T_{\lambda} \X$ at a fixed point $\lambda \in \X^{\bT}$).
		Then, we have
		$$
		\tilde{\Psi}_{0}(z)=\tilde{\Psi}_{\infty}(z)\, \tilde{\bf \Lambda}(z)
		$$
		where 
		$$
		\tilde{\bf \Lambda}(z)=\mathscr{O}(1)^{1/2}\,{ \bf \Lambda}(z)\, \mathscr{O}(1)^{-1/2}.
		$$
		and ${\bf \Lambda}(z)$ is given by (\ref{vertmondr})}.
\end{Theorem}
\begin{proof}
	
	Applying Lemma \ref{philemm} to normalizations (\ref{vernorm}) and (\ref{psinorm}) we write (\ref{basrecap}) in the form:
	$$
	\tilde{\Psi}_{0}(z)\, \mathscr{O}(1)^{1/2}  \, \tilde{\bf V}^{(\tau)}_{0}(z) =
	\tilde{\Psi}_{\infty}(z)\, \mathscr{O}(1)^{1/2}\, \tilde{\bf V}^{(\tau)}_{\infty}(z)
	$$
	The theorem then follows from (\ref{vermonoddes}).
\end{proof}

For the monodromy (\ref{mondef}) we obtain 
\begin{Corollary} \label{corrmonodgam}
	{\it The monodromy of the QDE equals:
		\be \label{gaussell}
		\begin{array}{l}
			\Mon^{reg}(z)=\\ 
			\\
			(-1)^{n}\,\Phi( q\, T\X) \mathscr{O}(1)^{1/2}\, {\bf{U}}(a,z)^{-1} \, {\bf{L}}(a^{-1},z^{-1} \hbar^{-1})\,\mathscr{O}(1)^{-1/2} \, \Phi( q\, T\X)^{-1}
		\end{array}
		\ee}
\end{Corollary}

\begin{Remark}
	Note that, in this corollary,
	$$
	\Phi( q\, T\X) \mathscr{O}(1)^{1/2}\, {\bf{U}}(a,z)^{-1}$$ is an upper triangular and 
	$$
	{\bf{L}}(a^{-1},z^{-1} \hbar^{-1})\,\mathscr{O}(1)^{1/2} \Phi( q\, T\X)^{-1}
	$$ is a lower-triangular matrix. Thus, the corollary  provides a {\it Gauss decomposition} of the monodromy.
\end{Remark}
\section{Gauss decomposition of wall-crossing operators \label{gaussBsection}}
Theorem \ref{limmon} says that the wall-crossing operators $\B^{\bullet}_{w}(z)$ appear as $q\to0$ limits of the monodromy $\Mon^{reg}(z)$. Thus, by Corollary \ref{corrmonodgam} these operators can be expressed via limits of elliptic stable envelopes ${\bf{U}}(a,z)$ and ${\bf{L}}(a,z)$. In this section, following \cite{KononovSmirnov1,KononovSmirnov2} we show that $q\to0$ limits of matrices ${\bf{U}}(a,z)$ factors into product of {\it K-theoretic stable envelopes} of $\X$ and its 3D-mirror $\X^{!}$. As a result, we obtain natural Gauss decomposition of matrices the $\B^{\bullet}_{w}(z)$ and $\B^{\bullet}_{w}$ expressed via K-theoretic stable classes of $\X$ and $\X^{!}$. 

\subsection{K-theoretic stable envelope} 
In the limit $q\to 0$, the elliptic cohomology scheme of $\X$ degenerates to the scheme $\textrm{spec}(K_{\bT}(\X))$. The limit of the elliptic stable envelope then gives a section of the trivial line bundle over $\textrm{Spec}(K_{\bT}(\X))$, i.e., the K-theory class. Here is the precise statement: 

\begin{Theorem} \label{Kththem}
	{\it For generic $\sl\in H^{2}(\X,\matQ) \cong \matQ$ we have
		\be \label{ktheorlimdef}
		\lim\limits_{q\to 0} \left.\Stab^{Ell}_{\sigma}(\lambda)\right|_{z=z q^{\sl}}=\Stab^{\X,Kth,[\sl]}_{\sigma}(\lambda) \in K_{\bT}(\X)\otimes \matQ[t_1^{\pm 1/2},t_2^{\pm 1/2}]
		\ee
		where $\Stab^{\X,Kth,[\sl]}_{\sigma}(\lambda)$ is the K-theoretic stable envelope of a fixed point $\lambda$ with a slope $\sl$. In particular, for the matrices of the fixed point components (\ref{ellipticmat}) we have
		$$
		\lim\limits_{q\to 0}\, U_{\lambda,\mu}(a,z q^{\sl})=
		\left.\Stab^{\X,Kth,[\sl]}_{\sigma}(\lambda)\right|_{\mu} \in \matQ[t_1^{\pm 1/2},t_2^{\pm 1/2}]. 
	$$ }
\end{Theorem}
For the definition of K-theory classes $\Stab^{\X,Kth,[\sl]}_{\sigma}(\lambda)$ we refer to Section 9 of \cite{Okpcmi} or Section 2 of \cite{OS}.

\begin{proof}
	Proposition 4.3 in \cite{AOElliptic}.
\end{proof}
\begin{Remark}
	We note that in \cite{AOElliptic}, to get rid of square roots, the K-theoretic limit is additionally twisted by the square root of the polarization. In this case, the limit is an element of $K$-theory:
	$$
	\det (P^{1/2}) \otimes \lim\limits_{q\to 0} \left. \Stab^{Ell}_{\sigma}(\lambda)\right|_{z=z q^{\sl}} \in K_{\bT}(\X)
	$$
	We will use (\ref{ktheorlimdef}) as definition of K-theoretic stable envelope in this paper. It differs from the K-theoretic stable envelope of \cite{AOElliptic} by  the factor $\det (P^{1/2})$.
\end{Remark}

\subsection{K-theoretic limit for general slopes}
For the Hilbert scheme $\X$, by generic slope in Theorem \ref{Kththem} we mean $\sl \in \matQ\setminus \Walls_n$.
For this paper we need a more general version of Theorem \ref{Kththem}, which includes the limits for non-generic slopes $\sl \in \Walls_n$. The limits of this kind were studied in \cite{KononovSmirnov2}, in particular see Section 8 in \cite{KononovSmirnov2} for discussion of the Hilbert scheme $\X$. 

\noindent
For $w=\frac{a}{b} \in \matQ$ we consider the following cyclic subgroup of $\bA$:
\be \label{cycmu}
\mu_{w} = \{e^{ 2 \pi i k w}: k=0,\dots, b-1 \} \subset \bA
\ee
As a subgroup of $\bA$ is acts naturally on $\X$. Let 
$$
\Y_{w} = \X^{\mu_{w}}\subset \X.
$$
be its fixed point set. 
\begin{Proposition}
	{\it The subvariety $\Y_w$ is a union of connected components:
		\be \label{cvar}
		\Y_w = \coprod_{n_0+\dots+n_{b-1}=n}\, \X(n_0,\dots,n_{b-1}) 
		\ee
		where $\X(n_0,\dots,n_{b-1})$ is isomorphic to the Nakajima variety associated with cyclic quiver of length $b$ with dimension vector $\textsf{v}=(n_0,\dots,n_{b-1})$ and framing vector 
		$\textsf{w}=(1,0,\dots,0)$, see Fig \ref{cicqvpic}: }
	\begin{figure}[H]
		\centering
		\includegraphics[width=4cm]{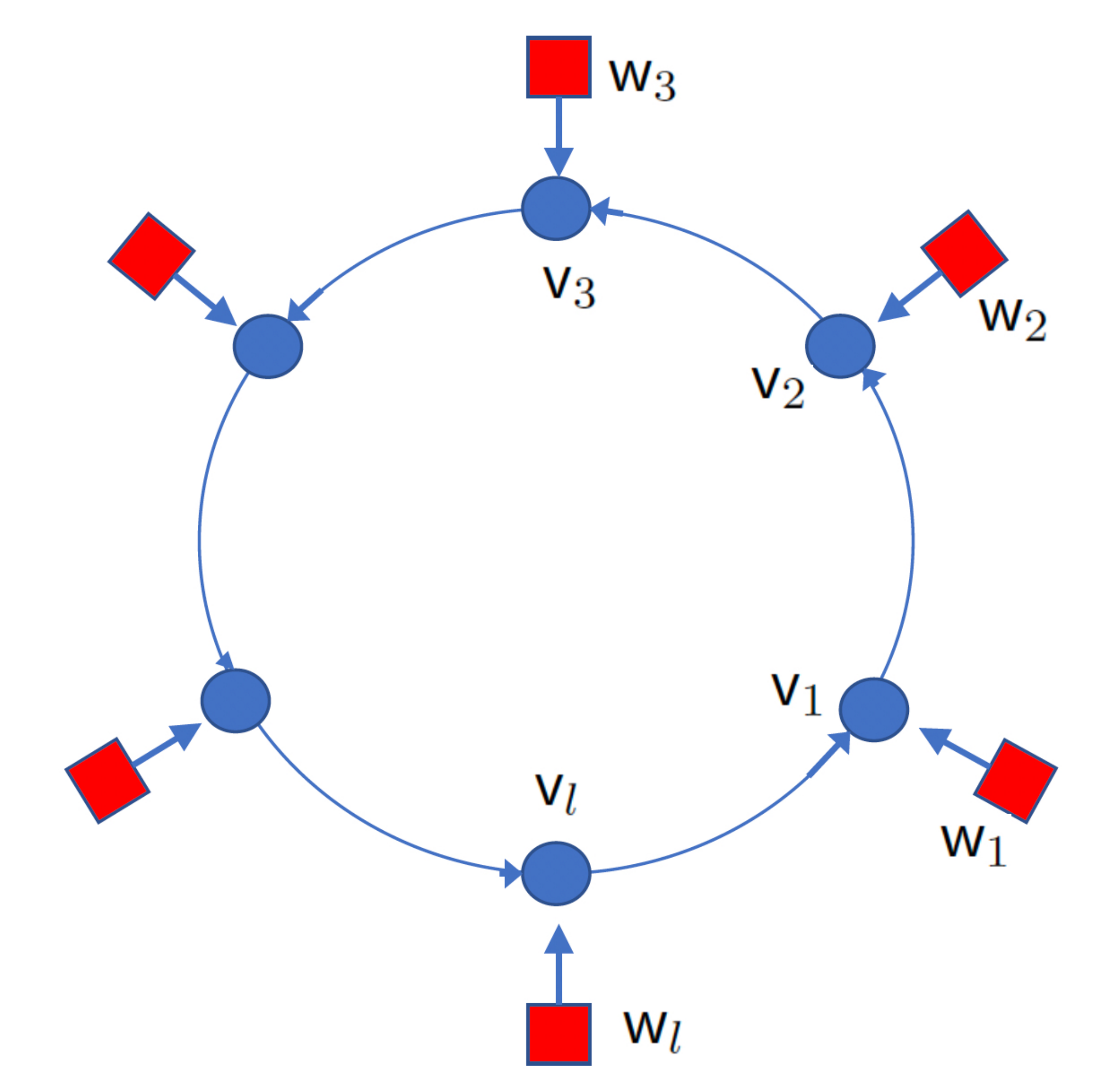}
		\caption{\label{cicqvpic} A cyclic quiver with dimension vector $\textsf{v}$ and framing vector $\textsf{w}$.}
	\end{figure}
\end{Proposition}
\begin{proof}
Proposition 6 in \cite{KononovSmirnov1}.
\end{proof}

\begin{Remark} \label{rempoints}
We note that if at least one $n_i=0$ then the cyclic quiver variety $$\X(n_0,\dots,n_{b-1})\cong pt$$ or empty, see \cite{dinksmir2} for combinatorial description of zero dimensional $A_n$-type quiver varieties. In particular  for $w=\frac{a}{b}$ with $b>n$ all components of $\Y_{w}=\X^{\bT}$ are points. This can be summarized as: 
\be \label{ywalls}
\{w\in \matQ:  \Y_{w} \neq \Y^{\bT} \} = \Walls_{n}.
\ee
\end{Remark}

\subsection{K-theory limit of the elliptic stable envelopes} 
Let ${\bf U}^{+}_w(a)$ (respectively ${\bf U}^{-}_w(a)$) denote the matrix of fixed point components of the K-theoretic stable envelopes of $\X$ with slope 
$w+\epsilon$ (respectively $w-\epsilon$) for small ample $0<\epsilon \ll 1$:
\be \label{defrestr}
{\bf U}^{\pm}_w(a)_{\lambda,\mu}: = \left.\Stab_{\sigma}^{\X,Kth,w\pm \epsilon }(\lambda)\right|_{\mu}, \ \ \ {\bf L}^{\pm}_w(a)_{\lambda,\mu}: = \left.\Stab_{-\sigma}^{\X,Kth,w\pm \epsilon }(\lambda)\right|_{\mu}.
\ee
We will also need the normalized matrices:
$$
\tilde{\bf U}^{\pm}_w(a)_{\lambda,\mu}: = \dfrac{{\bf U}^{\pm}_w(a)_{\lambda,\mu}}{{\bf U}^{\pm}_w(a)_{\mu,\mu}}, \ \ \ \tilde{\bf L}^{\pm}_w(a)_{\lambda,\mu}: = \dfrac{{\bf L}^{\pm}_w(a)_{\lambda,\mu}}{{\bf L}^{\pm}_w(a)_{\mu,\mu}}.
$$
Matrices (\ref{defrestr}) in cases $n=2,3$ can be found in Appendix~\ref{stabxapp}. 
\begin{Remark} \label{nowallrem}
	If $w \not \in \Walls_{n}$ then ${\bf U}^{+}_w(a)={\bf U}^{-}_w(a)$.
	It is also known that   ${\bf U}^{+}_w(a)={\bf U}^{-}_w(a)$, ${\bf L}^{+}_w(a)={\bf L}^{-}_w(a)$ for $w\in \matZ$, see \cite{GorskyNegut}.
\end{Remark}

Similarly we denote by ${ U}^{+}_{\Y_w}$, ${ L}^{+ }_{\Y_w}$ (respectively ${U}^{-}_{\Y_w}$, ${L}^{- }_{\Y_w}$) the matrices of $K$-theoretic stable envelopes for the cyclic quiver variety variety $\Y_{w}$ with small ample (respectively anti-ample) slope: 
\be \label{lufpy}
{U}^{\pm }_{\Y_w}(a)_{\lambda,\mu}={\left.\Stab^{\Y_{w},Kth,\pm \epsilon }_{\sigma}(\lambda)\right|_{\mu}}, \ \ \ 
{L}^{\pm }_{\Y_w}(a)_{\lambda,\mu}={\left.\Stab^{\Y_{w},Kth,\pm \epsilon }_{-\sigma}(\lambda)\right|_{\mu}}
\ee
We denote the corresponding normalized matrices by
$$
\tilde{U}^{\pm }_{\Y_w}(a)_{\lambda,\mu}=\dfrac{{U}^{\pm }_{\Y_w}(a)_{\lambda,\mu}}{{U}^{\pm }_{\Y_w}(a)_{\mu,\mu}}, \ \ \ \tilde{L}^{\pm }_{\Y_w}(a)_{\lambda,\mu}=\dfrac{{L}^{\pm }_{\Y_w}(a)_{\lambda,\mu}}{{L}^{\pm }_{\Y_w}(a)_{\mu,\mu}}.
$$
Matrices (\ref{lufpy}) for $n=2,3$ can be found in Appendix \ref{stabyapp}.
Applying Theorem~\ref{Kththem} to (\ref{invluiden}) we obtain:
\begin{Lemma} \label{kthinvlem}
	{\it There are matrix identities:
		$$
		\Big(\tilde{L}^{\pm }_{\Y_w}(a)\Big)^{t}\, \tilde{U}^{\mp }_{\Y_w}(a)=1, \ \ \ \Big(\tilde{\bf L}^{\pm}_w(a)\Big)^{t}\, \tilde{\bf U}^{\mp}_{-w}(a)=1.
		$$}
\end{Lemma}
Applying Theorem \ref{Kththem} to Lemma \ref{lulemma} we obtain:
\begin{Lemma} \label{luskthlem}
	{\it There are matrix identities:
		$$
		{L}^{\pm }_{\Y_w}(a) = \fp \, {U}^{\pm }_{\Y_w}(a^{-1}) \, \fp, \ \ \ {\bf L}^{\pm}_w(a) = \fp\, {\bf U}^{\pm}_w(a^{-1})\, \fp.
		$$}
\end{Lemma}
\noindent 
We are ready to formulate the main theorem of this section, describing the K-theoretic limit of the elliptic stable envelope for arbitrary slopes:
\begin{Theorem} \label{elslim}
	{\it If the mirror symmetry Conjecture \ref{conjmir} holds, then for all $w\in H^{2}(\X,\matQ)\cong\matQ$ we have the following limits for the elliptic stable envelope matrix (\ref{ellipticmat}):
		$$
		\lim\limits_{q\to 0} {{U}}(a,z q^{w}) = \gamma_{w}(a)\, \kappa^{*}(\tilde{ L}^{\mp }_{\Y_w}(a))^{t}\, \gamma_{w}(a)^{-1} \, { \bf U}^{\pm}_w(a)
		$$
		Here $\kappa^{*}$ denotes substitution (\ref{mirsym}) and $\gamma_{w}(a)$ denotes the diagonal matrix with elements
		\be \label{gamm}
		\gamma_{w}(a)_{\lambda,\lambda}=  \prod\limits_{\Box \in \lambda}\, (-a)^{w c_{\lambda}(\Box) } (-\hbar^{-\frac{1}{2}})^{\lfloor w h_{\lambda}(\Box) \rfloor + \frac{1}{2}}.  
		\ee
		where $c_{\lambda}(\Box)$ and $h_{\lambda}(\Box)$ are the standard content and hook lengths of a box $\Box$ in a Young diagram $\lambda$.
	}
\end{Theorem}
\begin{proof}
	Theorem 3 and Theorem 4 in \cite{KononovSmirnov2} applied to the Hilbert scheme $\X$.  
\end{proof}

\begin{Remark}
	For generic $w\not \in \Walls_n$ by (\ref{ywalls}) we have $\tilde{ L}^{\mp }_{\Y_w}(a)=1$ and the last theorem reduces to the Aganagic-Okounkov K-theoretic limit - Theorem~\ref{Kththem}.
\end{Remark}

\begin{Remark}
	By Lemma \ref{kthinvlem} we have
	$$
	(\tilde{ U}^{\pm }_{\Y_w}(a) )^{-1}=(\tilde{L}^{\mp }_{\Y_w}(a))^{t}
	$$
	and the above theorem can be also written as
	\be \label{betterfactor}
	\lim\limits_{q\to 0} {{U}}(a,z q^{w}) = \gamma_{w}(a)\, \kappa^{*}(\tilde{U}^{\pm}_{\Y_w}(a))^{-1}\, \gamma_{w}(a)^{-1} \, {\bf U}^{\pm}_w(a).
	\ee
\end{Remark}
\noindent
The following identity will be convenient for computations:
\begin{Lemma} \label{lemgamma}
	{\it In the basis of fixed points the operator 
		$\gamma_{w}(a)^{-1} \fp \gamma_{-w}(a) \fp$ acts by
		$$
		\gamma_{w}(a)^{-1} \fp \gamma_{-w}(a) \fp : \ \  P_{\lambda} \to \Big((-1)^{\frac{\mathrm{codim}_{\X}(\Y_{w})}{2} } \hbar^{\frac{\mathrm{codim}_{\X}(\Y_{w})}{4}} \prod\limits_{\Box\in \lambda} \,\hbar^{\lfloor h_{\lambda}(\Box) \cdot w  \rfloor} \Big)\, P_{\lambda}
		$$
		where  $\textrm{codim}_{\X}(\Y_{w})$ denotes the codimension in $\X$ of the component of (\ref{cvar}) containing the fixed point $\lambda$.} 
\end{Lemma}

\subsection{Gauss decomposition of $\B^{\bullet}_{w}(z)$}
Theorem \ref{elslim} suggests to introduce the following matrices:
\be \label{lunorms}
\begin{array}{l}
	{\bf {U}}^{\pm}_{\Y_{w}}(a,z,\hbar):=\gamma_{w}(a)\, \kappa^{*} ({{U}}^{\pm}_{\Y_{w}}(a)) ,
	\\
	\\
	{\bf {L}}^{\pm}_{\Y_{w}}(a,z,\hbar):=\fp \, \gamma_{-w}(a)\, \kappa^{*} ({{U}}^{\pm}_{\Y_{w}}(a^{-1})) \fp. 
\end{array}
\ee
By Lemma \ref{luskthlem} the last one takes the form
$$
{\bf {L}}^{\pm}_{\Y_{w}}(a,z,\hbar):=\fp \, \gamma_{-w}(a)\, \fp\, \kappa^{*}\, ({{L}}^{\pm}_{\Y_{w}}(a)). 
$$
We are ready to formulate the main result of this section:
\begin{Theorem} \label{factthm}
	{\it For all $w\in \matQ$ and $l_1,l_2\in\{+,-\}$ we have the following matrix identities

		\vspace{2mm}
		\noindent
		if $w<0$ then
		$$
		\begin{array}{l}
		\Big(\B^{\bullet}_{w}(z)\prod\limits_{s \in (w,0)}^{\longrightarrow} \, \B^{\bullet}_{s} \Big) \,\cdot  \T= 
		\\
		\\ (-1)^{n}\, \,\bU^{l_1}_{w}(a,\hbar)^{-1}  \bU^{l_1}_{\Y_{w}}(a,z,\hbar)   \bL^{l_2}_{\Y_{w}}(a,z,\hbar)^{-1}  \bL^{l_2}_{-w}(a^{-1},\hbar),
		\end{array}
		$$
		if $w \geq0$ then
		$$
		\begin{array}{l}
		\Big(\B^{\bullet}_{w}(z^{-1})^{-1} \prod\limits_{s\in [0,w)}^{\longleftarrow} \, (\B^{\bullet}_{s})^{-1} \Big) \cdot \T=\\
		\\
		(-1)^{n} \,\bU^{l_1}_{w}(a,\hbar)^{-1}  \bU^{l_1}_{\Y_{w}}(a,z,\hbar)   \bL^{l_2}_{\Y_{w}}(a,z,\hbar)^{-1}  \bL^{l_2}_{-w}(a^{-1},\hbar) 
		\end{array}
		$$
		where $\B^{\bullet}_{w}(z)$ and $\B^{\bullet}_{w}$ denote the matrices of operators $\B_{w}(z)$ and $\B_{w}$ in the basis of fixed points.}
\end{Theorem}

\begin{Remark}
	The matrices $\bU^{\pm}_{w}(a,\hbar),  \bU^{\pm}_{\Y_{w}}(a,z,\hbar)$ are upper-triangular and   $\bL^{\pm}_{w}(a),\bL^{\pm}_{\Y_{w}}(a,z,\hbar)$ are lower-triangular if the fixed point basis is ordered by~$\succ$, see examples in Appendix \ref{stabxapp}. Thus, the above theorem describes a Gauss decomposition of the matrices $\B^{\bullet}_{w}(z)$, $\B^{\bullet}_{w}$. 
\end{Remark}

\begin{proof}
	We compute the limit of monodromy (\ref{monlimt}) using Gauss decomposition~(\ref{gaussell}). First we note that the factors
	$\Phi( q\, T\X)$ in (\ref{gaussell}) do not depend on $z$ and 
	$$
	\lim_{q\to 0} \Phi( q\, T\X)=1.
	$$
	Thus, these factors do not contribute to the limit and we can ignore them.
	Let us denote by $\Theta(N^{\pm})$ the diagonal matrix with diagonal elements
	$$
	\Theta(N^{+}_{\lambda})=\prod\limits_{\Box\in \lambda} \vartheta(a^{h_{\lambda}(\Box)} \hbar^{1/2(-a_{\lambda}(\Box)+l_{\lambda}(\Box)+1)}), 
	$$
	$$
	\Theta(N^{-}_{\lambda})=\prod\limits_{\Box\in \lambda} \vartheta(a^{-h_{\lambda}(\Box)} \hbar^{1/2(a_{\lambda}(\Box)-l_{\lambda}(\Box)+1)}).
	$$
	By (\ref{sthetn}) we have
	$
	\fp\, \Theta(N^{-}_{\lambda})\, \fp =\left.\Theta(N^{+}_{\lambda})\right|_{a=a^{-1}}.
		$
		Using this identity and (\ref{lunorms2}) we write (\ref{gaussell}) as a product of three factors:
		$$
		U(z,a)^{-1}\, 
		\kappa^{*}\Big(\frac{\Theta(N_{\lambda}^{-})}{\Theta(N_{\lambda}^{+})}\Big) \,\fp \,U(z^{-1} \hbar^{-1},a)\, \fp.
		$$
		We compute the limits of these factors separately. 
		First, by (\ref{betterfactor}) we have
		$$
		\lim\limits_{q\to 0} U(z q^{w},a)^{-1}=\bU^{l_1}_{w}(a)^{-1} \gamma_{w}(a)\, \kappa^{*}(\tilde{U}^{l_1}_{\Y_{w}}(a)) \, \gamma_{w}(a)^{-1}, 
		$$
		for any $l_1 \in \{+,-\}$ and by (\ref{lunorms}) we obtain
		$$
		\lim\limits_{q\to 0} U(z q^{w},a)^{-1}=\bU^{l_1}_{w}(a)^{-1}\,{\bf {U}}^{l_1}_{\Y_{w}}(a,z,\hbar)\,\gamma_{w}(a)^{-1}. 
		$$
		Similarly, we compute
		$$
		\lim\limits_{q\to 0} \fp U(z^{-1}\hbar^{-1} q^{-w},a) \fp= \fp \, \gamma_{-w}(a) \kappa^{*} (\tilde{U}^{l_2}_{\Y_{w}}(a^{-1})^{-1}) \gamma_{-w}(a)^{-1} \bU^{l_2}_{-w}(a)\, \fp,
		$$
		for any $l_2\in \{+,-\}$. As $\fp^2=1$ we can write the last expression in the form
		$$
		\fp \, \gamma_{-w}(a) \, \fp\, \fp \, \kappa^{*} (\tilde{U}^{l_2}_{\Y_{w}}(a^{-1})^{-1}) \gamma_{-w}(a)^{-1}\,  \fp\, \fp \, \bU^{l_2}_{-w}(a)\, \fp,
		$$
		By Lemma \ref{luskthlem} we have
		$
		\fp \, \bU^{l_2}_{-w}(a)\, \fp=\bL^{l_2}_{-w}(a^{-1})
		$
		and by (\ref{lunorms}) we obtain:
		$$
		\lim\limits_{q\to 0} \fp\, U(z^{-1}\hbar^{-1} q^{-w},a) \, \fp=\fp \, \gamma_{-w}(a) \, \fp\, {\bf {L}}^{l_2}_{\Y_{w}}(a,z,\hbar)^{-1} \bL^{l_2}_{-w}(a^{-1}).
		$$
		Let us denote by  $N^{\pm}_{\lambda}(\Y_{w})$ the attracting and repelling parts of $T_{\lambda} \Y_{w}$. By definition, these are the subspaces $N^{\pm}_{\lambda}(\Y_{w}) \subset N^{\pm}_{\lambda}$ invariant with respect to the action of cyclic group (\ref{cycmu}) , thus 
		$$
		\Theta(N^{+}_{\lambda}(\Y_w))=\prod\limits_{{\Box\in \lambda,}\atop {h_{\lambda}(\Box) w \in \matZ } } \vartheta(a^{h_{\lambda}(\Box)} \hbar^{1/2(-a_{\lambda}(\Box)+l_{\lambda}(\Box)+1)}),
		$$
		or
		$$
		\Theta(N^{-}_{\lambda}(\Y_w))=\prod\limits_{{\Box\in \lambda,} \atop {h_{\lambda}(\Box)w \in \matZ}} \vartheta(a^{-h_{\lambda}(\Box)} \hbar^{1/2(a_{\lambda}(\Box)-l_{\lambda}(\Box)+1)}).
		$$
		Computing the limits of the theta functions using (8) in \cite{KononovSmirnov1}, we find:
		$$
		\lim\limits_{q\to 0}\,\left.\kappa^{*}\Big(\frac{\Theta(N_{\lambda}^{-})}{\Theta(N_{\lambda}^{+})}\Big)\right|_{z=z q^{w}} = \kappa^{*}\Big(\dfrac{\hat{s}(N^{-}_{\lambda}(\Y_{w}))}{\hat{s}(N^{+}_{\lambda}(\Y_{w}))}\Big)\, (-1)^{\frac{\textrm{codim}_{\X}(\Y_w)}{2}} \, \hbar^{-\frac{\textrm{codim}_{\X}(\Y_w)}{4}} \,\prod\limits_{\Box\in \lambda}\, \hbar^{- \lfloor w \cdot h_{\lambda}(\Box)\rfloor }
		$$
		Combining all these factors together and using Lemma \ref{lemgamma}:
		$$
		\gamma_{w}(a)^{-1} \fp \gamma_{-w}(a) \fp  = (-1)^{\frac{\textrm{codim}_{\X}(\Y_{w})}{2} } \hbar^{\frac{\textrm{codim}_{\X}(\Y_{w})}{4}} \prod\limits_{\Box\in \lambda} \,\hbar^{\lfloor w \cdot h_{\lambda}(\Box)   \rfloor}.
		$$
		we arrive at the statement of the theorem. 
	\end{proof}

\subsection{Gauss decomposition of  $\B^{\bullet}_{w}$}
\begin{Theorem} \label{monfac}
	{\it For all $w\in \matQ\setminus \Walls_n$ the matrices of monodromy operators in the basis of fixed points have the following Gauss decomposition:
		$$
		(-1)^n\,\bU^{+}_{w}(a,\hbar)^{-1}  \, \bD_{w}(\hbar) \, \bL^{-}_{-w}(a^{-1},\hbar)=\left\{\begin{array}{ll}
		\prod\limits_{s \in (w,0)}^{\longrightarrow} \, \B^{\bullet}_{s} \, \cdot \T, & w<0 \\
		\prod\limits_{s\in [0,w]}^{\longleftarrow} \, (\B^{\bullet}_{s})^{-1} \cdot \T, &w\geq 0
		\end{array}\right.
		$$
		where $\bD_{w}(\hbar)$ is the diagonal matrix with eigenvalues 
		\be \label{Ddef}
		\bD_{w}(\hbar)_{\lambda,\lambda}= (-1)^{n} \hbar^{-n/2} \prod\limits_{\Box\in \lambda} \,\hbar^{-\lfloor h_{\lambda}(\Box) \cdot w  \rfloor}.
		\ee
}
\end{Theorem} 
\begin{proof}
Assume $w<0$. Let $w'=w+\epsilon$ be a regular slope, i.e., $w'$ is not a wall. We apply Theorem \ref{factthm} to $w'$ which gives
$$
\B^{\bullet}_{w'}(z)\prod\limits_{s \in (w',0)}^{\longrightarrow} \, \B^{\bullet}_{s} \, \T= (-1)^n\, \bU^{l_1}_{w'}(a,\hbar)^{-1}  \bU^{l_1}_{\Y_{w'}}(a,z,\hbar)   \bL^{l_2}_{\Y_{w'}}(a,z,\hbar)^{-1}  \bL^{l_2}_{-w'}(a^{-1},\hbar) 
$$
As $w'$ is not a wall $\B^{\bullet}_{w'}(z)=1$ and
$$
\bU^{\pm }_{\Y_{w'}}(a,z,\hbar)  =\gamma_{w'}(a), \ \ \ \bL^{\pm}_{\Y_{w'}}(a,z,\hbar) = \fp \, \gamma_{-w'}(a)\, \fp.  
$$
and by Lemma \ref{lemgamma}:
$$
\gamma_{w'}(a) \fp \, \gamma_{-w'}(a)^{-1}\, \fp=(-1)^{\frac{\textrm{codim}_{\X}(\Y_{w'})}{2} } \hbar^{-\frac{\textrm{codim}_{\X}(\Y_{w'})}{4}} \prod\limits_{\Box\in \lambda} \,\hbar^{-\lfloor h_{\lambda}(\Box) \cdot w'  \rfloor}.
$$
Again, as $w'$ is not a wall, $\Y_{w'}$ is zero-dimensional, see Remark \ref{rempoints}. Thus
$$
\textrm{codim}_{\X}(\Y_{w'})=\dim(\X)=2 n.
$$
Finally, by our choice of $w'$ we have:
$$
\bU^{\pm }_{w'}(a,\hbar) = \bU^{+}_{w}(a,\hbar), \ \ \
\bL^{\pm}_{-w'}(a,\hbar)=\bL^{-}_{-w}(a,\hbar), \ \ \lfloor h_{\lambda}(\Box) \cdot w'  \rfloor=\lfloor h_{\lambda}(\Box) \cdot w  \rfloor.
$$
Combining all these factors together, we arrive at the statement of the theorem. 
For $w>0$ the proof is the same.
\end{proof}

\section{Wall-crossing operators as R-matrices}
The goal of this section is to relate the wall-crossing operators $\B_w(z)$ and the monodromy operators $\B_w$ to the K-theoretic $R$-matrices of the Hilbert scheme $\X$ and the cyclic quiver varieties  $\Y_w \subset \X^{!}$. In the next subsection we briefly recall the definition of the K-theoretic $R$-matrices and review their basic properties. For more more systematic introduction we refer to Section 2 of \cite{OS}, Section 9 of \cite{Okpcmi} or Negut's thesis \cite{NegutThesis}.  

\subsection{K-theoretic $R$-matrices}

We recall that for generic choices of a cocharacter $\sigma \in \cochar(\bA)$ and a slope $\sl \in H^{2}(\X,\matQ)$  the K-theoretic stable envelopes $\Stab^{\X,Kth,[\sl]}_{\sigma}(\lambda)$ of torus fixed points $\lambda \in \X^{\bT}$  provide bases of the localized K-theory of $\X$. The stable envelopes only change when $\sigma$ crosses certain hyperplanes in $\Lie_{\matQ}(\bA)$ or $\sl$ crosses certain hyperplanes in $ H^{2}(\X,\matQ)$. The K-theoretic $R$-matrices  are defined  as the corresponding transition matrices between the stable bases. 

\begin{Definition}
	{\it The total $K$-theoretic $R$ - matrix of a variety $X$ with slope $\sl \in H^{2}(X,\matQ)$ is the transition matrix from the stable basis
		$\Stab^{X,Kth,[\sl]}_{-\sigma}$ to the stable basis $\Stab^{X,Kth,[\sl]}_{\sigma}$ }
\end{Definition}
Let $R^{\pm}_{\Y_{w}}$ denote the total $R$-matrix of the cyclic quiver variety $\Y_w$ with small ample or anti-ample slopes $\pm \epsilon \in H^{2}(\Y_w,\matQ)$.
Using our notations (\ref{lufpy}) we find
\be \label{RmatForY}
R^{\pm}_{\Y_w}(a)= {{U}}^{\pm}_{\Y_{w}}(a) {{L}}^{\pm}_{\Y_{w}}(a)^{-1}. 
\ee
Note that this formula provides a Gauss decomposition of $R^{\pm}_{\Y_w}(a)$. 

\begin{Definition}
	{\it The wall $R$ - matrix of a variety $X$ is the transition matrix from the stable basis
		$\Stab^{X,Kth,[\sl]}_{\sigma}$ to the stable basis $\Stab^{X,Kth,[\sl']}_{\sigma}$ where 
		$\sl$ and $\sl'$ are two slopes separated by a single wall $w$ such that $\sl'-\sl$ is ample.  
	}
\end{Definition}
\noindent
For instance, the wall $R$-matrices of the Hilbert scheme $\X$ has the form:
$$
{\bf R}^{\X,+\sigma}_{wall, \, w}=\bU^{-}_{w}(a) \, \bU^{+}_{w}(a)^{-1}, \ \ \ 
{\bf R}^{\X,-\sigma}_{wall, \, w}=\bL^{-}_{w}(a) \, \bL^{+}_{w}(a)^{-1}.
$$
We will also need the twisted version of K-theoretic $R$-matrix of $\Y_{w}$:
\be \label{twistedR}
{\bf R}^{\pm}_{\Y_w}(z)=   \gamma_w(a) \kappa^{*}(R^{\pm}_{\Y_w}(a))\, \gamma_{w}(a)^{-1}
\ee
Explicit examples of $R^{\pm}_{\Y_w}(a)$ and ${\bf R}^{\pm}_{\Y_w}(z)$ for $n=2$ and $n=3$ can be found in Appendix \ref{twisrsect}. Here we list main properties of ${\bf R}^{\pm}_{\Y_w}(z)$:
\begin{Proposition} \label{monprop}
	{\it For two partitions $\lambda,\mu$ and $w\in \matQ$ let
		$$
		d_{\lambda,\mu}(w)=w (\sum_{\Box \in \lambda} c_{\lambda}(\Box)- \sum_{\Box \in \mu} c_{\mu}(\Box))
		$$
		then the matrix coefficients of ${\bf R}^{\pm}_{\Y_w}(z)$ are
		\begin{itemize}
			\item ${\bf R}^{\pm}_{\Y_w}(z)_{\lambda,\mu} \in \matQ(z,\hbar,a)$,
			\item if $d_{\lambda,\mu}(w)\not \in \matZ$ then ${\bf R}^{\pm}_{\Y_w}(z)_{\lambda,\mu}=0$,
			
			\item if $d_{\lambda,\mu}(w)\in \matZ$ then ${\bf R}^{\pm}_{\Y_w}(z)_{\lambda,\mu} = a^{d_{\lambda,\mu}(w)} c_{\lambda,\mu}$  for $c_{\lambda,\mu} \in \matQ(z,\hbar)$.
	\end{itemize}}
\end{Proposition}
\begin{proof}
	It is obvious from the definition of $\gamma_{w}(a)$ that
	$$
	{\bf R}^{\pm}_{\Y_w}(z)_{\lambda,\mu} \sim a^{d_{\lambda,\mu}(w)}
	$$
	The vanishing of matrix elements for non-integral $d_{\lambda,\mu}(w)$ follows immediately from Theorem 3 of \cite{KononovSmirnov2}. 
\end{proof}

Here is another  elementary property of the total K-theoretic  $R$-matrices:\begin{Proposition} \label{unitlem}
	{\it Assume  $\X$ is a symplectic variety for which the $K$-theoretic stable envelope exist, then  the total $K$-theoretic R-matrices of $\X$ have the following unitary property
		$$
		{ R}^{[\sl]}_{\X}(a)^{-1}={ R}^{[-\sl]}_{\X}(a^{-1}),
		$$
		in particular for small slopes $\sl=\pm \epsilon$ we have
		$$
		{ R}^{\pm}_{\X}(a)^{-1}={ R}^{\mp}_{\X}(a^{-1}).
		$$}
\end{Proposition}
\begin{proof}
	The proof follows the logic of Sections 4.5.1 - 4.5.3 in \cite{OkMaul}.
\end{proof}
\noindent For $R$-matrices $R^{\pm}_{\Y_w}(a)$ this proposition gives (\ref{unitarY}). For $n=2,3$ this can be checked explicitly using matrices from Appendix \ref{twisrsect}.

We list main properties of the wall $R$-matrices of $\X$:
\begin{Proposition}
	{\it The matrix elements of the wall R-matrices for the Hilbert scheme $\X$ have the following properties:
		\begin{itemize}
			\item $({\bf R}^{\X,\pm \sigma}_{wall, \, w})_{\lambda,\mu} \in \matQ[a^{\pm 1},\hbar^{\pm 1}]$,
			\item $({\bf R}^{\X,+\sigma}_{wall, \, w})_{\lambda,\mu}= ({\bf R}^{\X,-\sigma}_{wall, \, w})_{\mu,\lambda} =0$ if $\lambda\succ \mu$, i.e. the wall $R$-matrices are upper and lower triangular of the fixed points ordered by $\succ$.
			\item $({\bf R}^{\X,\pm \sigma}_{wall, \, w})_{\lambda,\lambda}=1$,
			\item if $d_{\lambda,\mu}(w)\not \in \matZ$ then $({\bf R}^{\X,\pm \sigma}_{wall, \, w})_{\lambda,\mu}=0$,
			\item if $d_{\lambda,\mu}(w) \in \matZ$ then $({\bf R}^{\X,\pm \sigma}_{wall, \, w})_{\lambda,\mu} \sim a^{\pm d_{\lambda,\mu}(w)}$.
	\end{itemize}}
\end{Proposition}
\begin{proof}
	All properties follow immediately from Theorem 1 of \cite{OS}.
\end{proof}

\subsection{Wall crossing operators $\B_{w}(z)$ as $R$-matrices of $\Y_w$} 
Let $\overset{\curvearrowright}{\B}_{w}(z)=\{\overset{\curvearrowright}{\B}_{w}(z)_{\lambda,\mu} : \lambda,\mu \in \X^{\bT} \}$ denote the matrix of the operator $\B_{w}(z)$ in the {\it mixed stable basis}: the input is the stable basis before a wall $w$, 
\be 
\label{stbef}
s^{ w-\epsilon}_{\lambda}:=\Stab^{\X,Kth,[w-\epsilon]}_{\sigma}(\lambda)
\ee
and the output in the stable basis after $w$: 
\be \label{staft}
s^{w+\epsilon}_{\lambda}:=\Stab^{\X,Kth,[w+\epsilon]}_{ \sigma}(\lambda).
\ee
Explicitly, we have
\be \label{mixedbasdef}
\B_{w}(z)(  s^{ w-\epsilon}_{\lambda} )  = \sum\limits_{\mu,\lambda} \, \overset{\curvearrowright}{\B}_{w}(z)_{\lambda,\mu} \, s^{ w+\epsilon}_{\lambda}. 
\ee
As usual, $0<\epsilon \ll 1$ denotes an infinitesimally small ample slope. We denote $\overset{\curvearrowright}{\B}_{w}:= \overset{\curvearrowright}{\B}_{w}(\infty)$, the matrix of the monodromy operators $\B_w$ in these bases.

\begin{Theorem} \label{brmatth}
	{\it The matrices $\overset{\curvearrowright}{\B}_{w}(z)$ coincide with the $K$-theoretic $R$-matrices of $\Y_{w}$:
		$$\overset{\curvearrowright}{\B}_{w}(z)=\hbar^{\Omega_{w}} \, {\bf R}^{-}_{\Y_{w}}(z), \ \ \ \overset{\curvearrowright}{\B}_{w}= \hbar^{\Omega_{w}} \, {\bf R}^{-}_{\Y_{w}}(\infty),$$
		where
		\be \label{homegaop}
		\hbar^{\Omega_{w}} := (-\hbar^{1/2})^{n-\frac{\mathrm{codim}_{\X}(\Y_w)}{2}}. 
		\ee
	}
\end{Theorem}
\begin{proof}
	Assume $w<0$ then by Theorem \ref{factthm} for $l_1=l_2=-$ we have
	\be \label{eq1}
	\B^{\bullet}_{w}(z)\prod\limits_{s \in (w,0)}^{\longrightarrow} \, \B^{\bullet}_{s} \, \T= (-1)^n\, \bU^{-}_{w}(a,\hbar)^{-1}  \bU^{-}_{\Y_{w}}(a,z,\hbar)   \bL^{-}_{\Y_{w}}(a,z,\hbar)^{-1}  \bL^{-}_{-w}(a^{-1},\hbar) 
	\ee
	By Theorem  \ref{monfac} we have
	\be \label{eq2}
	\prod\limits_{s \in (w,0)}^{\longrightarrow} \, \B^{\bullet}_{s} \, \T= (-1)^n\, \bU^{+}_{w}(a,\hbar)^{-1} \, \bD_{w}(a,\hbar)\,  \bL^{-}_{-w}(a^{-1},\hbar)
	\ee
	All operators here are invertible, thus dividing (\ref{eq1}) by (\ref{eq2}) from the right we obtain:
	\be \label{bldu}
	\B^{\bullet}_{w}(z)=\bU^{-}_{w}(a,\hbar)^{-1} \bU^{-}_{\Y_{w}}(a,z,\hbar)   \bL^{-}_{\Y_{w}}(a,z,\hbar)^{-1}\,\bD_{w}(a,\hbar)^{-1}\,\bU^{+}_{w}(a,\hbar)
	\ee
	We have:
	$$\bU^{-}_{\Y_{w}}(a,z,\hbar)   \bL^{-}_{\Y_{w}}(a,z,\hbar)^{-1}=\gamma_{w}(a)\, \kappa^{*} ({{U}}^{-}_{\Y_{w}}(a) \,{{L}}^{-}_{\Y_{w}}(a)^{-1}) 
	\fp \, \gamma_{-w}(a)^{-1}\, \fp\,
	$$
	and by Lemma \ref{lemgamma}
	$$
	\fp \, \gamma_{-w}(a)^{-1}\, \fp = \gamma_{w}(a)^{-1} (-\hbar^{1/2})^{-\frac{\mathrm{codim}_{\X}(\Y_w)}{2}}  \prod\limits_{\Box\in \lambda} \, \hbar^{-\lfloor h_{\lambda}(\Box) \cdot w \rfloor}
	$$
	which together with (\ref{Ddef}) gives
	$$
	\B^{\bullet}_{w}(z)=\bU^{-}_{w}(a,\hbar)^{-1} \, \hbar^{\Omega} \,{\bf R}^{-}_{\Y_{w}}(z) \,\bU^{+}_{w}(a,\hbar)
	$$
	or, equivalently 
	$$
	\hbar^{\Omega} \, {\bf R}^{-}_{\Y_{w}}(z)=\bU^{-}_{w}(a,\hbar) \B^{\bullet}_{w}(z) \bU^{+}_{w}(a,\hbar)^{-1}.
	$$
	By definition, $\B^{\bullet}_{w}(z)$ is the matrix of the operator  $\B_{w}(z)$ in the basis of fixed points and $\bU^{\pm}_{w}(a,\hbar)$ are the transition matrices from the basis of torus fixed points to stable bases $s^{w\pm \epsilon}_{\lambda}$. Thus, the last equality means that 
	\be\label{statmt}
	\overset{\curvearrowright}{\B}_{w}(z)=\hbar^{\Omega} \,{\bf R}^{-}_{\Y_{w}}(z).
	\ee
	For $w\geq0$ from Theorem \ref{factthm} with $l_1=l_2=-$ we obtain:
	$$
	\Big(\B^{\bullet}_{w}(z^{-1})^{-1} \prod\limits_{s\in [0,w)}^{\longleftarrow} \, (\B^{\bullet}_{s})^{-1} \Big) \cdot \T=
	\bU^{-}_{w}(a,\hbar)^{-1}  \bU^{-}_{\Y_{w}}(a,z,\hbar)   \bL^{-}_{\Y_{w}}(a,z,\hbar)^{-1}  \bL^{-}_{-w}(a^{-1},\hbar) 
	$$
	By Theorem  \ref{monfac} we have
	$$
	\prod\limits_{s\in [0,w]}^{\longleftarrow} \, (\B^{\bullet}_{s})^{-1} \cdot \T=\bU^{+}_{w}(a,\hbar)^{-1} \, \bD_{w}(a,\hbar)\,  \bL^{-}_{-w}(a^{-1},\hbar)
	$$
	dividing first by the second we obtain
	$$
	\B^{\bullet}_{w}(z^{-1})^{-1} \B^{\bullet}_{w}=\bU^{-}_{w}(a,\hbar)^{-1}  \bU^{-}_{\Y_{w}}(a,z,\hbar)   \bL^{-}_{\Y_{w}}(a,z,\hbar)^{-1}\,\bD_{w}(a,\hbar)^{-1}\,\bU^{+}_{w}(a,\hbar)
	$$
	Finally, Proposition \ref{invz} gives $\B^{\bullet}_{w}(z^{-1})^{-1} \B^{\bullet}_{w}=\B^{\bullet}_{w}(z)$ and we arrive at the same identity (\ref{bldu}) as in $w<0$ case.
	
	Substituting $z=\infty$ in (\ref{statmt}) we obtain 
	$
	\overset{\curvearrowright}{\B}_{w}={\bf R}^{-}_{\Y_{w}}(\infty).
	$
\end{proof}

\begin{Remark} \label{oppsigrem}
	Instead of (\ref{stbef})-(\ref{staft}) we could consider the stable bases for opposite chamber $-\sigma$:
	$$
	s^{w\pm \epsilon}_{\lambda}=\Stab^{\X,Kth,[w\pm \epsilon]}_{-\sigma}
	$$
	and define the matrices $\overset{\curvearrowright}{\B}_{w}(z)$ by (\ref{mixedbasdef}). Then Theorem \ref{brmatth} would take the form
	$$
	\overset{\curvearrowright}{\B}_{w}(z)=\hbar^{\Omega_{w}} \, {\bf R}^{+}_{\Y_{w}}(z), \ \ \ \overset{\curvearrowright}{\B}_{w}= \hbar^{\Omega_{w}} \, {\bf R}^{+}_{\Y_{w}}(\infty).
	$$
\end{Remark}

\subsection{Monodromy operators $\B_w$ as wall $R$-matrices} 
\begin{Theorem} \label{thmrwallr}
	{\it There are following matrix identities:
		\be \label{rmatid}
		{\bf R}_{\Y_w}^{-}(0)=\hbar^{-\Omega_{w}}\,{\bf R}^{\X,+\sigma}_{wall, \, w}, \ \ \ 
		{\bf R}_{\Y_w}^{+}(0)=\hbar^{-\Omega_{w}}\,{\bf R}^{\X,-\sigma}_{wall, \, w}.
		\ee}
\end{Theorem}
\begin{proof}
	By Theorem \ref{brmatth} we have 
	$$
	\overset{\curvearrowright}{\B}_{w}(z)=\hbar^{\Omega_{w}} \, {\bf R}^{-}_{\Y_{w}}(z)
	$$
	which means 
	$$
	\hbar^{\Omega_{w}} \, {\bf R}^{-}_{\Y_{w}}(z)=\bU^{-}_{w}(a,\hbar) \B^{\bullet}_{w}(z) \bU^{+}_{w}(a,\hbar)^{-1}
	$$
	Evaluating this at $z=0$ and noting that 
	$\B^{\bullet}_{w}(0)=1$ we obtain:
	$$
	{\bf R}^{-}_{\Y_{w}}(z)=\hbar^{-\Omega_{w}} \bU^{-}_{w}(a,\hbar)  \bU^{+}_{w}(a,\hbar)^{-1}=\hbar^{-\Omega_{w}}\,{\bf R}^{\X,+\sigma}_{wall, \, w}.
	$$
	Changing $\sigma \to -\sigma$ we obtain 
	$$
	\overset{\curvearrowright}{\B}_{w}(z)=\hbar^{\Omega_{w}} \, {\bf R}^{+}_{\Y_{w}}(z)
	$$
	see Remark \ref{oppsigrem} above.  At $z=0$ this gives the second identity. 
\end{proof}
\begin{Remark}
	Identities (\ref{rmatid}) are equivalent to Theorem 8 in \cite{KononovSmirnov2}. They relate the wall $R$-matrices  of $\X$ with slopes $\pm \epsilon$ to the $R$-matrices of $\Y_{w}$.
\end{Remark}

Let us consider $w=\frac{p}{m} \in \matQ$, with $gcd(p,m)=1$. We note that the variety $\Y_w$ is defined as the fixed subset of the finite group (\ref{cycmu}). In particular, it only depends on the denominator $m$ of $w$. Thus, the operators $\hbar^{\Omega_{w}}$ and $R^{\pm}_{\Y_w}(a)$ depend only on $m$ as well. The Theorem \ref{thmrwallr} and definition of twisted $R$-matrix (\ref{twistedR}) then provides the following result about the wall R-matrices for the Hilbert scheme $\X$: 
\begin{Corollary}
	{\it Let  $w=\frac{p}{m}, w'=\frac{p'}{m}$, with $gcd(p,m)=gcd(p',m)=1$ then
		\be \label{hdepr}
		\gamma_w(a)^{-1} \, {\bf R}^{\X,\pm \sigma}_{wall, \, w} \,\gamma_{w}(a)= \gamma_{w'}(a)^{-1} \, {\bf R}^{\X,\pm \sigma}_{wall, \, w'} \,\gamma_{w^{'}}(a)
		\ee
		and the matrix elements of (\ref{hdepr}) are elements of $\matQ[\hbar^{\pm 1}]$. }
\end{Corollary}

\subsection{QDE as a non-minuscule qKZ equation \label{qdesec}}
Let $s^{-\epsilon}_{\lambda}$ be the stable basis  of $K_{\bT}(\X)$ with small anti-ample slope $-\epsilon$, i.e., the basis (\ref{stbef}) for $w=0$.  Let $\mathbf{E}_{0}$ be the diagonal matrix with eigenvalues given by $\left.\mathscr{O}_{\X}(1)\right|_{\lambda}\in K_{\bT}(pt)$, as defined in Section \ref{monsolsec}.

\begin{Theorem}
	{\it In the stable basis $s^{-\epsilon}_{\lambda}$ the matrix of the operator $\M_{\cL}(z)$ (defined by (\ref{Mdef})) takes the form
		\be \label{rmprod}
		\mathbf{E}_{0} \prod\limits^{\longrightarrow}_{w\in [-1,0)}\,
		\hbar^{\Omega_{w}}\,{\bf R}_{\Y_w}^{-}(z q^{-w}).
		\ee}
\end{Theorem}
\begin{proof}
	Theorem \ref{brmatth} gives 
	$$
	\hbar^{\Omega_{w}} \, {\bf R}^{-}_{\Y_{w}}(z)=\bU^{-}_{w}(a,\hbar) \B^{\bullet}_{w}(z) \bU^{+}_{w}(a,\hbar)^{-1}.
	$$
	By (\ref{Bdef}) we have $\B^{\bullet}_{w}(z q^{-w})=\B^{\bullet}_{w}(z,q)$ and thus 
	$$
	\hbar^{\Omega_{w}} \, {\bf R}^{-}_{\Y_{w}}(z q^{-w})=\bU^{-}_{w}(a,\hbar) \B^{\bullet}_{w}(z,q) \bU^{+}_{w}(a,\hbar)^{-1}.
	$$
	Assume  the product in (\ref{rmprod}) is over the walls
	$w_1,w_2,\dots,w_m$ with $w_i<w_{i+1}$. By definition
	$$
	\bU^{+}_{w_{i}}(a,\hbar)=\bU^{-}_{w_{i+1}}(a,\hbar), \ \ \bU^{+}_{w_{m}}(a,\hbar)=\bU^{-}_{0}(a,\hbar)
	$$
	Thus we compute
	$$
	\mathbf{E}_{0} \prod\limits^{\longrightarrow}_{w\in [-1,0)}\,
	\hbar^{\Omega_{w}}\,{\bf R}_{\Y_w}^{-}(z q^{-w})=\mathbf{E}_{0}  \bU^{-}_{-1}(a,\hbar) \Big( \prod\limits^{\longrightarrow}_{w\in [-1,0)}\,
	\B^{\bullet}_w(z,q) \Big) \bU^{-}_{0}(a,\hbar)^{-1}
	$$
	Twists by line bundles change the slope of the stable bases by integral shifts, in particular $$\mathbf{E}_{0} \,\bU^{-}_{-1}(a)^{-1}=\bU^{-}_{0}(a)^{-1} \mathbf{E}_{0}$$
	We obtain:
	$$
	\mathbf{E}_{0} \prod\limits^{\longrightarrow}_{w\in [-1,0)}\,
	\hbar^{\Omega_{w}}\,{\bf R}_{\Y_w}^{-}(z q^{-w})=  \bU^{-}_{0}(a,\hbar) \,\M^{\bullet}_{\cL}(z) \, \bU^{-}_{0}(a,\hbar)^{-1}
	$$
	where
	$$
	\M^{\bullet}_{\cL}(z)=\mathbf{E}_{0} \Big( \prod\limits^{\longrightarrow}_{w\in [-1,0)}\,
	\B^{\bullet}_w(z,q) \Big)
	$$
	denotes the matrix of the operator $\M_{\cL}(z)$ in the basis of fixed points. 
	By definition,  $\bU^{-}_{0}(a)$ is the transition matrix between the stable basis $s^{-\epsilon}_{\lambda}$ and the basis of fixed points. The theorem follows. 
\end{proof}
By above theorem, the quantum difference equation (\ref{qdiffe}) for the Hilbert scheme $\X$, in the stable basis  $s^{-\epsilon}_{\lambda}$ takes the form
\be \label{qKZ}
\Psi(z q) =\mathbf{E}_{0}\,  \Big(\prod\limits^{\longrightarrow}_{w\in [-1,0)}\, \hbar^{\Omega_w} \, {\bf R}_{\Y_w}^{-}(z q^{-w}) \Big) \Psi(z)
\ee
where $\mathbf{E}_{0}$ is a diagonal matrix with eigenvalues given by monomials in $a$ and $\hbar$.  The $q$-difference equations of this type, which include the products of ``trigonometric'' (i.e. K-theoretic) R-matrices shifted by powers of $q$ appear in mathematical physics as the {\it quantum Knizhnik-Zamolodchikov (qKZ) equations} \cite{qkz}. The equation (\ref{qKZ}) is the proper version of the qKZ equation for the {\it toroidal algebra} $\gt$ associated with the Hilbert scheme $\X$.

The  difference between (\ref{qKZ}) and the standard qKZ equations is that ${\bf R}_{\Y_w}^{-}(z)$ for different walls $w$
are allowed to be the trigonometric $R$-matrices of {\it different quantum groups}. Indeed, the Nakajima varieties $\Y_{w}$ for different walls $w=p/N \in [-1,0)$  corresponds to cyclic quivers which may have different length $N\leq n$.  The matrices ${\bf R}_{\Y_w}^{-}(z)$ appearing in (\ref{qKZ}) correspond to the trigonometric $R$-matrices of the {\it quantum toroidal algebras}~$\gtn$ with $N$ varying in the interval $1 < N\leq n$.



\section{Monodromy of quantum connection \label{mondsec}}

Let $\psi_{0}(z)$, $\psi_{\infty}(z)$ be the fundamental solution matrices of the qde in Section \ref{qdesection}. We fix a branch of a factor $z^{c^{(0)}}$ by cutting  the Riemann sphere along $\matR_{+}$ - the positive part of real axis connecting $z=0$ and $z=\infty$. With this cut, $\psi_{0}(z)$ and $\psi_{\infty}(z)$ become single valued functions in their domains.

The solutions $\psi_{0}(z)$ and $\psi_{\infty}(z)$ can be obtained as {\it cohomological limits} of the corresponding solutions in K-theory. Explicitly, let $\Psi_{0}(t_1,t_2,q,z)$, $\Psi_{\infty}(t_1,t_2,q,z)$ be the solutions of the K-theoretic QDE described in Section~\ref{monsolsec}.  Then
$$
\lim\limits_{\tau\to 0} \Psi_{0,\infty}(e^{2 \pi i \epsilon_1 \tau},e^{2 \pi i \epsilon_2 \tau},e^{-2 \pi i \tau},z)= \psi_{0,\infty}(z)
$$
It follows that the transport of the qde:
\be \label{transpdef}
\textrm{Tran}(s): =  \psi_{0}(e^{2 \pi i s})^{-1} \psi_{\infty}(e^{2 \pi i s})
\ee
is a limit the monodromy (\ref{mondef}):
\be \label{transp}
\textrm{Tran}(s) =\lim_{\tau \to 0} \, {\Mon}(z = e^{2 \pi i s},t_1=e^{2 \pi i \epsilon_1 \tau}, t_2=e^{2 \pi i \epsilon_2 \tau}, q=e^{-2 \pi i \tau})
\ee
The next two subsections are devoted to the computation of this limit.   


\subsection{Transport of $\nabla_{\mathsf{DT}}$} 
Let us consider the transport of the fundamental solutions  (\ref{transpdef0}):
$$
\textrm{Tran}_{\mathsf{DT}}(s) =  \psi^{0}_{{\mathsf{DT}}}(e^{2 \pi i s})^{-1} \psi^{\infty}_{{\mathsf{DT}}}(e^{2 \pi i s}).
$$
This transport matrix is related to (\ref{transpdef}) via conjugation by $\Gamma_{\mathsf{DT}}$. 
\begin{Proposition}{\it
		\be \label{transpDT}
		\mathrm{Tran}_{\mathsf{DT}}(s)= \lim_{\tau\to 0} \widetilde{\Mon}(e^{2 \pi i s},e^{2 \pi i \epsilon_1 \tau}, e^{2 \pi i \epsilon_2 \tau}, e^{-2 \pi i \tau}) 
		\ee
		where 
		\be \label{rightmat}
		\widetilde{\Mon}(z,t_1,t_2,q):=e^{\frac{\ln(\mathsf{E}_{0}) \ln(z)}{\ln(q)}} \, {\bf{U}}(a,z)^{-1} \, {\bf{L}}(a^{-1},z^{-1} \hbar^{-1}) e^{-\frac{\ln(\mathsf{E}_{\infty}) \ln(z)}{\ln(q)}}.
		\ee}
\end{Proposition}

\begin{proof}
	By Corollary \ref{corrmonodgam} and (\ref{monregmon}) we have
	$$
	\Mon(z,t_1,t_2,q)=  \Phi( q\, T\X)\,  \widetilde{\Mon}(z,t_1,t_2,q)\, \Phi( q\, T\X)^{-1} 
	$$
	where $\widetilde{\Mon}(z,t_1,t_2,q)$ is given by (\ref{rightmat}). The conjugation by $\Phi( q\, T\X)$ multiplies  the $(\lambda,\mu)$-matrix element of a matrix by
	$$
	F(t_1,t_2,q)=\frac{\Phi(q T_{\lambda} \X)}{\Phi( q T_{\mu}\X) }. 
	$$
	$\X$ is a symplectic variety, with the $\bT$-weight of the symplectic form given by $\hbar$. Thus, for a $\bT$-weight $a$  of the tangent space $T_{\lambda} \X$ there is the symplectic dual $\bT$-weight of $T_{\lambda} \X$ given by $a^{-1} \hbar^{-1}$. This means that the above function is of the form
	$$
	F(t_1,t_2,q)=\prod \, \dfrac{\varphi(a q) \varphi ( q a^{-1} \hbar^{-1})}{\varphi(q b) \varphi (q b^{-1} \hbar^{-1})}
	$$
	for some weights $a$ and $b$. By Lemma \ref{gammalimlemma} below we have 
	$$
	\lim_{\tau \to 0} F(e^{2 \pi i \tau \epsilon_1},e^{2 \pi i \tau \epsilon_2},  e^{2 \pi i \tau} ) = \dfrac{\Gamma_{\lambda}(\epsilon_1,\epsilon_2)}{\Gamma_{\mu}(\epsilon_1,\epsilon_2)}
	$$
	where 
	\be \label{gamdt}
	\Gamma_{\lambda}(\epsilon_1,\epsilon_2)= \prod\limits_{\mathsf{w} \in \mathrm{char}_{\bT}(T_{\lambda}\X)} \dfrac{1}{\Gamma(\textsf{w}+1)}.
	\ee
	which are exactly the eigenvalues of the diagonal matrix $\Gamma_{\mathsf{DT}}$. 
	We conclude that the limits in (\ref{transp}) and in (\ref{transpDT}) are related via the conjugation by $\Gamma_{\mathsf{DT}}$. The proposition follows. 
\end{proof}


\begin{Lemma} \label{gammalimlemma}
	{\it 
		Let us consider a function of the form
		$$
		f(a,b,\hbar,q)=\dfrac{\varphi(a q) \varphi ( q a^{-1} \hbar^{-1})}{\varphi(q b) \varphi (q b^{-1} \hbar^{-1})},
		$$
		then, for generic values of $\alpha,\beta,h$ we have
		\be \label{limauxil}
		\lim\limits_{\tau \to 0} f(e^{2\pi i \tau \alpha},e^{2\pi i \tau \beta},e^{2\pi i \tau h},e^{2\pi i \tau}) = \dfrac{\Gamma(\beta +1)\Gamma(-\beta -h +1)}{\Gamma(\alpha +1)\Gamma(-\alpha -h +1)}.
		\ee}
\end{Lemma}
\begin{proof}
	Explicitly, we have:
	$$
	f(a,b,\hbar,q)=\prod\limits_{n=0}^{\infty}\dfrac{(1-a q^{n+1}) (1-a^{-1} \hbar^{-1} q^{n+1})}{(1-b q^{n+1}) (1-b^{-1} \hbar^{-1} q^{n+1})},
	$$
	and formally computing the limit (\ref{limauxil}) we obtain
	\be \label{prodexis}
	\lim\limits_{\tau \to 0} f(e^{2\pi i \tau \alpha},e^{2\pi i \tau \beta},e^{2\pi i \tau h},e^{2\pi i \tau}) = \prod\limits_{n=0}^{\infty} \dfrac{(\alpha +1+n)(-\alpha -h+1+n) }{(\beta +1+n)(-\beta -h+1+n)},
	\ee
	were we assume that the last infinite product exists. To show the convergence of the product, we recall the Weierstrass product expansion of the Gamma function:
	$$
	\Gamma(x)=\dfrac{e^{-\gamma x}}{x} \prod\limits_{n=1}^{\infty} \Big(1+\dfrac{x}{n}\Big)^{-1} e^{x/n},
	$$
	which holds for all complex $x$ except non-positive integers. Using this formula, after elementary manipulations, for generic $\alpha, \beta$ and $h$ we obtain
	$$
	\dfrac{\Gamma(\beta+1)\Gamma(-\beta-h+1)}{\Gamma(\alpha+1)\Gamma(-\alpha-h+1)}=\prod\limits_{n=0}^{\infty} \dfrac{(\alpha +1+n)(-\alpha -h+1+n) }{(\beta +1+n)(-\beta -h+1+n)}.
	$$
	This proves that the product (\ref{prodexis}) exists for generic values of the parameters and also proves proves (\ref{limauxil}).
\end{proof}

\subsection{Cohomological limits of balanced functions} 
Let us denote
$$
z=e^{2 \pi i \xi}, \ \  q=e^{2 \pi i \tau}, \ \ \hbar = e^{2 \pi i h},
$$
and consider the Jacobi theta function
$$
\vartheta(\xi,\tau)=\sum\limits_{n\in \matZ}\, e^{\pi i n^2 \tau +2 \pi i n \xi} =\sum\limits_{n\in \matZ} \, q^{n^2/2} z^{n}.
$$
This function is related to the $q$-theta function (\ref{thetafun})
via 
\be \label{relthe}
\vartheta\Big(\xi -\frac{\tau}{2}+\frac{1}{2},\tau \Big)=-z^{1/2} \varphi(q) \theta(z).
\ee
The modular transformation of $\vartheta(\xi,\tau)$ is described by the formula
\be \label{modtr}
\vartheta(\xi,\tau)=\frac{i}{\sqrt{\tau}}\, e^{-\frac{\pi i \xi^{2}}{\tau}} \vartheta\Big(\frac{\xi}{\tau},-\frac{1}{\tau}\Big).
\ee

\begin{Lemma} \label{limlemm}
	{\it Let us consider the function
		$$
		\tilde{f}(z,a,b,q)=e^{m \frac{\ln(z) \ln(a/b)}{\ln(q)}} {f}(z,a,b,q), \ \ \ {f}(z,a,b,q)=\frac{\theta(z^m a )}{\theta(z^m b )}
		$$
		with $m\in \matZ$.  Then, for generic $s\in \matR$, the limit
		$$
		\lim\limits_{\tau \to 0} \,  \tilde{f}(e^{2 \pi i s},e^{2 \pi i \epsilon_1 \tau}, e^{2 \pi i \epsilon_2 \tau},e^{-2 \pi i \tau}) 
		$$
		as $\tau$ approaches $0$  along the positive part of the imaginary axis, exists and is equal to the limit
		$$
		\lim_{q\to 0} f(q^{s},e^{2 \pi i \epsilon_1},e^{2 \pi i \epsilon_2 },q).
		$$
	}
\end{Lemma}
\begin{proof}
	For simplicity, let us consider a function of the form
	$$
	\tilde{f}(z,\hbar,q)=e^{\frac{\ln(z)\ln(\hbar)}{\ln(q)}} \dfrac{\theta(z \hbar)}{\theta(z)}
	$$
	i.e., we consider (\ref{monfact}) $m=1$, $a=\hbar$, $b=1$. For the general values  the calculation is the same. Then
	$$
	\tilde{f}(e^{2 \pi i s}, e^{2 \pi i h \tau},e^{-2 \pi i \tau})= e^{-2 \pi i s h}   \dfrac{\theta(e^{2\pi i (s - h \tau) }) }{\theta(e^{2 \pi i s} )}.   
	$$
	As $s$ is generic, we may assume that $s\not\in \matZ$, and $\theta(e^{2 \pi i s} )\neq 0$.
	
	Applying the modular transform using (\ref{relthe}) and (\ref{modtr}) we find
	$$
	\dfrac{\theta(e^{2\pi i (s-h\tau)} ) }{\theta(e^{2 \pi i s} )} = e^{2\pi i h (s+1/2)}  \frac{\vartheta\Big(\frac{s}{\tau}-h-\frac{1}{2}+\frac{1}{2\tau},-\frac{1}{\tau}\Big)}{\vartheta\Big(\frac{s}{\tau}-\frac{1}{2} +\frac{1}{2\tau},-\frac{1}{\tau}\Big)}.
	$$
	Let us denote new modular parameter by  $q=e^{-2 \pi i /\tau}$. If $\tau \to 0$ along the positive part of the imaginary axis then $q\to 0$. Thus
	$$
	\lim\limits_{\tau \to 0} \frac{\vartheta\Big(\frac{s}{\tau}-h-\frac{1}{2}+\frac{1}{2\tau},-\frac{1}{\tau}\Big)}{\vartheta\Big(\frac{s}{\tau}-\frac{1}{2} +\frac{1}{2\tau},-\frac{1}{\tau}\Big)}= e^{- \pi i h} \lim\limits_{q\to 0} \frac{\theta( \hbar   q^{s} )}{ \theta( q^{s})}.
	$$
	Combining all these terms together we find
	$$
	\lim\limits_{\tau \to 0} \tilde{f}(e^{2 \pi i z}, e^{2 \pi i h \tau},e^{-2 \pi i \tau})=\lim\limits_{q\to 0} \frac{\theta( \hbar   q^{s} )}{\theta( q^{s})}= \lim\limits_{q\to 0} f( q^{s},\hbar,q).
	$$
\end{proof}
This lemma relates $\tau\to 0$ limits of the elliptic functions to $q\to0$ limits. In particular for the transport (\ref{transpDT}) we obtain the following result:
\begin{Proposition} \label{prolim}
	{\it For $s\in \matR\setminus \Walls_n$, the transport of the quantum differential equation is determined by the following limit:
		\be \label{tranlim2}
		\mathrm{Tran}_{\mathsf{DT}}(s)= \lim\limits_{q\to 0} {\bf R}^{Ell}(q^s,e^{2 \pi i \epsilon_1},e^{2 \pi i \epsilon_2},q).
		\ee
		where 
		$$
		{\bf R}^{Ell}(z,t_1,t_2,q)={\bf{U}}(a,z)^{-1} \, {\bf{L}}(a^{-1},z^{-1} \hbar^{-1}).
		$$}
	\end{Proposition}
	\begin{proof}
		We have
		$$
		\widetilde{\Mon}(z,t_1,t_2,q):=e^{\frac{\ln(\mathsf{E}_{0}) \ln(z)}{\ln(q)}} \,{\bf R}^{Ell}(z,t_1,t_2,q)\,  e^{-\frac{\ln(\mathsf{E}_{\infty}) \ln(z)}{\ln(q)}}.
		$$
		Recall that the matrix elements of  ${\bU}(a,z)$ and $\bL(a,z)$  are balanced in $z$, see Definition 1 in \cite{KononovSmirnov1}. This means that they depend on $z$ through a combination of theta functions of the form
		\be \label{bal}
		\frac{\theta(z^m a )}{\theta(z^m b )}, \ \ m\in \matZ 
		\ee
		where $a$ and $b$ denote  monomials in the equivariant parameters $t_1,t_2$. For the Hilbert scheme $\X$ we know these functions explicitly \cite{SmirnovEllipticHilbert}.
		
		We see that the matrix elements of ${\bf R}^{Ell}(z,t_1,t_2,q)$ are also balanced in~$z$,  and the matrix elements of the monodromy matrix $\widetilde{\Mon}(z,t_1,t_2,q)$
		depend on $z$ via combinations (\ref{bal}) and exponential factors of the form 
		$$
		e^{\frac{\ln(a) \ln(z)}{\ln(q)}}
		$$
		where $a$ denotes a monomial in $t_1,t_2$. 
		As $\widetilde{\Mon}(z,t_1,t_2,q)$ is a $q$-periodic function in $z$, these exponential factors complete factors (\ref{bal}) to  $q$-periodic functions. In other words, the matrix elements of $\widetilde{\Mon}(z,t_1,t_2,q)$  only on $z$ via the functions of the form
		\be \label{monfact}
		\tilde{f}(z,t_1,t_2,q)= e^{m \frac{\ln(z) \ln(a/b)}{\ln(q)}} \frac{\theta(z^m a )}{\theta(z^m b )},
		\ee
		so that $\tilde{f}(z q ,t_1,t_2,q)=\tilde{f}(z,t_1,t_2,q)$.  
		
		Let us consider limit (\ref{transpDT}) of $\widetilde{\Mon}(z,t_1,t_2,q)$ with $\tau$ approaching $0$ along the positive part of the imaginary axis. By Lemma \ref{limlemm} this limit exists and is equal to (\ref{tranlim2}) for generic values of $s$. 
		
		Finally, arguing as in the proof of Lemma \ref{limlemm}, we see that the limits exists for those $s\in \matR$ with $m s \not \in \matZ$, for all $m$ appearing as exponents in (\ref{monfact}) for all matrix elements of $\widetilde{\Mon}(z,t_1,t_2,q)$. From the explicit formulas for the elliptic stable envelope classes \cite{SmirnovEllipticHilbert} we know that the exponents for $\X=\textrm{Hilb}^{n}(\matC^2)$ are bounded by
		$0<m\leq n$. This means that the limit exists for all $s \in \matR\setminus \Walls_n$.
		
	\end{proof}

	Combining Proposition \ref{prolim} with results of Theorems \ref{limmon} and \ref{monfac} we obtain a  representation-theoretic (in terms of $\B_w \in \gt$) and an algebro-geometric (in terms of the K-theoretic stable envelopes) descriptions of the transport.
	\begin{Theorem} \label{transtheor}
		{\it The transport of the quantum connection $\nabla_{\mathsf{DT}}$ from $z=0$ to $z=\infty$ along a line $\gamma_s$ intersecting $|z|=1$ at a non-singular  point $z=e^{2 \pi i s}$ equals
			$$
			\mathrm{Tran}_{\mathsf{DT}}(s)=  \left\{\begin{array}{ll}
			\prod\limits_{w\in (0,s)}^{\longleftarrow}\, (\B^{\bullet}_w)^{-1} \cdot \,  \T   & s \geq 0, \\
			\prod\limits^{\longrightarrow}_{w\in (s,0)}\, \B^{\bullet}_{w} \, \cdot  \T, & s<0.
			\end{array}\right.
			$$
			These matrices have the following Gauss decomposition
			$$
			\mathrm{Tran}_{\mathsf{DT}}(s)=  \bU^{\pm}_{s}(a)^{-1} \bD_{s}(\hbar)\,  \bL^{\pm}_{-s}(a^{-1})
			$$
			where $\bD_{s}(\hbar)$ is diagonal matrix (\ref{Ddef}) and  $\bU^{\pm}_{s}(a),\bL^{\pm}_{s}(a)$ denote the matrices of expansions of the K-theoretic stable envelope classes in the basis of fixed points (\ref{defrestr}).
		}
	\end{Theorem}
	
	\noindent
	\begin{Remark}
		Note that in the above theorem $s \not \in \Walls_n$ and thus 
		$$
		\bU^{+}_{s}(a)=\bU^{-}_{s}(a), \ \ \ \bL^{+}_{s}(a)=\bL^{-}_{s}(a).
		$$
	\end{Remark}
	
	\begin{Remark}
		We need to clarify the ambiguity for choosing a path $\gamma_s$ for values of $s$ which differ by $2\pi$. Let 
		$\lfloor s \rfloor$ the the integral part of $s$. Then $\gamma_s$ is homotopy equivalent to a path which winds around $z=0$  counterclockwise $\lfloor s \rfloor$ times and then proceeds to $z=\infty$ intersecting $|z|=1$ at 
		$e^{2 \pi i s}$. Fig. \ref{monodromies} below illustrates this situation.  As $z=0$ is a singular point of qde, winding around this point amounts in a non-trivial monodromy and thus $\mathrm{Tran}_{\DT}(s) \neq \mathrm{Tran}_{\DT}(s + 2\pi)$. 
	\end{Remark}

	\begin{figure}[H]
		\centering
		\includegraphics[width=10.5cm]{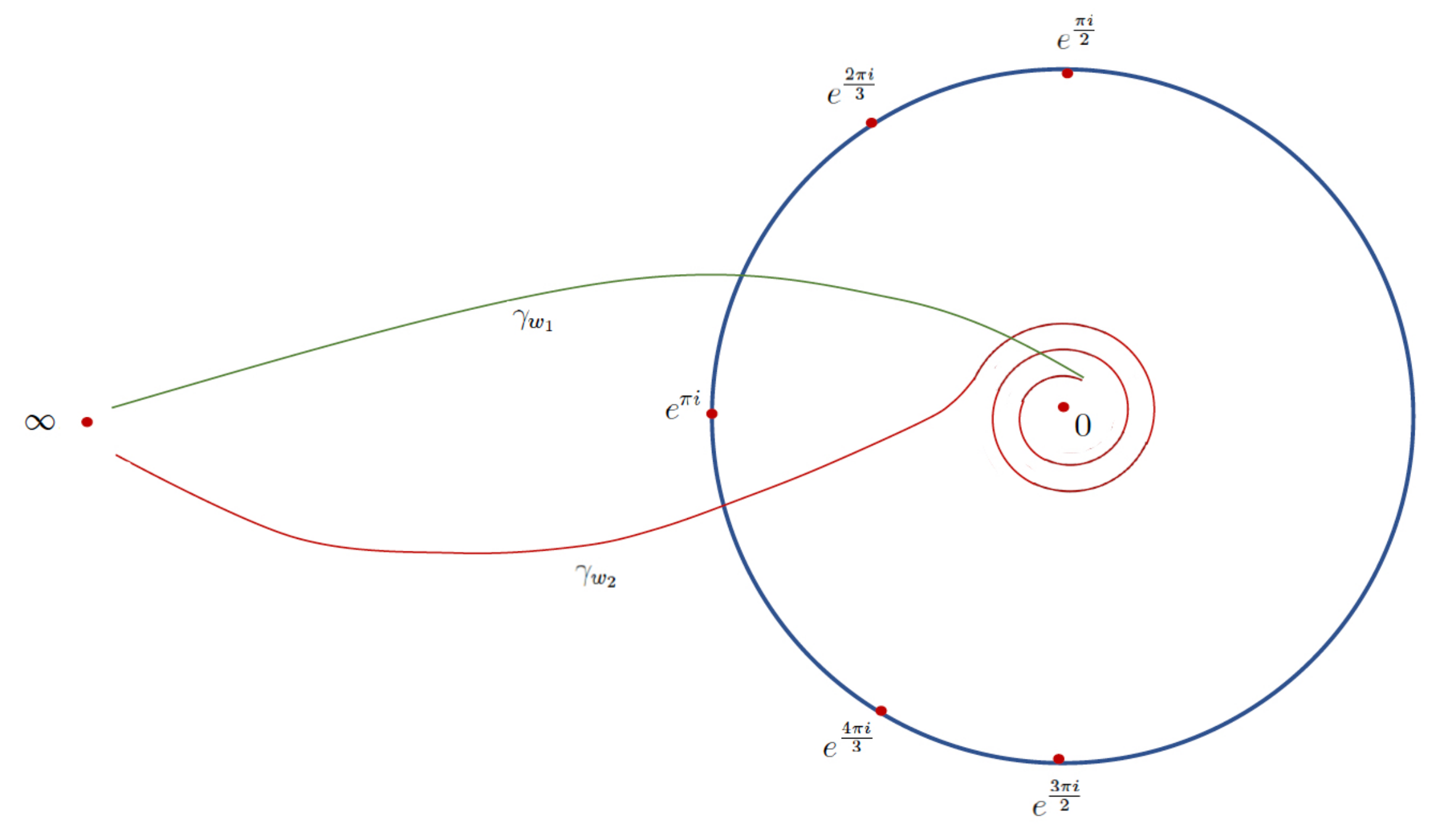}
		\caption{\label{monodromies} Example of two paths $\gamma_{w_1},\gamma_{w_2}$ with $\lfloor w_1\rfloor =0$ and $\lfloor w_2\rfloor =2$. }
	\end{figure}

\subsection{Transport for $s=0$ \label{tmatsec}}
For $s=0$ Theorem \ref{transtheor} gives
$$
\mathrm{Tran}_{\mathsf{DT}}(0)=\T. 
$$
Comparing definition (\ref{tmatdef}) with Theorem \ref{tran0}
we obtain that $\textsf{H}^{*}=\textsf{P}^{*}$ and we arrive at the following result:
\begin{Theorem} \label{t0th}
	
\indent 
{\it 
		\begin{itemize}
			\item The operator ${\bf M}(\infty)$ is diagonal in the basis of Macdonald polynomials with substitution $P^{\ast}_{\lambda}=\left.P_{\lambda}\right|_{p_{i}=-p_i}$. 
			\item The transport of the connection $\nabla_{\mathsf{DT}}$ from $z=0$ to $z=\infty$ along $\matR_{+}$ maps the eigenvectors of the operator 
			${\bf M}(0)$ to the eigenvectors of the operator ${\bf M}(\infty)$.
		\end{itemize}
	}
\end{Theorem}

\subsection{The fundamental group $\pi_{1}(\mathbb{P}^{1}\setminus\sing,0_{+})$ \label{loopmondrs}} 
In this section we compute the representation of the fundamental group $\pi_{1}(\mathbb{P}^{1}\setminus\sing,0)$. As $z=0$ is also a singularity of the quantum differential equation we consider the fundamental group $\pi_{1}(\mathbb{P}^{1}\setminus\sing,0_{+})$ where $0_{+}$ is a point infinitesimally close to $z=0$. More specifically (the relations in the group will depend on this choice) we choose $0_{+}$ above the cut along $\matR_{+}$, see Fig. \ref{monodex}.    This group is well defined.

We choose the  following generators of this group:
\begin{itemize}
	\item $\gamma_{0}$ is a counterclockwise oriented loop based at $0_{+}$ around the singularity $z=0$,
	\item $\gamma_{\infty}$ is a clockwise oriented loop based at $0_{+}$ around $z=\infty$ chosen as follows: first is goes from $0_{+}$ to $\infty$ along the $\matR_{+}$, then  around $z=\infty$ and goes back to $0_{+}$ along $\matR_{+}$. 
	\item $\gamma_{w}$ for $w=\frac{a}{b}$ is a counterclockwise oriented loop based at $0_{+}$ around singularity located at the root of unity
	$$
	\zeta = e^{\frac{2 \pi i a}{b}}
	$$
	which \underline{do not intersect} $\matR_{+}$.  It is clear that these loops are labeled by $w\in \Walls_n$ with $0<w<1$.
\end{itemize}
With this choice of the generators, $\pi_{1}(\mathbb{P}^{1}\setminus\sing,0_{+})$ is isomorphic to the following group:
$$
\Big\langle \gamma_{0}, \gamma_{w}, \gamma_{\infty}  , \ \ w \in \Walls_{n}, \ \ 0<w<1 \Big\rangle \Big/ \Big(\gamma_0 \prod\limits_{w \in (0,1)}^{\longleftarrow} \gamma_w = \gamma_{\infty}\Big)
$$
For $n=3$ this loop are shown in Fig \ref{monodex}. The relation satisfied by $\gamma_{w}$ is explained in Fig \ref{compex}. 

\begin{figure}[H]
	\centering
	\includegraphics[width=13.5cm]{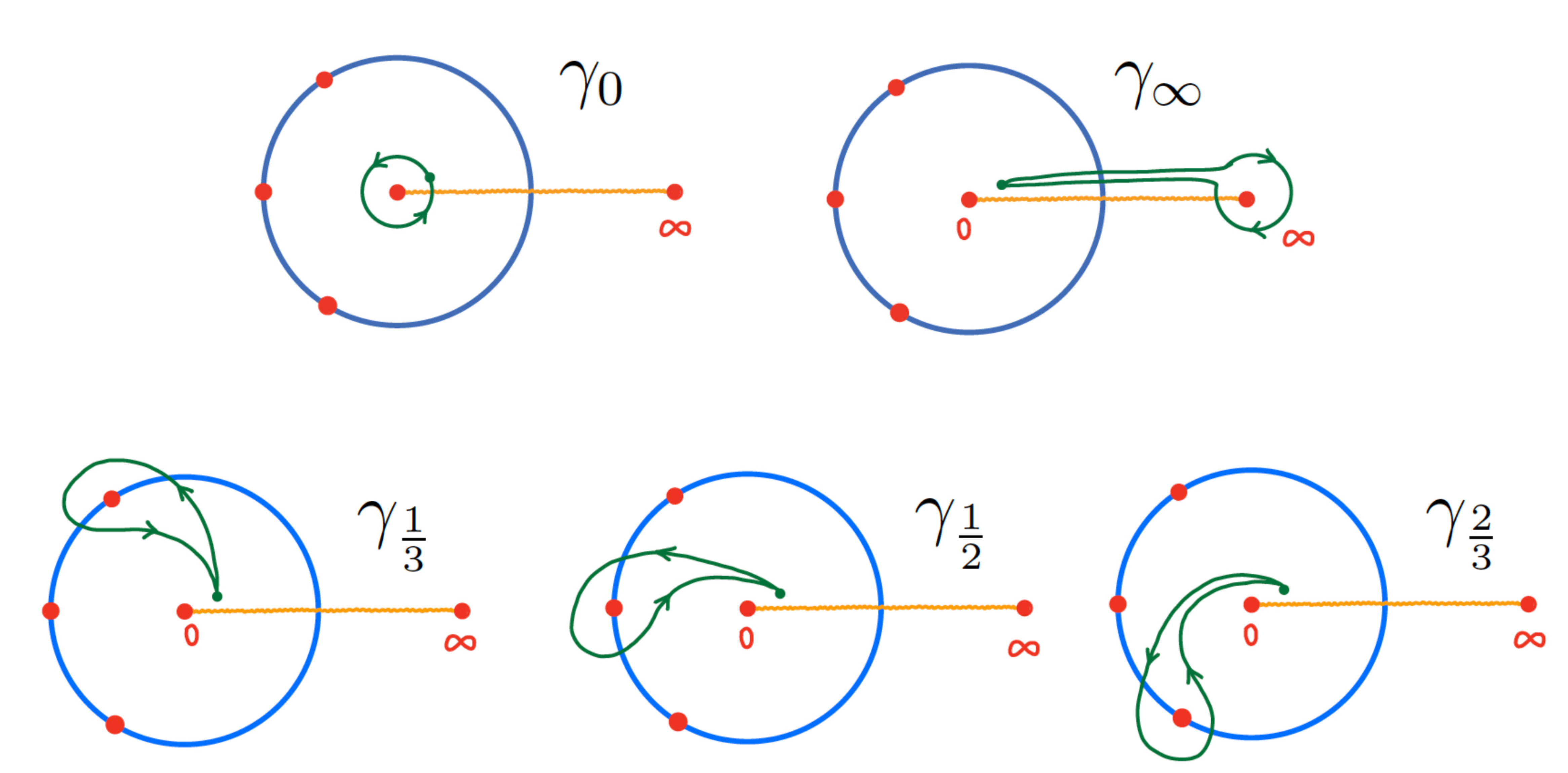}
	\caption{\label{monodex} Generators of $\pi_{1}(\mathbb{P}^{1}\setminus\sing,0_{+})$  for $n=3$.
	}
\end{figure}

\begin{Example}
For $n=3$ the examples of compositions of generators of $\pi_{1}(\mathbb{P}^{1}\setminus\sing,0_{+})$ for $n=3$:
\begin{figure}[H]
	\centering
	\includegraphics[width=14cm]{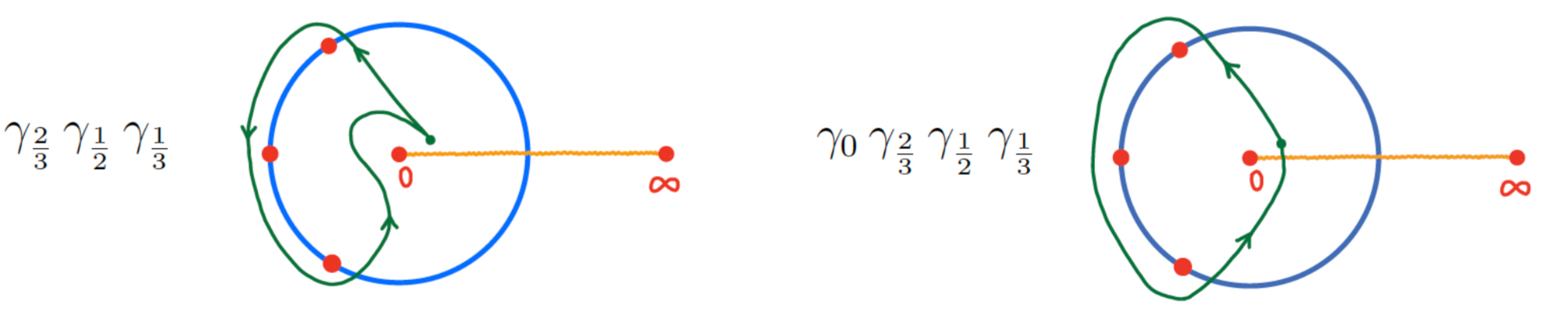}
	\caption{\label{compex} Generators of $\pi_{1}(\mathbb{P}^{1}\setminus\sing,0_{+})$  for $n=3$.
	}
\end{figure}

\noindent
It is clear that the loop representing $\gamma_{0} \gamma_{\frac{2}{3}} \gamma_{\frac{1}{2}} \gamma_{\frac{1}{3}}$, in the above figure, is homotopy equivalent to $\gamma_{\infty}$ in Fig. \ref{monodex} which gives the relation in the fundamental group
$$
\gamma_{0} \gamma_{\frac{2}{3}} \gamma_{\frac{1}{2}} \gamma_{\frac{1}{3}}=\gamma_{\infty}.
$$
\end{Example}

\subsection{Fock representation of the fundamental group}
The transformation properties of $\psi^{0}_{\mathsf{DT}}(z)$ in the neighborhood of $z=0$ are completely determined by the factor $z^{c^{(0)}}$ in (\ref{solzer}). Going around $\gamma_{0}$ amounts in the transformation
$$
\psi^0_{\mathsf{DT}}(\gamma_{0} \cdot z) = \psi^{0}_{\mathsf{DT}}(z)\, \mathsf{E}_{0}
$$
where $\mathsf{E}_{0}$ is the diagonal matrix with eigenvalues
$$
e^{2\pi i c^{(0)}_{\lambda}}=\left.\mathscr{O}(1)\right|_{\lambda}=\prod\limits_{(i,j) \in \lambda} t_1^{i-1} t_2^{j-1}.
$$
Thus 
$$
\nabla_{\DT}(\gamma_{0})=\mathsf{E}_{0}.
$$
Similarly, the fundamental solution of the qde in the neighborhood of $z=\infty$  transforms as
$$
\psi^{\infty}_{\mathsf{DT}}(z) \to \psi^{\infty}_{\mathsf{DT}}(z) \mathsf{E}_{\infty}
$$
if we go around a small loop containing $z=\infty$. This means that 
$$
\nabla_{\DT}(\gamma_{\infty}) = \T \,\mathsf{E}_{\infty}\, \T^{-1}
$$
where $\T=\mathrm{Tran}_{\mathsf{DT}}(0)$ denotes the transport of $\nabla_{\mathsf{DT}}$
from $z=0$ to $z=\infty$ along $\matR_{+}$. Theorem \ref{tran0} describes the matrix $\T$ explicitly. 

\begin{Proposition} \label{loopmnth}
	{\it For $w\in \Walls_{n}$ with $0<w<1$ we have
		$$
		\nabla_{\DT}(\gamma_w)=\B^{\bullet}_{w}
		$$
		where $\B^{\bullet}_{w}$ is the matrix of the operator (\ref{monodromyB}) in the basis of the Macdonald polynomials. 
	}
\end{Proposition}
\begin{proof}
	Let us consider $s',s \in \matR$ so that there is only one element $w\in \Walls_{n}$ such that $1>s'>w>s>0$. Then, by definition,
	$$
	\nabla_{\DT}(\gamma_w) = \textrm{Tran}(s)\, \textrm{Tran}(s')^{-1},
	$$
	The proposition follows from Theorem \ref{transtheor}. 
\end{proof}
Combining all this together we arrive at the following result 
\begin{Theorem} \label{nomodtheor}
	{\it The representation (\ref{defmonod}) in generators is given by
		$$
		\nabla_{\mathsf{DT}}: \ \ \gamma_{0} \to \mathsf{E}_{0}, \ \ \gamma_{\infty}\to  \T \,\mathsf{E}_{\infty}\, \T^{-1}, \ \ \gamma_{w}\to \B^{\bullet}_{w}.
		$$}
\end{Theorem}

\subsection{Relation in $\pi_{1}(\mathbb{P}^{1}\setminus\sing,0_{+})$ } 
The generators of $\pi_{1}(\mathbb{P}^{1} \setminus \textrm{Sing},0_{+})$ are subject to the following relation
$$
\gamma_0 \prod\limits_{w \in (0,1)}^{\longleftarrow}\, \gamma_w = \gamma_{\infty}.
$$
Applying the anti-homomorphism $\nabla_{\DT}$ we obtain
$$
\Big(\prod\limits_{w\in (0,1)}^{\longrightarrow} \B^{\bullet}_{w} \Big)\textsf{E}_{0}  \,= \T \, \textsf{E}_{\infty}\, \T^{-1}.
$$
The left side of this equation is simply the matrix of the operator ${\bf M}(\infty)$ in the basis of fixed points. Thus, the last relation says that $\T$ maps the eigenvectors of ${\bf M}(0)$ to the eigenvectors of ${\bf M}(\infty)$ which gives another proof of Theorem \ref{t0th}. 

\section{Appendices}
\appendix
\section{Jack Polynomials \label{appjack}}

Let $j^{(\beta)}_{\lambda}$ denote the standard Jack polynomial with parameter $\beta$. We denote 
$$
n^{\beta}_{\lambda}=\prod\limits_{\Box\in \lambda} (l_{\lambda}(\Box) \beta^{-1} +a_{\lambda}(\Box)+1  )
$$
where
$$
a_{\lambda}(\Box)= \lambda_{i}-j, \ \ \ l_{\lambda}(\Box)= \lambda'_{j}-i
$$
denote the standard arm and leg lengths of a box $\Box=(i,j)$ in $\lambda$ and $\lambda'$ is the transposed Young diagram. Let
$$
{\bf j}^{(\beta)}_{\lambda}=n^{\beta}_{\lambda}\,j^{(\beta)}_{\lambda}
$$
and
$$
J_{\lambda}= (-\epsilon_1)^{|\lambda|}\, \left.{\bf j}^{(-\epsilon_1/\epsilon_2)}_{\lambda}\right|_{p_i=-\epsilon_2 p_i}.
$$
The Jack polynomials normalized in this way correspond to the basis of torus fixed points in equivariant cohomology of $\textrm{Hilb}^{n}(\matC^2)$ under isomorphism (\ref{FockDef}). These polynomials form an eigenbasis of  ${\mathbf{m}}(0)$.

Here are the first several examples of Jack polynomials normalized this way:

\noindent
{$n=1$:}
$$
J_{[1]}=\epsilon_{{1}}\epsilon_{{2}}p_{{1}}
$$

\noindent{$n=2$:}
$$
J_{[2]}={\epsilon_{{1}}}^{2}{\epsilon_{{2}}}^{2}{p_{{1}}}^{2}+\epsilon_{{1}}{
	\epsilon_{{2}}}^{2}p_{{2}}, \ \ \ J_{[1,1]}={\epsilon_{{1}}}^{2}{\epsilon_{{2}}}^{2}{p_{{1}}}^{2}+{\epsilon_{{1}}}
^{2}\epsilon_{{2}}p_{{2}}
$$

\noindent{$n=3$:}
$$
\begin{array}{l}
J_{[3]}={\epsilon_{{1}}}^{3}{\epsilon_{{2}}}^{3}{p_{{1}}}^{3}+3\,{\epsilon_{{1
}}}^{2}{\epsilon_{{2}}}^{3}p_{{1}}p_{{2}}+2\,\epsilon_{{1}}{\epsilon_{
		{2}}}^{3}p_{{3}}\\
\\
J_{[2,1]}={\epsilon_{{1}}}^{3}{\epsilon_{{2}}}^{3}{p_{{1}}}^{3}+ \left( \epsilon
_{{1}}+\epsilon_{{2}} \right) {\epsilon_{{1}}}^{2}{\epsilon_{{2}}}^{2}
p_{{2}}p_{{1}}+{\epsilon_{{2}}}^{2}p_{{3}}{\epsilon_{{1}}}^{2}\\
\\
J_{[1,1,1]}={\epsilon_{{1}}}^{3}{\epsilon_{{2}}}^{3}{p_{{1}}}^{3}+3\,{\epsilon_{{1
}}}^{3}{\epsilon_{{2}}}^{2}p_{{1}}p_{{2}}+2\,{\epsilon_{{1}}}^{3}
\epsilon_{{2}}p_{{3}}
\end{array}
$$
The transition matrix from the basis $p_{\lambda}$ to the basis of $J_{\lambda}$ has the form:

\noindent
{$n=2$:}
$$
\textsf{J}= \left[ \begin {array}{cc} \epsilon_{{1}}{\epsilon_{{2}}}^{2}&{
	\epsilon_{{1}}}^{2}\epsilon_{{2}}\\ \noalign{\medskip}{\epsilon_{{1}}}
^{2}{\epsilon_{{2}}}^{2}&{\epsilon_{{1}}}^{2}{\epsilon_{{2}}}^{2}
\end {array} \right] 
$$

\noindent
{$n=3$:}
$$
\textsf{J}=  \left[ \begin {array}{ccc} 2\,\epsilon_{{1}}{\epsilon_{{2}}}^{3}&{
	\epsilon_{{1}}}^{2}{\epsilon_{{2}}}^{2}&2\,{\epsilon_{{1}}}^{3}
\epsilon_{{2}}\\ \noalign{\medskip}3\,{\epsilon_{{2}}}^{3}{\epsilon_{{
			1}}}^{2}& \left( \epsilon_{{1}}+\epsilon_{{2}} \right) {\epsilon_{{1}}
}^{2}{\epsilon_{{2}}}^{2}&3\,{\epsilon_{{2}}}^{2}{\epsilon_{{1}}}^{3}
\\ \noalign{\medskip}{\epsilon_{{2}}}^{3}{\epsilon_{{1}}}^{3}&{
	\epsilon_{{2}}}^{3}{\epsilon_{{1}}}^{3}&{\epsilon_{{2}}}^{3}{\epsilon_
	{{1}}}^{3}\end {array} \right] 
$$
These matrices provide initial values of the fundamental solutions of the quantum differential equations normalized as in (\ref{solzernorm}).
\section{Macdonald polynomials \label{macappen}}
Here is the first several polynomials in the normalization we use (see Section 2.2 in \cite{GorskyNegut} for definitions):

\noindent
$\underline{n=1}$:
$$
P_{[1]}=p_{{1}}.
$$

\noindent
$\underline{n=2}$:
$$
P_{[1,1]}={\frac { \left( 1+t_{{2}} \right) {p_{{1}}}^{2}}{2 t_{{2}}}}+{
	\frac{ \left( t_{{2}}-1 \right) p_{{2}}}{ 2 t_{{2}}}},
$$

$$
P_{[2]}={\frac{ \left( 1+t_{{1}} \right) {p_{{1}}}^{2}}{2 t_{{1}}}}+{
	\frac{ \left( t_{{1}}-1 \right) p_{{2}}}{2 t_{{1}}}}.
$$
The matrix $\mathsf{P}$, by definition, is such that its $\lambda$-th column is a vector $P_{\lambda}$ in basis $p_{\mu}$. From above formulas we find:
$$
\mathsf{P}= \left[ \begin {array}{cc} {\frac {1+t_{{2}}}{2 t_{{2}}}}&{
	\frac{1+t_{{1}}}{2 t_{{1}}}}\\ \noalign{\medskip}{\frac {t_{{2}}-1
	}{2 t_{{2}}}}&{\frac{t_{{1}}-1}{2 t_{{1}}}}\end {array} \right] 
$$
We  compute 
$$
(-1)^l= \left[ \begin {array}{cc} 1&0\\ \noalign{\medskip}0&-1\end {array}
\right], \ \ \ \fp= \left[ \begin {array}{cc} 0&1\\ \noalign{\medskip}1&0\end {array}
\right]
$$
and thus
$$
\mathsf{P}^{\ast}:=(-1)^{l}\, \mathsf{P}\, \fp = \left[ \begin {array}{cc} 1/2\,{\frac {1+t_{{1}}}{t_{{1}}}}&1/2\,{
	\frac {1+t_{{2}}}{t_{{2}}}}\\ \noalign{\medskip}-1/2\,{\frac {t_{{1}}-
		1}{t_{{1}}}}&-1/2\,{\frac {t_{{2}}-1}{t_{{2}}}}\end {array} \right]. 
$$

\noindent
$\underline{n=3}$:
$$
\begin{array}{l}
P_{[1,1,1]}={\dfrac { \left( {t_{{2}}}^{2}+t_{{2}}+1 \right)  \left( 1+t_{{2}}
		\right) {p_{{1}}}^{3}}{6 {t_{{2}}}^{3}}}+{\dfrac { \left( t_{{2}}-1
		\right)  \left( {t_{{2}}}^{2}+t_{{2}}+1 \right) p_{{2}}p_{{1}}}{2 {t_{{
					2}}}^{3}}}+{\dfrac { \left( t_{{2}}-1 \right) ^{2} \left( 1+t_{{2}
		} \right) p_{{3}}}{3 {t_{{2}}}^{3}}},\\
P_{[2,1]}={\dfrac { \left( t_{{1}}t_{{2}}+2\,t_{{1}}+2\,t_{{2}}+1 \right) {p
			_{{1}}}^{3}}{6 t_{{1}}t_{{2}}}}+{\dfrac { \left( t_{{1}}t_{{2}}-1
		\right) p_{{2}}p_{{1}}}{2 t_{{1}}t_{{2}}}}+{\dfrac { \left( t_{{1}}
		-1 \right)  \left( t_{{2}}-1 \right) p_{{3}}}{3 t_{{1}}t_{{2}}}},\\
P_{[3]}={\dfrac { \left( 1+t_{{1}} \right)  \left( {t_{{1}}}^{2}+t_{{1}}+1
		\right) {p_{{1}}}^{3}}{6 {t_{{1}}}^{3}}}+{\dfrac { \left( t_{{1}}-1
		\right)  \left( {t_{{1}}}^{2}+t_{{1}}+1 \right) p_{{2}}p_{{1}}}{{2 t_{{
					1}}}^{3}}}+{\dfrac { \left( t_{{1}}-1 \right) ^{2} \left( 1+t_{{1}
		} \right) p_{{3}}}{3 {t_{{1}}}^{3}}}.
\end{array}
$$

$$
\mathsf{P}= \left[ \begin {array}{ccc} {\frac{ \left( {t_{{2}}}^{2}+t_{{2}}
		+1 \right)  \left( 1+t_{{2}} \right) }{6 {t_{{2}}}^{3}}}&{\frac {t_
		{{1}}t_{{2}}+2\,t_{{1}}+2\,t_{{2}}+1}{6 t_{{1}}t_{{2}}}}&{\frac {
		\left( 1+t_{{1}} \right)  \left( {t_{{1}}}^{2}+t_{{1}}+1 \right) }{6 {t
			_{{1}}}^{3}}}\\ \noalign{\medskip}{\frac { \left( t_{{2}}-1
		\right)  \left( {t_{{2}}}^{2}+t_{{2}}+1 \right) }{2 {t_{{2}}}^{3}}}&
{\frac{t_{{1}}t_{{2}}-1}{2 t_{{1}}t_{{2}}}}&{\frac{ \left( t_{{
				1}}-1 \right)  \left( {t_{{1}}}^{2}+t_{{1}}+1 \right) }{2 {t_{{1}}}^{3}}
}\\ \noalign{\medskip}{\frac{ \left( t_{{2}}-1 \right) ^{2}
		\left( 1+t_{{2}} \right) }{3 {t_{{2}}}^{3}}}&{\frac{ \left( t_{{1
		}}-1 \right)  \left( t_{{2}}-1 \right) }{3 t_{{1}}t_{{2}}}}&{\frac 
	{ \left( t_{{1}}-1 \right) ^{2} \left( 1+t_{{1}} \right) }{3 {t_{{1}}}^{
			3}}}\end {array} \right] 
$$
We compute 
\be \label{smat3}
(-1)^{l}= \left[ \begin {array}{ccc} 1&0&0\\ \noalign{\medskip}0&-1&0
\\ \noalign{\medskip}0&0&1\end {array} \right], \ \ \ 
\fp = \left[ \begin {array}{ccc} 0&0&1\\ \noalign{\medskip}0&1&0
\\ \noalign{\medskip}1&0&0\end {array} \right] 
\ee
and thus
$$
\mathsf{P}^{\ast}=(-1)^{l}\, \mathsf{P}\, \fp= \left[ \begin {array}{ccc} 1/6\,{\frac { \left( 1+t_{{1}} \right) 
		\left( {t_{{1}}}^{2}+t_{{1}}+1 \right) }{{t_{{1}}}^{3}}}&1/6\,{\frac 
	{t_{{1}}t_{{2}}+2\,t_{{1}}+2\,t_{{2}}+1}{t_{{1}}t_{{2}}}}&1/6\,{\frac 
	{ \left( {t_{{2}}}^{2}+t_{{2}}+1 \right)  \left( 1+t_{{2}} \right) }{{
			t_{{2}}}^{3}}}\\ \noalign{\medskip}-1/2\,{\frac { \left( t_{{1}}-1
		\right)  \left( {t_{{1}}}^{2}+t_{{1}}+1 \right) }{{t_{{1}}}^{3}}}&-1/
2\,{\frac {t_{{1}}t_{{2}}-1}{t_{{1}}t_{{2}}}}&-1/2\,{\frac { \left( t_
		{{2}}-1 \right)  \left( {t_{{2}}}^{2}+t_{{2}}+1 \right) }{{t_{{2}}}^{3
}}}\\ \noalign{\medskip}1/3\,{\frac { \left( t_{{1}}-1 \right) ^{2}
		\left( 1+t_{{1}} \right) }{{t_{{1}}}^{3}}}&1/3\,{\frac { \left( t_{{1
		}}-1 \right)  \left( t_{{2}}-1 \right) }{t_{{1}}t_{{2}}}}&1/3\,{\frac 
	{ \left( t_{{2}}-1 \right) ^{2} \left( 1+t_{{2}} \right) }{{t_{{2}}}^{
			3}}}\end {array} \right] 
$$

\section{Stable envelopes of $\X$ \label{stabxapp} } 
In this section we give explicit formulas for the matrices of K-theoretic stable envelopes defined by (\ref{defrestr}). All formulas are written in variables:
\be \label{tahvar}
t_{{1}}=ah, \ \ \ t_{{2}}={\frac {h}{a}}.
\ee
so that $h=\hbar^{1/2}$.

\noindent
$\underline{n=2}$:
$$
{\bf U}^{+}_{0}(a)={\bf U}^{-}_{1/2}(a)= \left[ \begin {array}{cc} {\frac { \left( -h+a \right) ^{2} \left( a+
		h \right) }{{h}^{2}a}}&-{\frac { \left( h-1 \right)  \left( h+1
		\right)  \left( -h+a \right) }{{h}^{2}a}}\\ \noalign{\medskip}0&{
	\frac { \left( -h+a \right)  \left( a+1 \right)  \left( a-1 \right) }{
		{a}^{2}h}}\end {array} \right] 
$$
$$
{\bf U}^{+}_{1/2}(a)= {\bf U}^{-}_{1}(a)=  \left[ \begin {array}{cc} {\frac { \left( -h+a \right) ^{2} \left( a+
		h \right) }{{h}^{2}a}}&-{\frac {a \left( -h+a \right)  \left( h-1
		\right)  \left( h+1 \right) }{{h}^{2}}}\\ \noalign{\medskip}0&{\frac 
	{ \left( -h+a \right)  \left( a+1 \right)  \left( a-1 \right) }{{a}^{2
		}h}}\end {array} \right] 
$$
By Lemma \ref{luskthlem}, ${\bf L}^{\pm}_{w}(a)=\fp {\bf U}^{\pm}_{w}(a^{-1}) \fp $ and we compute:
$$
{\bf L}^{+}_{0}(a)= {\bf L}^{-}_{1/2}(a)=\left[ \begin {array}{cc} {\frac { \left( ah-1 \right)  \left( a+1
		\right)  \left( a-1 \right) }{ah}}&0\\ \noalign{\medskip}{\frac {
		\left( h+1 \right)  \left( h-1 \right)  \left( ah-1 \right) }{{h}^{2}
}}&{\frac { \left( ah-1 \right) ^{2} \left( ah+1 \right) }{{a}^{2}{h}^
		{2}}}\end {array} \right], 
$$
$$
{\bf L}^{+}_{1/2}(a)= {\bf L}^{-}_{1}(a)= \left[ \begin {array}{cc} {\frac { \left( ah-1 \right)  \left( a+1
		\right)  \left( a-1 \right) }{ah}}&0\\ \noalign{\medskip}{\frac {
		\left( h+1 \right)  \left( h-1 \right)  \left( ah-1 \right) }{{a}^{2}
		{h}^{2}}}&{\frac { \left( ah-1 \right) ^{2} \left( ah+1 \right) }{{a}^
		{2}{h}^{2}}}\end {array} \right]. 
$$
All other stable envelope matrices differ from the ones given above by an integral shift of the slope, and can be computed by:
$$
{\bf U}^{\pm}_{w+1}(a)=\mathscr{O}(1) {\bf U}^{\pm}_{w}(a) \mathscr{O}(1)^{-1}, \ \ \ {\bf L}^{\pm}_{w+1}(a)=\mathscr{O}(1) {\bf L}^{\pm}_{w}(a) \mathscr{O}(1)^{-1},
$$
with
$$
\mathscr{O}(1)= \left[ \begin {array}{cc} {\frac {a}{h}}&0\\ \noalign{\medskip}0&{
	\frac {1}{ah}}\end {array} \right]. 
$$

\noindent
$\underline{n=3}$:
$$
\begin{array}{l}
{\bf U}^{+}_{0}(a)={\bf U}^{-}_{1/3}(a)= \\ \\ \left[ \begin {array}{ccc} -{\frac { \left( -h+a \right) ^{3} \left( 
		a+h \right)  \left( {a}^{2}+ah+{h}^{2} \right) }{{a}^{3/2}{h}^{9/2}}}&
{\frac { \left( -h+a \right) ^{2} \left( h-1 \right)  \left( h+1
		\right)  \left( a+h \right) }{{h}^{7/2}{a}^{3/2}}}&-{\frac { \left( -
		h+a \right)  \left( h-1 \right)  \left( h+1 \right)  \left( a{h}^{3}-1
		\right) }{{a}^{5/2}{h}^{9/2}}}\\ \noalign{\medskip}0&-{\frac {
		\left( -h+a \right) ^{2} \left( {a}^{3}-h \right) }{{h}^{5/2}{a}^{5/2
}}}&{\frac { \left( ah+1 \right)  \left( -h+a \right)  \left( h-1
		\right)  \left( h+1 \right)  \left( a+1 \right)  \left( a-1 \right) 
	}{{a}^{7/2}{h}^{7/2}}}\\ \noalign{\medskip}0&0&-{\frac { \left( a+1
		\right)  \left( a-1 \right)  \left( {a}^{3}h-1 \right)  \left( -h+a
		\right) }{{h}^{5/2}{a}^{9/2}}}\end {array} \right] 
\end{array}
$$

$$
\begin{array}{l}
{\bf U}^{+}_{1/3}(a)={\bf U}^{-}_{1/2}(a)= \\ \\  \left[ \begin {array}{ccc} -{\frac { \left( -h+a \right) ^{3} \left( 
		a+h \right)  \left( {a}^{2}+ah+{h}^{2} \right) }{{a}^{3/2}{h}^{9/2}}}&
{\frac { \left( h-1 \right)  \left( h+1 \right)  \left( -h+a \right) ^
		{2} \left( {a}^{2}+1 \right) }{{h}^{7/2}\sqrt {a}}}&-{\frac { \left( h
		-1 \right)  \left( h+1 \right)  \left( -h+a \right)  \left( {a}^{4}{h}
		^{3}-{a}^{2}h-a+h \right) }{{a}^{3/2}{h}^{9/2}}}\\ \noalign{\medskip}0
&-{\frac { \left( -h+a \right) ^{2} \left( {a}^{3}-h \right) }{{h}^{5/
			2}{a}^{5/2}}}&{\frac { \left( {a}^{2}+1 \right)  \left( -h+a \right) 
		\left( h-1 \right)  \left( h+1 \right)  \left( a-1 \right)  \left( a+
		1 \right) }{{h}^{5/2}{a}^{5/2}}}\\ \noalign{\medskip}0&0&-{\frac {
		\left( a+1 \right)  \left( a-1 \right)  \left( {a}^{3}h-1 \right) 
		\left( -h+a \right) }{{h}^{5/2}{a}^{9/2}}}\end {array} \right] 
\end{array}
$$

$$
\begin{array}{l}
{\bf U}^{+}_{1/2}(a)={\bf U}^{-}_{2/3}(a)= \\ \\  \left[ \begin {array}{ccc} -{\frac { \left( -h+a \right) ^{3} \left( 
		a+h \right)  \left( {a}^{2}+ah+{h}^{2} \right) }{{a}^{3/2}{h}^{9/2}}}&
{\frac { \left( h-1 \right)  \left( h+1 \right)  \left( -h+a \right) ^
		{2} \left( {a}^{2}+1 \right) }{{h}^{7/2}\sqrt {a}}}&{\frac { \left( h-
		1 \right)  \left( h+1 \right)  \left( -h+a \right)  \left( {a}^{4}{h}^
		{2}-{a}^{3}{h}^{3}-{a}^{2}{h}^{2}+1 \right) }{{h}^{9/2}\sqrt {a}}}
\\ \noalign{\medskip}0&-{\frac { \left( -h+a \right) ^{2} \left( {a}^{
			3}-h \right) }{{h}^{5/2}{a}^{5/2}}}&{\frac { \left( {a}^{2}+1 \right) 
		\left( -h+a \right)  \left( h-1 \right)  \left( h+1 \right)  \left( a
		-1 \right)  \left( a+1 \right) }{{h}^{5/2}{a}^{5/2}}}
\\ \noalign{\medskip}0&0&-{\frac { \left( a+1 \right)  \left( a-1
		\right)  \left( {a}^{3}h-1 \right)  \left( -h+a \right) }{{h}^{5/2}{a
		}^{9/2}}}\end {array} \right]
\end{array}
$$

$$
\begin{array}{l}
{\bf U}^{+}_{2/3}(a)={\bf U}^{-}_{1}(a)= \\ \\  \left[ \begin {array}{ccc} -{\frac { \left( -h+a \right) ^{3} \left( 
		a+h \right)  \left( {a}^{2}+ah+{h}^{2} \right) }{{a}^{3/2}{h}^{9/2}}}&
{\frac { \left( -h+a \right) ^{2}{a}^{3/2} \left( h-1 \right)  \left( 
		h+1 \right)  \left( a+h \right) }{{h}^{9/2}}}&-{\frac { \left( h-1
		\right)  \left( h+1 \right) {a}^{7/2} \left( -h+a \right)  \left( a{h
		}^{3}-1 \right) }{{h}^{9/2}}}\\ \noalign{\medskip}0&-{\frac { \left( -
		h+a \right) ^{2} \left( {a}^{3}-h \right) }{{h}^{5/2}{a}^{5/2}}}&{
	\frac { \left( ah+1 \right)  \left( -h+a \right)  \left( h-1 \right) 
		\left( h+1 \right)  \left( a+1 \right)  \left( a-1 \right) }{\sqrt {a
		}{h}^{5/2}}}\\ \noalign{\medskip}0&0&-{\frac { \left( a+1 \right) 
		\left( a-1 \right)  \left( {a}^{3}h-1 \right)  \left( -h+a \right) }{
		{h}^{5/2}{a}^{9/2}}}\end {array} \right]
\end{array}
$$
By Lemma \ref{luskthlem}, ${\bf L}^{\pm}_{w}(a)=\fp {\bf U}^{\pm}_{w}(a^{-1}) \fp $ and using (\ref{smat3}) these matrices are easy to compute explicitly. The stable envelope matrices for other slopes are determined from
$$
{\bf U}^{\pm}_{w+1}(a)=\mathscr{O}(1) {\bf U}^{\pm}_{w}(a) \mathscr{O}(1)^{-1}, \ \ \ {\bf L}^{\pm}_{w+1}(a)=\mathscr{O}(1) {\bf L}^{\pm}_{w}(a) \mathscr{O}(1)^{-1},
$$
with
$$
\mathscr{O}(1)= \left[ \begin {array}{ccc} {\frac {{a}^{3}}{{h}^{3}}}&0&0
\\ \noalign{\medskip}0&{h}^{-2}&0\\ \noalign{\medskip}0&0&{\frac {1}{{
			a}^{3}{h}^{3}}}\end {array} \right]. 
$$

\section{Stable envelopes for $\Y_{w}$ \label{stabyapp}}
Here we give examples of stable envelope matrices (\ref{lufpy}):

\noindent
\underline{Case $n=2$}:
$$ \hspace{-10mm}
\begin{array}{l}
U^{+}_{\Y_{0}}(a)= \left[ \begin {array}{cc} {\frac { \left( a+{\it hh} \right)  \left( 
		-{\it hh}+a \right) ^{2}}{{{\it hh}}^{2}a}}&-{\frac { \left( {\it hh}+
		1 \right)  \left( {\it hh}-1 \right)  \left( -{\it hh}+a \right) }{{{
				\it hh}}^{2}a}}\\ \noalign{\medskip}0&{\frac { \left( -{\it hh}+a
		\right)  \left( a+1 \right)  \left( a-1 \right) }{{a}^{2}{\it hh}}}
\end {array} \right]
\ \ 
U^{+}_{\Y_{1/2}}(a)= \left[ \begin {array}{cc} -{\frac { \left( -{\it hh}+a \right) 
		\left( a+{\it hh} \right) }{\sqrt {a}{{\it hh}}^{3/2}}}&{\frac {
		\sqrt {a} \left( {\it hh}-1 \right)  \left( {\it hh}+1 \right) }{{{
				\it hh}}^{3/2}}}\\ \noalign{\medskip}0&-{\frac { \left( a+1 \right) 
		\left( a-1 \right) }{\sqrt {{\it hh}}{a}^{3/2}}}\end {array} \right]
\\
\\
U^{-}_{\Y_{0}}(a)=  \left[ \begin {array}{cc} {\frac { \left( a+{\it hh} \right)  \left( 
		-{\it hh}+a \right) ^{2}}{{{\it hh}}^{2}a}}&-{\frac { \left( {\it hh}+
		1 \right)  \left( {\it hh}-1 \right)  \left( -{\it hh}+a \right) }{{{
				\it hh}}^{2}a}}\\ \noalign{\medskip}0&{\frac { \left( -{\it hh}+a
		\right)  \left( a+1 \right)  \left( a-1 \right) }{{a}^{2}{\it hh}}}
\end {array} \right] 
\ \ 
U^{-}_{\Y_{1/2}}(a)=\left[ \begin {array}{cc} -{\frac { \left( -{\it hh}+a \right) 
		\left( a+{\it hh} \right) }{\sqrt {a}{{\it hh}}^{3/2}}}&{\frac {
		\left( {\it hh}-1 \right)  \left( {\it hh}+1 \right) }{{a}^{3/2}{{
				\it hh}}^{3/2}}}\\ \noalign{\medskip}0&-{\frac { \left( a+1 \right) 
		\left( a-1 \right) }{\sqrt {{\it hh}}{a}^{3/2}}}\end {array} \right] 
\end{array}
$$
By Lemma \ref{luskthlem} ${L}^{\pm }_{\Y_w}(a) = \fp \, {U}^{\pm }_{\Y_w}(a^{-1}) \, \fp$ and we compute:
$$
\hspace{-10mm}
\begin{array}{l}
L^{+}_{\Y_{0}}(a)=  \left[ \begin {array}{cc} {\frac { \left( a{\it hh}-1 \right) 
		\left( a+1 \right)  \left( a-1 \right) }{a{\it hh}}}&0
\\ \noalign{\medskip}{\frac { \left( a{\it hh}-1 \right)  \left( {\it 
			hh}-1 \right)  \left( {\it hh}+1 \right) }{{{\it hh}}^{2}}}&{\frac {
		\left( a{\it hh}-1 \right) ^{2} \left( a{\it hh}+1 \right) }{{a}^{2}{
			{\it hh}}^{2}}}\end {array} \right] 
\ \ 
L^{+}_{\Y_{1/2}}(a)=   \left[ \begin {array}{cc} {\frac { \left( a+1 \right)  \left( a-1
		\right) }{\sqrt {a}\sqrt {{\it hh}}}}&0\\ \noalign{\medskip}{\frac {
		\left( {\it hh}-1 \right)  \left( {\it hh}+1 \right) }{\sqrt {a}{{
				\it hh}}^{3/2}}}&{\frac { \left( a{\it hh}-1 \right)  \left( a{\it hh}
		+1 \right) }{{a}^{3/2}{{\it hh}}^{3/2}}}\end {array} \right]
\\
\\
L^{-}_{\Y_{0}}(a)=   \left[ \begin {array}{cc} {\frac { \left( a{\it hh}-1 \right) 
		\left( a+1 \right)  \left( a-1 \right) }{a{\it hh}}}&0
\\ \noalign{\medskip}{\frac { \left( a{\it hh}-1 \right)  \left( {\it 
			hh}-1 \right)  \left( {\it hh}+1 \right) }{{{\it hh}}^{2}}}&{\frac {
		\left( a{\it hh}-1 \right) ^{2} \left( a{\it hh}+1 \right) }{{a}^{2}{
			{\it hh}}^{2}}}\end {array} \right]
\ \ 
L^{-}_{\Y_{1/2}}(a)= \left[ \begin {array}{cc} {\frac { \left( a+1 \right)  \left( a-1
		\right) }{\sqrt {a}\sqrt {{\it hh}}}}&0\\ \noalign{\medskip}{\frac {
		\left( {\it hh}-1 \right)  \left( {\it hh}+1 \right) {a}^{3/2}}{{{
				\it hh}}^{3/2}}}&{\frac { \left( a{\it hh}-1 \right)  \left( a{\it hh}
		+1 \right) }{{a}^{3/2}{{\it hh}}^{3/2}}}\end {array} \right]
\end{array}
$$

\noindent
\underline{Case $n=3$}:
$$
\begin{array}{l}
U^{+}_{\Y_{0}}(a)= \\
\\
\left[ \begin {array}{ccc} -{\frac { \left( a+{\it hh} \right) 
		\left( -{\it hh}+a \right) ^{3} \left( {a}^{2}+a{\it hh}+{{\it hh}}^{
			2} \right) }{{a}^{3/2}{{\it hh}}^{9/2}}}&{\frac { \left( {\it hh}-1
		\right)  \left( {\it hh}+1 \right)  \left( -{\it hh}+a \right) ^{2}
		\left( a+{\it hh} \right) }{{{\it hh}}^{7/2}{a}^{3/2}}}&-{\frac {
		\left( a{{\it hh}}^{3}-1 \right)  \left( {\it hh}+1 \right)  \left( {
			\it hh}-1 \right)  \left( -{\it hh}+a \right) }{{a}^{5/2}{{\it hh}}^{9
			/2}}}\\ \noalign{\medskip}0&-{\frac { \left( -{\it hh}+a \right) ^{2}
		\left( {a}^{3}-{\it hh} \right) }{{{\it hh}}^{5/2}{a}^{5/2}}}&{\frac 
	{ \left( a{\it hh}+1 \right)  \left( -{\it hh}+a \right)  \left( a+1
		\right)  \left( a-1 \right)  \left( {\it hh}-1 \right)  \left( {\it 
			hh}+1 \right) }{{a}^{7/2}{{\it hh}}^{7/2}}}\\ \noalign{\medskip}0&0&-{
	\frac { \left( -{\it hh}+a \right)  \left( a+1 \right)  \left( a-1
		\right)  \left( {a}^{3}{\it hh}-1 \right) }{{{\it hh}}^{5/2}{a}^{9/2}
}}\end {array} \right]
\\
\\
U^{+}_{\Y_{1/2}}(a)= 
\left[ \begin {array}{ccc} -{\frac {\sqrt {a} \left( -{\it hh}+a
		\right)  \left( a+{\it hh} \right) }{{{\it hh}}^{5/2}}}&0&{\frac {
		\left( {\it hh}-1 \right)  \left( {\it hh}+1 \right) }{{{\it hh}}^{5/
			2}\sqrt {a}}}\\ \noalign{\medskip}0&{{\it hh}}^{-1}&0
\\ \noalign{\medskip}0&0&-{\frac { \left( a+1 \right)  \left( a-1
		\right) }{{{\it hh}}^{3/2}{a}^{5/2}}}\end {array} \right]
\\
\\
U^{+}_{\Y_{1/3}}(a)= 
\left[ \begin {array}{ccc} -{\frac { \left( -{\it hh}+a \right) 
		\left( {a}^{2}+a{\it hh}+{{\it hh}}^{2} \right) }{{{\it hh}}^{3}}}&{
	\frac { \left( {\it hh}-1 \right)  \left( {\it hh}+1 \right) {a}^{3/2}
	}{{{\it hh}}^{5/2}}}&{\frac { \left( {\it hh}-1 \right)  \left( {\it 
			hh}+1 \right) }{{{\it hh}}^{3}}}\\ \noalign{\medskip}0&-{\frac {{a}^{3
		}-{\it hh}}{{a}^{3/2}{{\it hh}}^{3/2}}}&{\frac { \left( {\it hh}-1
		\right)  \left( {\it hh}+1 \right) }{{{\it hh}}^{2}}}
\\ \noalign{\medskip}0&0&-{\frac {{a}^{3}{\it hh}-1}{{a}^{3}{{\it hh}}
		^{2}}}\end {array} \right]
\end{array}
$$

$$
\begin{array}{l}
U^{-}_{\Y_{0}}(a)= \\
\\
\left[ \begin {array}{ccc} -{\frac { \left( a+{\it hh} \right) 
		\left( -{\it hh}+a \right) ^{3} \left( {a}^{2}+a{\it hh}+{{\it hh}}^{
			2} \right) }{{a}^{3/2}{{\it hh}}^{9/2}}}&{\frac { \left( {\it hh}-1
		\right)  \left( {\it hh}+1 \right)  \left( -{\it hh}+a \right) ^{2}
		\left( a+{\it hh} \right) }{{{\it hh}}^{7/2}{a}^{3/2}}}&-{\frac {
		\left( a{{\it hh}}^{3}-1 \right)  \left( {\it hh}+1 \right)  \left( {
			\it hh}-1 \right)  \left( -{\it hh}+a \right) }{{a}^{5/2}{{\it hh}}^{9
			/2}}}\\ \noalign{\medskip}0&-{\frac { \left( -{\it hh}+a \right) ^{2}
		\left( {a}^{3}-{\it hh} \right) }{{{\it hh}}^{5/2}{a}^{5/2}}}&{\frac 
	{ \left( a{\it hh}+1 \right)  \left( -{\it hh}+a \right)  \left( a+1
		\right)  \left( a-1 \right)  \left( {\it hh}-1 \right)  \left( {\it 
			hh}+1 \right) }{{a}^{7/2}{{\it hh}}^{7/2}}}\\ \noalign{\medskip}0&0&-{
	\frac { \left( -{\it hh}+a \right)  \left( a+1 \right)  \left( a-1
		\right)  \left( {a}^{3}{\it hh}-1 \right) }{{{\it hh}}^{5/2}{a}^{9/2}
}}\end {array} \right] 
\\
\\
U^{-}_{\Y_{1/2}}(a)= \left[ \begin {array}{ccc} -{\frac {\sqrt {a} \left( -{\it hh}+a
		\right)  \left( a+{\it hh} \right) }{{{\it hh}}^{5/2}}}&0&{\frac {
		\left( {\it hh}-1 \right)  \left( {\it hh}+1 \right) }{{{\it hh}}^{5/
			2}{a}^{5/2}}}\\ \noalign{\medskip}0&{{\it hh}}^{-1}&0
\\ \noalign{\medskip}0&0&-{\frac { \left( a+1 \right)  \left( a-1
		\right) }{{{\it hh}}^{3/2}{a}^{5/2}}}\end {array} \right] 
\\
\\
U^{-}_{\Y_{1/3}}(a)= \left[ \begin {array}{ccc} -{\frac { \left( -{\it hh}+a \right) 
		\left( {a}^{2}+a{\it hh}+{{\it hh}}^{2} \right) }{{{\it hh}}^{3}}}&{
	\frac { \left( {\it hh}-1 \right)  \left( {\it hh}+1 \right) }{{a}^{3/
			2}{{\it hh}}^{3/2}}}&{\frac { \left( {\it hh}-1 \right)  \left( {\it 
			hh}+1 \right) }{{a}^{3}{{\it hh}}^{2}}}\\ \noalign{\medskip}0&-{\frac 
	{{a}^{3}-{\it hh}}{{a}^{3/2}{{\it hh}}^{3/2}}}&{\frac { \left( {\it hh
		}-1 \right)  \left( {\it hh}+1 \right) }{{a}^{3}{{\it hh}}^{3}}}
\\ \noalign{\medskip}0&0&-{\frac {{a}^{3}{\it hh}-1}{{a}^{3}{{\it hh}}
		^{2}}}\end {array} \right] 
\end{array}
$$
By Lemma \ref{luskthlem}, ${L}^{\pm }_{\Y_w}(a) = \fp \, {U}^{\pm }_{\Y_w}(a^{-1}) \, \fp$ and the matrices ${L}^{\pm }_{\Y_w}(a)$ are easy to compute explicitly using (\ref{smat3}).

If $w=a/b, w'=a'/b'$ with the same denominator $b=b'$ then ${\Y_w}={\Y_{w'}}$. Thus, in cases 
$n=2$ and $n=3$ the matrices ${U}^{\pm }_{\Y_w}(a)$, ${L}^{\pm }_{\Y_w}(a)$ are equal to one of listed above. 

\section{Twisted R-matrices for $\Y_w$  \label{twisrsect} }

The R-matrices of $\Y_w$ are defined by (\ref{RmatForY}). Using explicit matrices from the previous section we compute:

\noindent
$\underline{n=2}$:
$$
R^{+}_{\Y_0}(a)=\left[ \begin {array}{cc} {\frac { \left( -h+a \right)  \left( a+1
		\right)  \left( a-1 \right) h}{ \left( ah-1 \right) ^{2} \left( ah+1
		\right) }}&-{\frac { \left( h+1 \right)  \left( h-1 \right)  \left( -
		h+a \right) a}{ \left( ah-1 \right) ^{2} \left( ah+1 \right) }}
\\ \noalign{\medskip}-{\frac { \left( h+1 \right)  \left( h-1 \right) 
		\left( -h+a \right) a}{ \left( ah-1 \right) ^{2} \left( ah+1 \right) 
}}&{\frac { \left( -h+a \right)  \left( a+1 \right)  \left( a-1
		\right) h}{ \left( ah-1 \right) ^{2} \left( ah+1 \right) }}
\end {array} \right]
$$
$$
R^{-}_{\Y_0}(a)= \left[ \begin {array}{cc} {\frac { \left( -h+a \right)  \left( a+1
		\right)  \left( a-1 \right) h}{ \left( ah-1 \right) ^{2} \left( ah+1
		\right) }}&-{\frac { \left( h+1 \right)  \left( h-1 \right)  \left( -
		h+a \right) a}{ \left( ah-1 \right) ^{2} \left( ah+1 \right) }}
\\ \noalign{\medskip}-{\frac { \left( h+1 \right)  \left( h-1 \right) 
		\left( -h+a \right) a}{ \left( ah-1 \right) ^{2} \left( ah+1 \right) 
}}&{\frac { \left( -h+a \right)  \left( a+1 \right)  \left( a-1
		\right) h}{ \left( ah-1 \right) ^{2} \left( ah+1 \right) }}
\end {array} \right]
$$
$$
R^{+}_{\Y_{1/2}}(a)= \left[ \begin {array}{cc} -{\frac { \left( a+1 \right)  \left( a-1
		\right) h}{ \left( ah-1 \right)  \left( ah+1 \right) }}&{\frac {{a}^{
			2} \left( h-1 \right)  \left( h+1 \right) }{ \left( ah-1 \right) 
		\left( ah+1 \right) }}\\ \noalign{\medskip}{\frac { \left( h+1
		\right)  \left( h-1 \right) }{ \left( ah-1 \right)  \left( ah+1
		\right) }}&-{\frac { \left( a+1 \right)  \left( a-1 \right) h}{
		\left( ah-1 \right)  \left( ah+1 \right) }}\end {array} \right] 
$$
$$
R^{-}_{\Y_{1/2}}(a)=\left[ \begin {array}{cc} -{\frac { \left( a+1 \right)  \left( a-1
		\right) h}{ \left( ah-1 \right)  \left( ah+1 \right) }}&{\frac {
		\left( h+1 \right)  \left( h-1 \right) }{ \left( ah-1 \right) 
		\left( ah+1 \right) }}\\ \noalign{\medskip}{\frac {{a}^{2} \left( h-1
		\right)  \left( h+1 \right) }{ \left( ah-1 \right)  \left( ah+1
		\right) }}&-{\frac { \left( a+1 \right)  \left( a-1 \right) h}{
		\left( ah-1 \right)  \left( ah+1 \right) }}\end {array} \right] 
$$

\noindent
$\underline{n=3}$:

$$
\hspace{-15mm}
\begin{array}{l}
R^{+}_{\Y_{0}}(a)=\\
\\
\left[ \begin {array}{ccc} -{\frac { \left( a+1 \right)  \left( a-1
		\right)  \left( {a}^{3}h-1 \right)  \left( -h+a \right) {h}^{2}}{
		\left( ah-1 \right) ^{3} \left( ah+1 \right)  \left( {a}^{2}{h}^{2}+a
		h+1 \right) }}&{\frac { \left( h+1 \right)  \left( h-1 \right) 
		\left( -h+a \right)  \left( a-1 \right)  \left( a+1 \right) ah}{
		\left( ah-1 \right) ^{3} \left( {a}^{2}{h}^{2}+ah+1 \right) }}&-{
	\frac { \left( -h+a \right)  \left( h-1 \right)  \left( h+1 \right) 
		\left( a{h}^{3}-1 \right) {a}^{2}}{ \left( ah-1 \right) ^{3} \left( a
		h+1 \right)  \left( {a}^{2}{h}^{2}+ah+1 \right) }}
\\ \noalign{\medskip}{\frac { \left( h+1 \right)  \left( h-1 \right) 
		\left( -h+a \right)  \left( a-1 \right)  \left( a+1 \right) ah}{
		\left( ah-1 \right) ^{3} \left( {a}^{2}{h}^{2}+ah+1 \right) }}&-{
	\frac { \left( {a}^{4}{h}^{2}-{a}^{3}{h}^{3}+{a}^{2}{h}^{4}-2\,{a}^{2}
		{h}^{2}+{a}^{2}-ah+{h}^{2} \right)  \left( -h+a \right) }{ \left( ah-1
		\right) ^{3} \left( {a}^{2}{h}^{2}+ah+1 \right) }}&{\frac { \left( h+
		1 \right)  \left( h-1 \right)  \left( -h+a \right)  \left( a-1
		\right)  \left( a+1 \right) ah}{ \left( ah-1 \right) ^{3} \left( {a}^
		{2}{h}^{2}+ah+1 \right) }}\\ \noalign{\medskip}-{\frac { \left( -h+a
		\right)  \left( h-1 \right)  \left( h+1 \right)  \left( a{h}^{3}-1
		\right) {a}^{2}}{ \left( ah-1 \right) ^{3} \left( ah+1 \right) 
		\left( {a}^{2}{h}^{2}+ah+1 \right) }}&{\frac { \left( h+1 \right) 
		\left( h-1 \right)  \left( -h+a \right)  \left( a-1 \right)  \left( a
		+1 \right) ah}{ \left( ah-1 \right) ^{3} \left( {a}^{2}{h}^{2}+ah+1
		\right) }}&-{\frac { \left( a+1 \right)  \left( a-1 \right)  \left( {
			a}^{3}h-1 \right)  \left( -h+a \right) {h}^{2}}{ \left( ah-1 \right) ^
		{3} \left( ah+1 \right)  \left( {a}^{2}{h}^{2}+ah+1 \right) }}
\end {array} \right]
\end{array}
$$

$$
R^{+}_{\Y_{1/3}}(a)=  \left[ \begin {array}{ccc} -{\frac {h \left( {a}^{3}h-1 \right) }{
		\left( ah-1 \right)  \left( {a}^{2}{h}^{2}+ah+1 \right) }}&{\frac {
		\left( h+1 \right)  \left( h-1 \right) h{a}^{3}}{ \left( ah-1
		\right)  \left( {a}^{2}{h}^{2}+ah+1 \right) }}&{\frac { \left( h+1
		\right)  \left( h-1 \right) {a}^{3}}{ \left( ah-1 \right)  \left( {a}
		^{2}{h}^{2}+ah+1 \right) }}\\ \noalign{\medskip}{\frac { \left( h+1
		\right)  \left( h-1 \right) }{ \left( ah-1 \right)  \left( {a}^{2}{h}
		^{2}+ah+1 \right) }}&-{\frac {h \left( {a}^{3}h-1 \right) }{ \left( ah
		-1 \right)  \left( {a}^{2}{h}^{2}+ah+1 \right) }}&{\frac { \left( h+1
		\right)  \left( h-1 \right) h{a}^{3}}{ \left( ah-1 \right)  \left( {a
		}^{2}{h}^{2}+ah+1 \right) }}\\ \noalign{\medskip}{\frac { \left( h+1
		\right)  \left( h-1 \right) h}{ \left( ah-1 \right)  \left( {a}^{2}{h
		}^{2}+ah+1 \right) }}&{\frac { \left( h+1 \right)  \left( h-1 \right) 
	}{ \left( ah-1 \right)  \left( {a}^{2}{h}^{2}+ah+1 \right) }}&-{\frac 
	{h \left( {a}^{3}h-1 \right) }{ \left( ah-1 \right)  \left( {a}^{2}{h}
		^{2}+ah+1 \right) }}\end {array} \right] 
$$

$$
R^{+}_{\Y_{1/2}}(a)=  \left[ \begin {array}{ccc} -{\frac { \left( a+1 \right)  \left( a-1
		\right) h}{ \left( ah-1 \right)  \left( ah+1 \right) }}&0&{\frac {{a}
		^{2} \left( h-1 \right)  \left( h+1 \right) }{ \left( ah-1 \right) 
		\left( ah+1 \right) }}\\ \noalign{\medskip}0&1&0\\ \noalign{\medskip}
{\frac { \left( h+1 \right)  \left( h-1 \right) }{ \left( ah-1
		\right)  \left( ah+1 \right) }}&0&-{\frac { \left( a+1 \right) 
		\left( a-1 \right) h}{ \left( ah-1 \right)  \left( ah+1 \right) }}
\end {array} \right]
$$

$$
\hspace{-15mm}
\begin{array}{l}
R^{-}_{\Y_{0}}(a)=\\
\\
\left[ \begin {array}{ccc} -{\frac { \left( a+1 \right)  \left( a-1
		\right)  \left( {a}^{3}h-1 \right)  \left( -h+a \right) {h}^{2}}{
		\left( ah-1 \right) ^{3} \left( ah+1 \right)  \left( {a}^{2}{h}^{2}+a
		h+1 \right) }}&{\frac { \left( h+1 \right)  \left( h-1 \right) 
		\left( -h+a \right)  \left( a-1 \right)  \left( a+1 \right) ah}{
		\left( ah-1 \right) ^{3} \left( {a}^{2}{h}^{2}+ah+1 \right) }}&-{
	\frac { \left( -h+a \right)  \left( h-1 \right)  \left( h+1 \right) 
		\left( a{h}^{3}-1 \right) {a}^{2}}{ \left( ah-1 \right) ^{3} \left( a
		h+1 \right)  \left( {a}^{2}{h}^{2}+ah+1 \right) }}
\\ \noalign{\medskip}{\frac { \left( h+1 \right)  \left( h-1 \right) 
		\left( -h+a \right)  \left( a-1 \right)  \left( a+1 \right) ah}{
		\left( ah-1 \right) ^{3} \left( {a}^{2}{h}^{2}+ah+1 \right) }}&-{
	\frac { \left( {a}^{4}{h}^{2}-{a}^{3}{h}^{3}+{a}^{2}{h}^{4}-2\,{a}^{2}
		{h}^{2}+{a}^{2}-ah+{h}^{2} \right)  \left( -h+a \right) }{ \left( ah-1
		\right) ^{3} \left( {a}^{2}{h}^{2}+ah+1 \right) }}&{\frac { \left( h+
		1 \right)  \left( h-1 \right)  \left( -h+a \right)  \left( a-1
		\right)  \left( a+1 \right) ah}{ \left( ah-1 \right) ^{3} \left( {a}^
		{2}{h}^{2}+ah+1 \right) }}\\ \noalign{\medskip}-{\frac { \left( -h+a
		\right)  \left( h-1 \right)  \left( h+1 \right)  \left( a{h}^{3}-1
		\right) {a}^{2}}{ \left( ah-1 \right) ^{3} \left( ah+1 \right) 
		\left( {a}^{2}{h}^{2}+ah+1 \right) }}&{\frac { \left( h+1 \right) 
		\left( h-1 \right)  \left( -h+a \right)  \left( a-1 \right)  \left( a
		+1 \right) ah}{ \left( ah-1 \right) ^{3} \left( {a}^{2}{h}^{2}+ah+1
		\right) }}&-{\frac { \left( a+1 \right)  \left( a-1 \right)  \left( {
			a}^{3}h-1 \right)  \left( -h+a \right) {h}^{2}}{ \left( ah-1 \right) ^
		{3} \left( ah+1 \right)  \left( {a}^{2}{h}^{2}+ah+1 \right) }}
\end {array} \right]
\end{array}
$$

$$
R^{-}_{\Y_{1/3}}(a)=\left[ \begin {array}{ccc} -{\frac {h \left( {a}^{3}h-1 \right) }{
		\left( ah-1 \right)  \left( {a}^{2}{h}^{2}+ah+1 \right) }}&{\frac {
		\left( h+1 \right)  \left( h-1 \right) }{ \left( ah-1 \right) 
		\left( {a}^{2}{h}^{2}+ah+1 \right) }}&{\frac { \left( h+1 \right) 
		\left( h-1 \right) h}{ \left( ah-1 \right)  \left( {a}^{2}{h}^{2}+ah+
		1 \right) }}\\ \noalign{\medskip}{\frac { \left( h+1 \right)  \left( h
		-1 \right) h{a}^{3}}{ \left( ah-1 \right)  \left( {a}^{2}{h}^{2}+ah+1
		\right) }}&-{\frac {h \left( {a}^{3}h-1 \right) }{ \left( ah-1
		\right)  \left( {a}^{2}{h}^{2}+ah+1 \right) }}&{\frac { \left( h+1
		\right)  \left( h-1 \right) }{ \left( ah-1 \right)  \left( {a}^{2}{h}
		^{2}+ah+1 \right) }}\\ \noalign{\medskip}{\frac { \left( h+1 \right) 
		\left( h-1 \right) {a}^{3}}{ \left( ah-1 \right)  \left( {a}^{2}{h}^{
			2}+ah+1 \right) }}&{\frac { \left( h+1 \right)  \left( h-1 \right) h{a
		}^{3}}{ \left( ah-1 \right)  \left( {a}^{2}{h}^{2}+ah+1 \right) }}&-{
	\frac {h \left( {a}^{3}h-1 \right) }{ \left( ah-1 \right)  \left( {a}^
		{2}{h}^{2}+ah+1 \right) }}\end {array} \right] 
$$

$$
R^{-}_{\Y_{1/2}}(a)= \left[ \begin {array}{ccc} -{\frac { \left( a+1 \right)  \left( a-1
		\right) h}{ \left( ah-1 \right)  \left( ah+1 \right) }}&0&{\frac {
		\left( h+1 \right)  \left( h-1 \right) }{ \left( ah-1 \right) 
		\left( ah+1 \right) }}\\ \noalign{\medskip}0&1&0\\ \noalign{\medskip}
{\frac {{a}^{2} \left( h-1 \right)  \left( h+1 \right) }{ \left( ah-1
		\right)  \left( ah+1 \right) }}&0&-{\frac { \left( a+1 \right) 
		\left( a-1 \right) h}{ \left( ah-1 \right)  \left( ah+1 \right) }}
\end {array} \right]
$$
Using (\ref{RmatForY}) one also computes similar matrices $R^{+}_{\Y_{1/2}}(a)$. It might be instructive to check Proposition \ref{unitlem} which in this case claims that 
\be \label{unitarY}
R^{+}_{\Y_{1/2}}(a^{-1}) R^{-}_{\Y_{1/2}}(a)=1
\ee

The matrix $\gamma_{w}(a)$  defined by (\ref{gamm}). Here are the first examples:
$$
\gamma_{0}(a)= \left[ \begin {array}{cc} -{h}^{-1}&0\\ \noalign{\medskip}0&-{h}^{-1}
\end {array} \right], \ \ \   \gamma_{1/2}(a)= \left[ \begin {array}{cc} {\frac {\sqrt {-a}}{{h}^{2}}}&0
\\ \noalign{\medskip}0&{\frac {1}{{h}^{2}\sqrt {-a}}}\end {array}
\right] 
$$

$$
\gamma_{0}(a)=\left[ \begin {array}{ccc}  \left( -{h}^{-1} \right) ^{3/2}&0&0
\\ \noalign{\medskip}0&-{\frac {1}{h}\sqrt {-{h}^{-1}}}&0
\\ \noalign{\medskip}0&0& \left( -{h}^{-1} \right) ^{3/2}\end {array}
\right], 
$$
$$
\gamma_{1/3}(a)=\left[ \begin {array}{ccc} - \left( -{h}^{-1} \right) ^{5/2}a&0&0
\\ \noalign{\medskip}0&{\frac {1}{{h}^{2}}\sqrt {-{h}^{-1}}}&0
\\ \noalign{\medskip}0&0&-{\frac {1}{a} \left( -{h}^{-1} \right) ^{5/2
}}\end {array} \right]
$$

$$
\gamma_{1/2}(a)=\left[ \begin {array}{ccc} - \left( -{h}^{-1} \right) ^{7/2}a \sqrt {-a
}&0&0\\ \noalign{\medskip}0&{\frac {1}{{h}^{2}}\sqrt {-{h}^{-1}}}&0
\\ \noalign{\medskip}0&0&-{\frac {1}{a\sqrt {-a}} \left( -{h}^{-1}
	\right) ^{7/2}}\end {array} \right] 
$$
$$
\gamma_{2/3}(a)= \left[ \begin {array}{ccc}  \left( -{h}^{-1} \right) ^{9/2}{a}^{2}&0&0
\\ \noalign{\medskip}0&-{\frac {1}{{h}^{3}}\sqrt {-{h}^{-1}}}&0
\\ \noalign{\medskip}0&0&{\frac {1}{{a}^{2}} \left( -{h}^{-1} \right) 
	^{9/2}}\end {array} \right]
$$
Let $\kappa^{\ast}$ denote the substitution (\ref{mirsym}): 
$$
\kappa^{*}: \ \ \  a\to z h, \ \ h\to {h}^{-1},\ \ z\to a h 
$$
The twisted $R$-matrices are defined by (\ref{twistedR}). Using the above explicit formulas we find the first several examples:

\noindent
$\underline{n=2}$:
$$
{\bf R}^{-}_{\Y_{0}}(z)=   \left[ \begin {array}{cc} {\frac { \left( zh+1 \right)  \left( zh-1
		\right)  \left( {h}^{2}z-1 \right) }{{h}^{2} \left( z-1 \right) ^{2}
		\left( z+1 \right) }}&{\frac {z \left( h+1 \right)  \left( h-1
		\right)  \left( {h}^{2}z-1 \right) }{{h}^{2} \left( z-1 \right) ^{2}
		\left( z+1 \right) }}\\ \noalign{\medskip}{\frac {z \left( h+1
		\right)  \left( h-1 \right)  \left( {h}^{2}z-1 \right) }{{h}^{2}
		\left( z-1 \right) ^{2} \left( z+1 \right) }}&{\frac { \left( zh+1
		\right)  \left( zh-1 \right)  \left( {h}^{2}z-1 \right) }{{h}^{2}
		\left( z-1 \right) ^{2} \left( z+1 \right) }}\end {array} \right] 
$$

$$
{\bf R}^{-}_{\Y_{1/2}}(z)=   \left[ \begin {array}{cc} -{\frac {{h}^{2}{z}^{2}-1}{h \left( {z}^{2}
		-1 \right) }}&{\frac {a \left( {h}^{2}-1 \right) }{{h}^{2} \left( {z}^
		{2}-1 \right) }}\\ \noalign{\medskip}{\frac {{z}^{2} \left( {h}^{2}-1
		\right) }{a \left( {z}^{2}-1 \right) }}&-{\frac {{h}^{2}{z}^{2}-1}{h
		\left( {z}^{2}-1 \right) }}\end {array} \right] 
$$

\noindent
$\underline{n=3}$:
$$
\hspace{-20mm}
\begin{array}{l}
{\bf R}^{-}_{\Y_{0}}(z)= \\ 
\\
\left[ \begin {array}{ccc} -{\frac { \left( {h}^{2}{z}^{3}-1 \right) 
		\left( {h}^{2}z-1 \right)  \left( zh-1 \right)  \left( zh+1 \right) 
	}{{h}^{3} \left( z-1 \right) ^{3} \left( z+1 \right)  \left( {z}^{2}+z
		+1 \right) }}&-{\frac {z \left( zh+1 \right)  \left( zh-1 \right) 
		\left( h+1 \right)  \left( h-1 \right)  \left( {h}^{2}z-1 \right) }{{
			h}^{3} \left( z-1 \right) ^{3} \left( {z}^{2}+z+1 \right) }}&-{\frac {
		\left( h-1 \right)  \left( h+1 \right)  \left( {h}^{2}z-1 \right) 
		\left( {h}^{2}-z \right) {z}^{2}}{{h}^{3} \left( z-1 \right) ^{3}
		\left( z+1 \right)  \left( {z}^{2}+z+1 \right) }}
\\ \noalign{\medskip}-{\frac {z \left( zh+1 \right)  \left( zh-1
		\right)  \left( h+1 \right)  \left( h-1 \right)  \left( {h}^{2}z-1
		\right) }{{h}^{3} \left( z-1 \right) ^{3} \left( {z}^{2}+z+1 \right) 
}}&-{\frac { \left( {h}^{4}{z}^{4}+{h}^{4}{z}^{2}-{h}^{2}{z}^{3}-2\,{h
		}^{2}{z}^{2}-{h}^{2}z+{z}^{2}+1 \right)  \left( {h}^{2}z-1 \right) }{{
			h}^{3} \left( z-1 \right) ^{3} \left( {z}^{2}+z+1 \right) }}&-{\frac {
		z \left( zh+1 \right)  \left( zh-1 \right)  \left( h+1 \right) 
		\left( h-1 \right)  \left( {h}^{2}z-1 \right) }{{h}^{3} \left( z-1
		\right) ^{3} \left( {z}^{2}+z+1 \right) }}\\ \noalign{\medskip}-{
	\frac { \left( h-1 \right)  \left( h+1 \right)  \left( {h}^{2}z-1
		\right)  \left( {h}^{2}-z \right) {z}^{2}}{{h}^{3} \left( z-1
		\right) ^{3} \left( z+1 \right)  \left( {z}^{2}+z+1 \right) }}&-{
	\frac {z \left( zh+1 \right)  \left( zh-1 \right)  \left( h+1 \right) 
		\left( h-1 \right)  \left( {h}^{2}z-1 \right) }{{h}^{3} \left( z-1
		\right) ^{3} \left( {z}^{2}+z+1 \right) }}&-{\frac { \left( {h}^{2}{z
		}^{3}-1 \right)  \left( {h}^{2}z-1 \right)  \left( zh-1 \right) 
		\left( zh+1 \right) }{{h}^{3} \left( z-1 \right) ^{3} \left( z+1
		\right)  \left( {z}^{2}+z+1 \right) }}\end {array} \right] 
\end{array}
$$

$$ \hspace{-30mm}
{\bf R}^{-}_{\Y_{1/3}}(z)=   \left[ \begin {array}{ccc} -{\frac {{h}^{2}{z}^{3}-1}{h \left( {z}^{3
		}-1 \right) }}&{\frac {a \left( {h}^{2}-1 \right) }{{h}^{2} \left( {z}
		^{3}-1 \right) }}&-{\frac {{a}^{2} \left( {h}^{2}-1 \right) }{{h}^{3}
		\left( {z}^{3}-1 \right) }}\\ \noalign{\medskip}{\frac {{z}^{3}
		\left( {h}^{2}-1 \right) }{a \left( {z}^{3}-1 \right) }}&-{\frac {{h}
		^{2}{z}^{3}-1}{h \left( {z}^{3}-1 \right) }}&{\frac {a \left( {h}^{2}-
		1 \right) }{{h}^{2} \left( {z}^{3}-1 \right) }}\\ \noalign{\medskip}-{
	\frac {h{z}^{3} \left( {h}^{2}-1 \right) }{{a}^{2} \left( {z}^{3}-1
		\right) }}&{\frac {{z}^{3} \left( {h}^{2}-1 \right) }{a \left( {z}^{3
		}-1 \right) }}&-{\frac {{h}^{2}{z}^{3}-1}{h \left( {z}^{3}-1 \right) }
}\end {array} \right]
$$

$$
{\bf R}^{-}_{\Y_{1/2}}(z)=  \left[ \begin {array}{ccc} -{\frac {{h}^{2}{z}^{2}-1}{h \left( {z}^{2
		}-1 \right) }}&0&{\frac {{a}^{3} \left( {h}^{2}-1 \right) }{{h}^{2}
		\left( {z}^{2}-1 \right) }}\\ \noalign{\medskip}0&1&0
\\ \noalign{\medskip}{\frac {{z}^{2} \left( {h}^{2}-1 \right) }{{a}^{3
		} \left( {z}^{2}-1 \right) }}&0&-{\frac {{h}^{2}{z}^{2}-1}{h \left( {z
		}^{2}-1 \right) }}\end {array} \right]
$$

$$
{\bf R}^{-}_{\Y_{2/3}}(z)=  \left[ \begin {array}{ccc} -{\frac {{h}^{2}{z}^{3}-1}{h \left( {z}^{3
		}-1 \right) }}&{\frac {{a}^{2} \left( {h}^{2}-1 \right) }{{h}^{3}
		\left( {z}^{3}-1 \right) }}&-{\frac {{a}^{4} \left( {h}^{2}-1
		\right) }{{h}^{3} \left( {z}^{3}-1 \right) }}\\ \noalign{\medskip}{
	\frac {h{z}^{3} \left( {h}^{2}-1 \right) }{{a}^{2} \left( {z}^{3}-1
		\right) }}&-{\frac {{h}^{2}{z}^{3}-1}{h \left( {z}^{3}-1 \right) }}&{
	\frac {{a}^{2} \left( {h}^{2}-1 \right) }{h \left( {z}^{3}-1 \right) }
}\\ \noalign{\medskip}-{\frac {h{z}^{3} \left( {h}^{2}-1 \right) }{{a}
		^{4} \left( {z}^{3}-1 \right) }}&{\frac {{z}^{3} \left( {h}^{2}-1
		\right) }{h{a}^{2} \left( {z}^{3}-1 \right) }}&-{\frac {{h}^{2}{z}^{3
		}-1}{h \left( {z}^{3}-1 \right) }}\end {array} \right]
$$
Note that the properties of these matrices are in full agreement with Proposition \ref{monprop}. In particular, their matrix elements are monomials in $a$.

\section{Monodromy and wall-crossing operators}
By Theorem \ref{brmatth}
$$
\overset{\curvearrowright}{\B}_{w}(z)=\hbar^{\Omega_{w}} \, {\bf R}^{-}_{\Y_{w}}(z). 
$$
the twisted $R$-matrices ${\bf R}^{-}_{\Y_{w}}(z)$ are computed in the previous section and the operators $\hbar^{\Omega_{w}}$ are given by (\ref{homegaop}). We find:

\noindent
$n=2$:
$$
\hbar^{\Omega_{0}}=\left[ \begin {array}{cc} {h}^{2}&0\\ \noalign{\medskip}0&{h}^{2}
\end {array} \right], \ \ \ 
\hbar^{\Omega_{1/2}}=\left[ \begin {array}{cc} -h&0\\ \noalign{\medskip}0&-h\end {array}
\right]
$$

\noindent
$n=3$:
$$
\begin{array}{l}
\hbar^{\Omega_{0}}=\left[ \begin {array}{ccc} -{h}^{3}&0&0\\ \noalign{\medskip}0&-{h}^{3
}&0\\ \noalign{\medskip}0&0&-{h}^{3}\end {array} \right],  \ \ 
\hbar^{\Omega_{1/3}}= \left[ \begin {array}{ccc} -h&0&0\\ \noalign{\medskip}0&-h&0
\\ \noalign{\medskip}0&0&-h\end {array} \right], 
\\
\\
\hbar^{\Omega_{1/2}}=  \left[ \begin {array}{ccc} -h&0&0\\ \noalign{\medskip}0&1&0
\\ \noalign{\medskip}0&0&-h\end {array} \right], 
\ \ \
\hbar^{\Omega_{2/3}}=   \left[ \begin {array}{ccc} -h&0&0\\ \noalign{\medskip}0&-h&0
\\ \noalign{\medskip}0&0&-h\end {array} \right]. 
\end{array}
$$
The wall crossing operators in the basis of fixed points (i.e.,  Macdonald polynomials) can be computed  using explicit matrices from Appendix \ref{stabxapp} by $$
{\B}^{\bullet}_{w}(z)={\bf U}^{-}_{w}(a)^{-1} \hbar^{\Omega_{w}} \, {\bf R}^{-}_{\Y_{w}}(z)\, {\bf U}^{+}_{w}(a)$$
For instance, in the case $n=2$ we find:
$$ \hspace{-5mm}
\begin{array}{l}
{\B}^{\bullet}_{0}(z)= \\
\\ \left[ \begin {array}{cc} {\frac { \left( {h}^{2}z-1 \right)  \left( 
		{a}^{2}{h}^{3}{z}^{2}+a{h}^{4}z-{h}^{3}{z}^{2}-2\,a{h}^{2}z-{a}^{2}h+a
		z+h \right) }{h \left( z-1 \right) ^{2} \left( a+1 \right)  \left( a-1
		\right)  \left( z+1 \right) }}&{\frac { \left( h-1 \right)  \left( h+
		1 \right)  \left( ah+1 \right)  \left( ah-1 \right)  \left( {h}^{2}z-1
		\right) z}{ah \left( z-1 \right) ^{2} \left( a+1 \right)  \left( a-1
		\right)  \left( z+1 \right) }}\\ \noalign{\medskip}{\frac {a \left( h
		-1 \right)  \left( h+1 \right)  \left( -h+a \right)  \left( a+h
		\right)  \left( {h}^{2}z-1 \right) z}{h \left( z-1 \right) ^{2}
		\left( a+1 \right)  \left( a-1 \right)  \left( z+1 \right) }}&{\frac 
	{ \left( {h}^{2}z-1 \right)  \left( {a}^{2}{h}^{3}{z}^{2}-a{h}^{4}z-{h
		}^{3}{z}^{2}+2\,a{h}^{2}z-{a}^{2}h-az+h \right) }{h \left( z-1
		\right) ^{2} \left( a+1 \right)  \left( a-1 \right)  \left( z+1
		\right) }}\end {array} \right]
\end{array}
$$

$$
{\B}^{\bullet}_{1/2}(z)=  \left[ \begin {array}{cc} {\frac {{a}^{2}{h}^{2}{z}^{2}-{h}^{4}{z}^{2
		}+{h}^{2}{z}^{2}-{a}^{2}-{z}^{2}+1}{ \left( z-1 \right)  \left( z+1
		\right)  \left( a+1 \right)  \left( a-1 \right) }}&-{\frac { \left( a
		h-1 \right)  \left( h-1 \right)  \left( h+1 \right) {z}^{2} \left( ah+
		1 \right) }{ \left( z-1 \right)  \left( z+1 \right)  \left( a+1
		\right)  \left( a-1 \right) }}\\ \noalign{\medskip}-{\frac { \left( -
		h+a \right)  \left( h-1 \right)  \left( h+1 \right) {z}^{2} \left( a+h
		\right) }{ \left( z-1 \right)  \left( z+1 \right)  \left( a+1
		\right)  \left( a-1 \right) }}&{\frac {{a}^{2}{h}^{4}{z}^{2}-{a}^{2}{
			h}^{2}{z}^{2}+{z}^{2}{a}^{2}-{h}^{2}{z}^{2}-{a}^{2}+1}{ \left( z-1
		\right)  \left( z+1 \right)  \left( a+1 \right)  \left( a-1 \right) }
}\end {array} \right] 
$$
For $n=3$ these matrices are already too large to print  here. 

\section{Representation of the fundamental group \label{monapp}}
For $n=2$, the fundamental group $\pi_{1}(\mathbb{P}^{1}\setminus\sing,0_{+})$ is generated by 
$\gamma_{0},\gamma_{1/2},\gamma_{\infty}$ subject to a relation
$
\gamma_{0} \gamma_{1/2}=\gamma_{\infty}.
$
By Theorem \ref{nomodtheor} 
$$
\nabla_{\DT}(\gamma_{0})=\textsf{E}_{0}, \ \
\nabla_{\DT}(\gamma_{\infty})=\T^{-1} \textsf{E}_{\infty} \T, \ \ \nabla_{\DT}(\gamma_{1/2})={\B}^{\bullet}_{1/2}
$$
and thus, we need to check that
\be \label{reln2}
{\B}^{\bullet}_{1/2} \textsf{E}_{0} = \T^{-1} \textsf{E}_{\infty} \T.
\ee
From definitions, we compute
$$
\textsf{E}_{0}= \left[ \begin {array}{cc} {\frac {a}{h}}&0\\ \noalign{\medskip}0&{
	\frac {1}{ah}}\end {array} \right], \ \ \ \textsf{E}_{\infty} = \left[ \begin {array}{cc} ah&0\\ \noalign{\medskip}0&{\frac {h}{a}}
\end {array} \right],
$$
${\B}^{\bullet}_{w}:={\B}^{\bullet}_{1/2}(\infty)$ and thus from the matrix in the previous section we find:
$$
{\B}^{\bullet}_{1/2}= \left[ \begin {array}{cc} {\frac {{a}^{2}{h}^{2}-{h}^{4}+{h}^{2}-1}{
		\left( a+1 \right)  \left( a-1 \right) }}&-{\frac { \left( ah-1
		\right)  \left( h-1 \right)  \left( h+1 \right)  \left( ah+1 \right) 
	}{ \left( a+1 \right)  \left( a-1 \right) }}\\ \noalign{\medskip}-{
	\frac { \left( -h+a \right)  \left( h-1 \right)  \left( h+1 \right) 
		\left( a+h \right) }{ \left( a+1 \right)  \left( a-1 \right) }}&{
	\frac {{a}^{2}{h}^{4}-{a}^{2}{h}^{2}+{a}^{2}-{h}^{2}}{ \left( a+1
		\right)  \left( a-1 \right) }}\end {array} \right] 
$$
Finally, by definition $\T=\mathsf{P}^{-1} \mathsf{P}^{\ast}$ and from explicit matrices in Appendix \ref{macappen}, in variables (\ref{tahvar}) we compute:
$$
\T=\left[ \begin {array}{cc} {\frac { \left( ah+1 \right)  \left( ah-1
		\right) }{ah \left( a+1 \right)  \left( a-1 \right) }}&{\frac {a
		\left( h-1 \right)  \left( h+1 \right) }{ \left( a-1 \right)  \left( 
		a+1 \right) h}}\\ \noalign{\medskip}-{\frac {a \left( h-1 \right) 
		\left( h+1 \right) }{ \left( a-1 \right)  \left( a+1 \right) h}}&{
	\frac {a \left( a+h \right)  \left( -h+a \right) }{ \left( a-1
		\right)  \left( a+1 \right) h}}\end {array} \right],
$$
and we check that relation (\ref{reln2}) holds. 

As another non-trivial consistency check we can compute the monodropmy operators of $\nabla_{\GW}(\gamma')$ for loops based at $z=1$. Using Proposition \ref{gwdtmonod}  we compute:
$$
\begin{array}{l}
\nabla_{\GW}(\gamma'_{0}) = \mathsf{P} \, \textsf{E}_{0} \, \mathsf{P}^{-1}= \left[ \begin {array}{cc} {\frac {{a}^{2}h+a{h}^{2}-a+h}{2 a{h}^{2
}}}&-{\frac { \left( ah+1 \right)  \left( a+h \right) }{2 a{h}^{2}}
}\\ \noalign{\medskip}-{\frac { \left( ah-1 \right)  \left( -h+a
		\right) }{2 a{h}^{2}}}&{\frac {{a}^{2}h-a{h}^{2}+a+h}{2 a{h}^{2}}}
\end {array} \right],
\\
\\
\nabla_{\GW}(\gamma'_{1/2}) = \mathsf{P} \, {\B}^{\bullet}_{1/2} \, \mathsf{P}^{-1}=  \left[ \begin {array}{cc} -{\frac {{a}^{2}{h}^{3}-a{h}^{4}-{a}^{
			2}h+{h}^{3}-a-h}{2 a}}&-{\frac { \left( a+h \right)  \left( h-1
		\right)  \left( h+1 \right)  \left( ah+1 \right) }{2 a}}
\\ \noalign{\medskip}{\frac { \left( h+1 \right)  \left( h-1
		\right)  \left( ah-1 \right)  \left( -h+a \right) }{2 a}}&{\frac {
		{a}^{2}{h}^{3}+a{h}^{4}-{a}^{2}h+{h}^{3}+a-h}{2 a}}\end {array} \right], 
\\
\\
\nabla_{\GW}(\gamma'_{\infty}) =\mathsf{P} \, \T^{-1} \textsf{E}_{\infty} \T\, \mathsf{P}^{-1} = \left[ \begin {array}{cc} {\frac {{a}^{2}h-a{h}^{2}+a+h}{2 a}}&-{\frac { \left( ah+1 \right)  \left( a+h \right) }{2 a}}
\\ \noalign{\medskip}-{\frac { \left( ah-1 \right)  \left( -h+a
		\right) }{2 a}}&{\frac {{a}^{2}h+a{h}^{2}-a+h}{2 a}}\end {array}
\right].
\end{array}
$$
We observe that, in full agreement with Theorem \ref{thmbased1}, all matrix of $\nabla_{\GW}(\gamma')$ are Laurent polynomials in $a$ and $h$.

Similarly, in the case $n=3$, the fundamental group $\pi_{1}(\mathbb{P}^{1}\setminus\sing,0_{+})$ is generated by 
$\gamma_{0},\gamma_{1/3},\gamma_{1/2},\gamma_{2/3},\gamma_{\infty}$ subject to the relation
$$
\gamma_{0}\, \gamma_{1/3}\,\gamma_{1/2}\,\gamma_{2/3}=\gamma_{\infty}
$$
By Theorem \ref{nomodtheor} we need to check the relation
$$
{\B}^{\bullet}_{1/3} {\B}^{\bullet}_{1/2} {\B}^{\bullet}_{2/3} \textsf{E}_{0} = \T^{-1} \textsf{E}_{\infty} \T. 
$$
This relation, and the polynomiality of $\nabla_{\GW}(\gamma')$ are checked using explicit matrices below:

\begin{landscape}
	\thispagestyle{empty}
	\begin{scriptsize}
		$$ \hspace{-10mm}
		\begin{array}{l}
		{\B}^{\bullet}_{1/3}=1+ \left[ \begin {array}{ccc} {\frac { \left( a+h \right)  \left( {a}^{2
				}+ah+{h}^{2} \right)  \left( h+1 \right)  \left( h-1 \right)  \left( -
				h+a \right) ^{2}}{ \left( a+1 \right)  \left( a-1 \right)  \left( {a}^
				{3}-h \right) }}&-{\frac { \left( {a}^{2}h+a+h \right) h \left( h+1
				\right)  \left( h-1 \right)  \left( ah-1 \right)  \left( -h+a
				\right) }{ \left( a+1 \right)  \left( a-1 \right)  \left( {a}^{3}-h
				\right) }}&{\frac { \left( ah+1 \right)  \left( {a}^{2}{h}^{2}+ah+1
				\right)  \left( h+1 \right)  \left( h-1 \right)  \left( ah-1 \right) 
				^{2}}{a \left( {a}^{3}-h \right)  \left( a-1 \right)  \left( a+1
				\right) }}\\ \noalign{\medskip}-{\frac { \left( {a}^{2}+ah+1 \right) 
				\left( a+h \right)  \left( {a}^{2}+ah+{h}^{2} \right)  \left( h+1
				\right)  \left( h-1 \right)  \left( -h+a \right) ^{2}}{ \left( {a}^{3
				}-h \right)  \left( {a}^{3}h-1 \right) }}&{\frac {h \left( h-1
				\right)  \left( h+1 \right)  \left( -h+a \right)  \left( {a}^{2}h+a+h
				\right)  \left( {a}^{2}+ah+1 \right)  \left( ah-1 \right) }{ \left( {
					a}^{3}-h \right)  \left( {a}^{3}h-1 \right) }}&-{\frac { \left( {a}^{2
				}+ah+1 \right)  \left( ah+1 \right)  \left( {a}^{2}{h}^{2}+ah+1
				\right)  \left( h+1 \right)  \left( h-1 \right)  \left( ah-1 \right) 
				^{2}}{a \left( {a}^{3}-h \right)  \left( {a}^{3}h-1 \right) }}
		\\ \noalign{\medskip}{\frac { \left( {a}^{2}+ah+{h}^{2} \right) 
				\left( a+h \right) a \left( h+1 \right)  \left( h-1 \right)  \left( -
				h+a \right) ^{2}}{ \left( a+1 \right)  \left( a-1 \right)  \left( {a}^
				{3}h-1 \right) }}&-{\frac { \left( {a}^{2}h+a+h \right) ha \left( h+1
				\right)  \left( h-1 \right)  \left( ah-1 \right)  \left( -h+a
				\right) }{ \left( a+1 \right)  \left( a-1 \right)  \left( {a}^{3}h-1
				\right) }}&{\frac { \left( ah+1 \right)  \left( {a}^{2}{h}^{2}+ah+1
				\right)  \left( h+1 \right)  \left( h-1 \right)  \left( ah-1 \right) 
				^{2}}{ \left( a+1 \right)  \left( a-1 \right)  \left( {a}^{3}h-1
				\right) }}\end {array} \right] 
		\\
		\\
		{\B}^{\bullet}_{1/3}= 1+\left[ \begin {array}{ccc} {\frac { \left( h-1 \right)  \left( h+1
				\right)  \left( -h+a \right)  \left( a+h \right)  \left( ah-1
				\right)  \left( {a}^{2}+ah+{h}^{2} \right) }{h \left( a+1 \right) 
				\left( a-1 \right)  \left( {a}^{3}-h \right) }}&-{\frac {a \left( {a}
				^{2}+1 \right)  \left( h+1 \right) ^{2} \left( h-1 \right) ^{2}
				\left( ah-1 \right) }{ \left( a+1 \right)  \left( a-1 \right) 
				\left( {a}^{3}-h \right) }}&-{\frac { \left( ah+1 \right)  \left( {a}
				^{2}{h}^{2}+ah+1 \right)  \left( h+1 \right)  \left( h-1 \right) 
				\left( ah-1 \right) ^{2}}{h \left( a+1 \right)  \left( a-1 \right) 
				\left( {a}^{3}-h \right) }}\\ \noalign{\medskip}-{\frac { \left( {a}^
				{2}+1 \right)  \left( a+h \right)  \left( {a}^{2}+ah+{h}^{2} \right) 
				\left( h+1 \right) ^{2} \left( h-1 \right) ^{2} \left( -h+a \right) 
			}{h \left( {a}^{3}-h \right)  \left( {a}^{3}h-1 \right) }}&{\frac {a
				\left( h-1 \right) ^{3} \left( h+1 \right) ^{3} \left( {a}^{2}+1
				\right) ^{2}}{ \left( {a}^{3}-h \right)  \left( {a}^{3}h-1 \right) }}
		&{\frac { \left( {a}^{2}+1 \right)  \left( ah+1 \right)  \left( {a}^{2
				}{h}^{2}+ah+1 \right)  \left( h+1 \right) ^{2} \left( h-1 \right) ^{2}
				\left( ah-1 \right) }{h \left( {a}^{3}-h \right)  \left( {a}^{3}h-1
				\right) }}\\ \noalign{\medskip}-{\frac { \left( a+h \right)  \left( {
					a}^{2}+ah+{h}^{2} \right)  \left( h+1 \right)  \left( h-1 \right) 
				\left( -h+a \right) ^{2}}{h \left( a+1 \right)  \left( a-1 \right) 
				\left( {a}^{3}h-1 \right) }}&{\frac {a \left( {a}^{2}+1 \right) 
				\left( h+1 \right) ^{2} \left( h-1 \right) ^{2} \left( -h+a \right) 
			}{ \left( a+1 \right)  \left( a-1 \right)  \left( {a}^{3}h-1 \right) }
		}&{\frac { \left( h-1 \right)  \left( h+1 \right)  \left( ah-1
				\right)  \left( -h+a \right)  \left( ah+1 \right)  \left( {a}^{2}{h}^
				{2}+ah+1 \right) }{h \left( a+1 \right)  \left( a-1 \right)  \left( {a
				}^{3}h-1 \right) }}\end {array} \right] 
		\\
		\\
		{\B}^{\bullet}_{2/3}=1+ \left[ \begin {array}{ccc} {\frac { \left( a+h \right)  \left( {a}^{2
				}+ah+{h}^{2} \right)  \left( h+1 \right)  \left( h-1 \right)  \left( -
				h+a \right) ^{2}}{ \left( a+1 \right)  \left( a-1 \right)  \left( {a}^
				{3}-h \right) }}&-{\frac { \left( {a}^{2}+ah+1 \right) ha \left( h+1
				\right)  \left( h-1 \right)  \left( ah-1 \right)  \left( -h+a
				\right) }{ \left( a+1 \right)  \left( a-1 \right)  \left( {a}^{3}-h
				\right) }}&{\frac { \left( {a}^{2}{h}^{2}+ah+1 \right)  \left( ah+1
				\right)  \left( h+1 \right)  \left( h-1 \right) a \left( ah-1
				\right) ^{2}}{ \left( a+1 \right)  \left( a-1 \right)  \left( {a}^{3}
				-h \right) }}\\ \noalign{\medskip}-{\frac { \left( {a}^{2}h+a+h
				\right)  \left( a+h \right)  \left( {a}^{2}+ah+{h}^{2} \right) 
				\left( h+1 \right)  \left( h-1 \right)  \left( -h+a \right) ^{2}}{a
				\left( {a}^{3}-h \right)  \left( {a}^{3}h-1 \right) }}&{\frac {h
				\left( h-1 \right)  \left( h+1 \right)  \left( -h+a \right)  \left( {
					a}^{2}h+a+h \right)  \left( {a}^{2}+ah+1 \right)  \left( ah-1 \right) 
			}{ \left( {a}^{3}-h \right)  \left( {a}^{3}h-1 \right) }}&-{\frac {
				\left( {a}^{2}h+a+h \right)  \left( ah+1 \right)  \left( {a}^{2}{h}^{
					2}+ah+1 \right)  \left( h+1 \right)  \left( h-1 \right)  \left( ah-1
				\right) ^{2}}{ \left( {a}^{3}-h \right)  \left( {a}^{3}h-1 \right) }}
		\\ \noalign{\medskip}{\frac { \left( a+h \right)  \left( {a}^{2}+ah+{h
				}^{2} \right)  \left( h+1 \right)  \left( h-1 \right)  \left( -h+a
				\right) ^{2}}{a \left( a+1 \right)  \left( a-1 \right)  \left( {a}^{3
				}h-1 \right) }}&-{\frac { \left( {a}^{2}+ah+1 \right) h \left( h+1
				\right)  \left( h-1 \right)  \left( ah-1 \right)  \left( -h+a
				\right) }{ \left( a+1 \right)  \left( a-1 \right)  \left( {a}^{3}h-1
				\right) }}&{\frac { \left( ah+1 \right)  \left( {a}^{2}{h}^{2}+ah+1
				\right)  \left( h+1 \right)  \left( h-1 \right)  \left( ah-1 \right) 
				^{2}}{ \left( a+1 \right)  \left( a-1 \right)  \left( {a}^{3}h-1
				\right) }}\end {array} \right] 
		\\
		\\
		\T= \left[ \begin {array}{ccc} {\frac { \left( ah-1 \right) ^{2} \left( {
					a}^{2}{h}^{2}+ah+1 \right)  \left( ah+1 \right) }{{h}^{2}{a}^{3}
				\left( a+1 \right)  \left( a-1 \right)  \left( {a}^{3}-h \right) }}&{
			\frac { \left( ah-1 \right)  \left( h-1 \right)  \left( h+1 \right) 
				\left( ah+1 \right) }{h \left( a+1 \right)  \left( a-1 \right) 
				\left( {a}^{3}-h \right) }}&-{\frac {{a}^{2} \left( h-1 \right) 
				\left( h+1 \right)  \left( -{h}^{3}+a \right) }{{h}^{2} \left( a+1
				\right)  \left( a-1 \right)  \left( {a}^{3}-h \right) }}
		\\ \noalign{\medskip}-{\frac { \left( h-1 \right)  \left( h+1 \right) 
				\left( {a}^{2}{h}^{2}+ah+1 \right)  \left( ah+1 \right)  \left( ah-1
				\right)  \left( a+h \right) }{a{h}^{3} \left( {a}^{3}-h \right) 
				\left( {a}^{3}h-1 \right) }}&{\frac {{a}^{6}{h}^{3}-{a}^{4}{h}^{5}-{a
				}^{3}{h}^{6}+2\,{a}^{4}{h}^{3}-{a}^{2}{h}^{5}-{a}^{4}h+2\,{a}^{2}{h}^{
					3}-{a}^{3}-{a}^{2}h+{h}^{3}}{{h}^{2} \left( {a}^{3}h-1 \right) 
				\left( {a}^{3}-h \right) }}&{\frac { \left( -h+a \right)  \left( a+h
				\right)  \left( {a}^{2}+ah+{h}^{2} \right)  \left( ah+1 \right) 
				\left( h-1 \right)  \left( h+1 \right) {a}^{2}}{{h}^{3} \left( {a}^{3
				}-h \right)  \left( {a}^{3}h-1 \right) }}\\ \noalign{\medskip}{\frac {
				\left( a{h}^{3}-1 \right)  \left( h-1 \right)  \left( h+1 \right) {a}
				^{2}}{{h}^{2} \left( {a}^{3}h-1 \right)  \left( a+1 \right)  \left( a-
				1 \right) }}&-{\frac {{a}^{3} \left( -h+a \right)  \left( h-1 \right) 
				\left( h+1 \right)  \left( a+h \right) }{h \left( a+1 \right) 
				\left( a-1 \right)  \left( {a}^{3}h-1 \right) }}&{\frac { \left( {a}^
				{2}+ah+{h}^{2} \right)  \left( a+h \right)  \left( -h+a \right) ^{2}{a
				}^{3}}{{h}^{2} \left( {a}^{3}h-1 \right)  \left( a+1 \right)  \left( a
				-1 \right) }}\end {array} \right] 
		\\
		\\
		\mathsf{E}_{0}= \left[ \begin {array}{ccc} {\frac {{a}^{3}}{{h}^{3}}}&0&0
		\\ \noalign{\medskip}0&{h}^{-2}&0\\ \noalign{\medskip}0&0&{\frac {1}{{
					a}^{3}{h}^{3}}}\end {array} \right], \ \ \ \ \ \mathsf{E}_{\infty}=\left[ \begin {array}{ccc} {a}^{3}{h}^{3}&0&0\\ \noalign{\medskip}0&{
			h}^{2}&0\\ \noalign{\medskip}0&0&{\frac {{h}^{3}}{{a}^{3}}}
		\end {array} \right] 
		\end{array}
		$$
	\end{scriptsize}
\end{landscape}

\bibliographystyle{abbrv}
\bibliography{bib}

\end{document}